\documentclass[11pt,leqno]{amsart}
\usepackage{hyperref}
\usepackage{amsmath,amsthm,amscd,amssymb}\usepackage{amsmath}
\usepackage{amsfonts}
\usepackage[utf8]{inputenc}
\usepackage{enumerate}

\newcommand{\arxiv}[1]{\href{http://arxiv.org/#1}{arXiv:#1}}
\newcommand*{\mailto}[1]{\href{mailto:#1}{\nolinkurl{#1}}}

\newtheorem{theorem}{Theorem}[section]
\newtheorem{lemma}[theorem]{Lemma}

\numberwithin{equation}{section}

\newtheorem{re}{Remark}[section]
\newtheorem{prop}{Proposition}[section]
\newtheorem{theo}{Theorem}[section]
\newtheorem{lem}{Lemma}[section]
\newtheorem{col}{Corollary}[section]

\newcommand{\be}{\begin{equation}}
\newcommand{\ee}{\end{equation}}
\newcommand\bes{\begin{eqnarray}} \newcommand\ees{\end{eqnarray}}
\newcommand{\bess}{\begin{eqnarray*}}
\newcommand{\eess}{\end{eqnarray*}}
\newcommand{\D}{\displaystyle}

\newcommand{\bu}{{\bf u}}

\newcommand{\bv}{{\bf v}}
\newcommand{\bw}{{\bf w}}
\newcommand{\equ}{\overset{\rm def}{=}}

\def\XXint#1#2#3{{\setbox0=\hbox{$#1{#2#3}{\int}$}
     \vcenter{\hbox{$#2#3$}}\kern-.5\wd0}}

\numberwithin{equation}{section}

\begin{document}

\title[Optimal decay for full compressible N-S equations]{Optimal decay for the full compressible Navier-Stokes system in critical $L^p$ Besov spaces}

\author[Q. Bie]{Qunyi Bie}
\address[Q. Bie]{College of Science $\&$  Three Gorges Mathematical Research Center, China Three Gorges University, Yichang 443002, PR China}
\email{\mailto{qybie@126.com}}

\author[Q. Wang]{Qiru Wang}
\address[Q. Wang]{School of Mathematics, Sun Yat-Sen University, Guangzhou 510275, PR China}
\email{\mailto{mcswqr@mail.sysu.edu.cn}}
\author[Z.-A. Yao]{Zheng-an Yao}
\address[Z.-A. Yao]{School of Mathematics, Sun Yat-Sen University, Guangzhou 510275, PR China}
\email{\mailto{mcsyao@mail.sysu.edu.cn}}


\begin{abstract} Danchin and He (Math. Ann. 64: 1-38, 2016) recently established  the global existence  in critical $L^p$ type regularity framework for the $N$-dimensional $(N\geq 3)$ non-isentropic compressible Navier-Stokes equations. The purpose of this paper is to further investigate the large time behavior of solutions constructed by them. More precisely, we prove that if the initial data at the low frequencies additionally belong to some Besov space $\dot{B}_{2,\infty}^{-\sigma_1}$ with $\sigma_1\in (2-N/2, 2N/p-N/2]$, then the $\dot{B}_{p,1}^s$ norm of the critical global solutions exhibits the optimal decay  $(1+t)^{-\frac{N}{2}(\frac{1}{2}-\frac{1}{p})-\frac{s+\sigma_1}{2}}$ for suitable $p$ and $s$. The main tool we use is the pure energy argument without the spectral analysis, which enables us
to \emph{remove the smallness assumption} of initial data at the low-frequency.

\end{abstract}

\maketitle

\section{Introduction} \label{s:1}
The  full compressible Navier-Stokes system  for $(t,x)\in \mathbb{R}_+\times\mathbb{R}^N$ reads as
\begin{equation}\label{1.1}
 \left\{\begin{array}{ll}\displaystyle\partial_t \rho+{\rm div}(\rho{\bf u})
=0,\\[1ex]
 \displaystyle\partial_t(\rho{\bf u})+{\rm div}(\rho{\bf u}\otimes{\bf u})+\nabla P={\rm div}{\bf\tau},\\[1ex]
 \displaystyle\partial_t\Big[\rho\Big(\frac{|\bu|^2}{2}+e\Big)\Big]+{\rm div}\left[\bu\Big(\rho\Big(\frac{|\bu|^2}{2}+e\Big)+P\Big)\right]={\rm div}({\bf\tau}\cdot\bu+\kappa\nabla\mathcal{T}),
 \end{array}
 \right.
 \end{equation}
where $\rho=\rho(t,x)\in\mathbb{R}_+, \bu=\bu(t, x)\in\mathbb{R}^N, P=P(t,x), e=e(t,x)\in\mathbb{R}_+$ and $\mathcal{T}=\mathcal{T}(t,x)$ denote, respectively,  the density, velocity, pressure, the internal energy per unit mass and the absolute temperature. The parameter $\kappa>0$ acting as a thermal conduction coefficient is assumed to be constant. The internal stress tensor ${\bf \tau}$ is given by
$$
\tau\equ\mu(\nabla\bu+(\nabla\bu)^T)+\lambda({\rm div}\bu){\bf I},
$$
in which ${\bf I}$ is the identity matrix, $(\nabla\bu)^T$ is the transpose matrix of $\nabla\bu$,  and the constants $\mu$ and $\lambda$ are  viscosity coefficients fulfilling the usual condition
$$\mu>0\,\,\,\,{\rm and}\,\,\,\, \nu\equ \lambda+2\mu>0.$$

From the second and third equations of \eqref{1.1}, it is easy to see  that
$$
\partial_t(\rho e)+{\rm div}(\rho\bu e)+P{\rm div}\bu-\kappa\Delta\mathcal{T}=\frac{\mu}{2}
 |\nabla\bu+(\nabla\bu)^T|^2+\lambda({\rm div}\bu)^2.
$$
In order to reformulate system \eqref{1.1} in terms of $\rho, \bu$ and $\mathcal{T}$ only, we suppose that
$e=e(\rho,\mathcal{T})$ satisfies Joule law
\be\label{1.9}
\partial_{\mathcal{T}}e=C_v\,\,\,\,{\rm with\,\,some\,\,positive\,\,constant}\,\,C_v
\ee
and that the pressure function $P=P(\rho, \mathcal{T})$ fulfills the following pressure laws
$$
P(\rho,\mathcal{T})=\pi_0(\rho)+\mathcal{T}\pi_1(\rho),
$$
where $\pi_0$ and $\pi_1$ are given smooth functions (for example, perfect gases when $\pi_0(\rho)=0$ and $\pi_1(\rho)=R\rho$ with $R>0$ or Van-der-Waals fluids when $\pi_0(\rho)=-\alpha\rho^2, \pi_1(\rho)=\beta\rho/(\delta-\rho)$ with $\alpha,\beta,\delta>0$). By using the Gibbs relations for the internal energy and the Helmholtz free energy, we have the following  Maxwell relation
$$
\rho^2\partial_{\rho}e(\rho, \mathcal{T})=P(\rho, \mathcal{T})-\mathcal{T}\partial_{\mathcal{T}}P(\rho, \mathcal{T})=\pi_0(\rho),
$$
and then system \eqref{1.1} may be rewritten as
\begin{equation}\label{1.2}
 \left\{\begin{array}{ll}\displaystyle\partial_t \rho+{\rm div}(\rho{\bf u})
=0,\\[1ex]
 \displaystyle\rho(\bu_t+\bu\cdot\nabla\bu)-\mu\Delta\bu-(\lambda+\mu)\nabla{\rm div}\bu+\nabla P=0,\\[1ex]
 \displaystyle\rho C_v(\partial_t\mathcal{T}+\bu\cdot\nabla\mathcal{T})+\mathcal{T}\pi_1(\rho){\rm div}\bu-\kappa\Delta\mathcal{T}=\frac{\mu}{2}
 |\nabla\bu+(\nabla\bu)^T|^2+\lambda({\rm div}\bu)^2.
 \end{array}
 \right.
 \end{equation}

We are concerned with the large time behavior of global solutions to the Cauchy problem of system \eqref{1.2} subject to
the initial condition
\begin{equation}\label{1.3}
  (\rho, \bu, \mathcal{T})|_{t=0}=(\rho_0(x), \bu_0(x), \mathcal{T}_0(x)), \,\,x\in \mathbb{R}^N,
\end{equation}
and focus on solutions that are close to some constant state $(\rho^\ast, {\bf 0}, \mathcal{T}^\ast)$ with $\rho^\ast>0$ and   $\mathcal{T}^\ast>0$, at infinity,  which fulfills the following linear stability condition:
\begin{equation}\label{1.4}
  \partial_\rho P(\rho^\ast, \mathcal{T}^\ast)>0\,\,\,\,{\rm and}\,\,\,\,\,\partial_\mathcal{T} P(\rho^\ast, \mathcal{T}^\ast)>0.
\end{equation}

Since the pioneering work by Matsumura  and Nishida \cite{matsumura1980initial}, many papers have been dedicated to system \eqref{1.1} in the case of solutions with high Sobolev regularity. In this paper, we focus on the study of  long time asymptotic behavior in the so-called \emph{critical regularity} framework. Here we observe  that system \eqref{1.1} is invariant by the transformation
\be\nonumber
\tilde{\rho}(t,x)=\rho(l^2t, lx),~~\tilde{\bf u}(t, x)=l{\bf u}(l^2t, lx),~~\tilde{\mathcal{T}}(t, x)=l^2{\mathcal{T}}(l^2t, lx),
\ee
 up to a change of the pressure law $\tilde{P}=l^2P$. A critical space is a space in which the norm is invariant under the scaling
$
(\tilde{e},\tilde{\bf f}, \tilde{g})(x)=(e(lx), l{\bf f}(lx), l^2{g}(lx)).
$

In the critical framework, there have been many results for the compressible (or incompressible) Navier-Stokes equations, see for example \cite{cannone1997a,  charve2010global,  chen2010global, danchin2000global, danchin2016fourier, danchin2016incompressible, danchin2016optimal, fujita1964on,  haspot2011existence, kozono1994semilinear,  okita2014optimal,  xin2018optimal, xu2019a}. In particular, concerning the large time asymptotic behavior of strong solutions for the compressible Navier-Stokes equations, Okita \cite{okita2014optimal} exploited low and high frequency decompositions to get the time decay rate for strong solutions in the $L^2$  critical framework and in dimension $N\geq 3$. In the survey paper \cite{danchin2016fourier},
 Danchin proposed another description of the time decay which allows to proceed with dimension $N\geq 2$ in the $L^2$ critical framework. Recently,  Danchin and Xu \cite{danchin2016optimal} extended the work of \cite{danchin2016fourier} and got the optimal time decay rate in the
 general $L^p$ type critical spaces and in any dimension $N\geq 2$. Later on,  Xu \cite{xu2019a} developed a general low-frequency condition for optimal decay estimates, where the regularity $\sigma_1$ of $\dot{B}_{2,\infty}^{-\sigma_1}$
 belongs to a whole range $(1-\frac{N}{2}, \frac{2N}{p}-\frac{N}{2}]$,  and the proof  depends mainly on the refined time-weighted energy approach in the Fourier semi-group framework. Very recently, originated from the ideas as in \cite{guo2012decay, strain2006almost}, Xin and Xu \cite{xin2018optimal} developed a new energy  argument to remove the usual smallness assumption of low frequencies.

 Let us also recall some important progress concerning non-isentropic compressible Navier-Stokes system \eqref{1.1} in the critical framework. Danchin \cite{danchin2001local}
and Chikami and Danchin \cite{chikami2015on} proved the local well-posedness to system \eqref{1.1} in the critical spaces $\dot{B}_{p,1}^{\frac{N}{p}}\times(\dot{B}_{p,1}^{\frac{N}{p}-1})^N\times\dot{B}_{p,1}^{\frac{N}{p}-2}$ if $N\geq 3$ and $1<p<N$.
As regards the global well-posedness issue, Danchin \cite{danchin2001global} obtained the global existence and uniqueness of solutions to system \eqref{1.1} in the $L^2$-type critical Besov spaces.  Recently, Danchin and He \cite{danchin2016incompressible} generalized the results of \cite{danchin2001global} to $L^p$-type critical Besov spaces, and derived the following global existence results:
\begin{theo}\label{th1}
  {\rm(}\cite{danchin2016incompressible}{\rm)} Let $\rho^\ast$ and $\mathcal{T}^\ast$ be two positive constants such that \eqref{1.4} is satisfied. Assume that $N\geq 3$, and that $p$ fulfills
  \begin{equation}\label{1.5}
    2\leq p<N\,\,\,\,{\rm and}\,\,\,p\leq 2N/(N-2).
  \end{equation}
There exist a constant $c=c(p,N,\lambda,\mu,P,\kappa,C_v,\rho^\ast,\mathcal{T}^\ast)$ and a universal integer $j_0\in \mathbb{Z}$ such that if $a_0\equ\rho_0-\rho^\ast$ is in $\dot{B}_{p,1}^{\frac{N}{p}}$, if $\bv_0\equ\bu_0$ is in $\dot{B}_{p,1}^{\frac{N}{p}-1}$, if $\theta_0\equ\mathcal{T}_0-\mathcal{T}^\ast$ is in $\dot{B}_{p,1}^{\frac{N}{p}-2}$ and if in addition $(a_0^\ell, \bv_0^\ell, \theta_0^\ell)\in \dot{B}_{2,1}^{\frac{N}{2}-1}$
{\rm(}with the notation $z^\ell\equ \dot{S}_{k_0+1}z$ and $z^h=z-z^\ell${\rm )}
with
$$
\mathcal{X}_{p,0}\equ \|(a_0, \bv_0, \theta_0)\|_{\dot{B}_{2,1}^{\frac{N}{2}-1}}^\ell
+\|a_0\|_{\dot{B}_{p,1}^{\frac{N}{p}}}^h+\|\bv_0\|_{\dot{B}_{p,1}^{\frac{N}{p}-1}}^h
+\|\theta_0\|_{\dot{B}_{p,1}^{\frac{N}{p}-2}}^h\leq c,
$$
then the Cauchy problem \eqref{1.2}-\eqref{1.3} admits a unique global-in-time solution $(\rho, \bu, \mathcal{T})$
with $\rho=\rho^\ast+a, \bu=\bv$ and $\mathcal{T}=\mathcal{T}^\ast+\theta$, where $(a, \bv, \theta)$ is in the space $X_p$ defined by
\begin{equation}\nonumber
\left.\begin{array}{ll}
(a,\bv,\theta)^\ell\in \widetilde{\mathcal{C}_b}(\mathbb{R}_+; \dot{B}_{2,1}^{\frac{N}{2}-1})\cap
L^1(\mathbb{R}_+; \dot{B}_{2,1}^{\frac{N}{2}+1}), \,\,\,a^h\in \widetilde{\mathcal{C}_b}(\mathbb{R}_+; \dot{B}_{p,1}^{\frac{N}{p}})\cap
L^1(\mathbb{R}_+; \dot{B}_{p,1}^{\frac{N}{p}}),\\[1ex]
\bv^h\in \widetilde{\mathcal{C}_b}(\mathbb{R}_+; \dot{B}_{p,1}^{\frac{N}{p}-1})\cap
L^1(\mathbb{R}_+; \dot{B}_{p,1}^{\frac{N}{p}+1}),\,\,\,\,\theta^h\in \widetilde{\mathcal{C}_b}(\mathbb{R}_+; \dot{B}_{p,1}^{\frac{N}{p}-2})\cap
L^1(\mathbb{R}_+; \dot{B}_{p,1}^{\frac{N}{p}}).
\end{array}
\right.
\end{equation}
Furthermore, we get for some constant $C=C(p,N,\lambda, \mu, \kappa, C_v, P, \rho^\ast, \mathcal{T}^\ast)$,
$$
\mathcal{X}_p(t)\leq C\mathcal{X}_{p,0},
$$
for any $t>0$, where
\begin{equation}\label{1.6}
\begin{split}
\mathcal{X}_p(t)&\equ \|(a, \bv, \theta)\|_{\widetilde{L}^\infty(\dot{B}_{2,1}^{\frac{N}{2}-1})}^\ell+
\|(a, \bv, \theta)\|_{{L}^1(\dot{B}_{2,1}^{\frac{N}{2}+1})}^\ell
+\|a\|_{\widetilde{L}^\infty(\dot{B}_{p,1}^{\frac{N}{p}})}^h
+\|a\|_{{L}^1(\dot{B}_{p,1}^{\frac{N}{p}})}^h\\[1ex]
&\quad+\|\bv\|_{\widetilde{L}^\infty(\dot{B}_{p,1}^{\frac{N}{p}-1})}^h
+\|\bv\|_{{L}^1(\dot{B}_{p,1}^{\frac{N}{p}+1})}^h+\|\theta\|_{\widetilde{L}^\infty(\dot{B}_{p,1}^{\frac{N}{p}-2})}^h
+\|\theta\|_{{L}^1(\dot{B}_{p,1}^{\frac{N}{p}})}^h.
\end{split}
\end{equation}
\end{theo}
The natural next problem is to investigate the large time asymptotic behavior of global solutions constructed above. In this respect, we could refer to the recent works \cite{danchin2018optimal, shi2019thelarge, zhai2019longtime}. Therein, Danchin and Xu \cite{danchin2018optimal} applied  Fourier analysis techniques to give precise description for the large time asymptotic behavior of solutions with the additional condition concerning the low frequencies of initial data. Shi and Xu \cite{shi2019thelarge}  further enlarged the range of the low regularity index assumption in \cite{danchin2018optimal}. Let us remark that in \cite{danchin2018optimal, shi2019thelarge} the additional condition  $\|(a_0, \bv_0, \theta_0)\|_{\dot{B}_{2,\infty}^{-\sigma_1}}^\ell$ is required to be small. At this point,  in the case of $N=3$, Zhai and Chen \cite{zhai2019longtime} applied some energy arguments developed by Xin and Xu \cite{xin2018optimal} to obtain the optimal decay rates without the smallness assumption. There, they assumed the additional condition $\|(a_0, \bv_0, \theta_0)\|_{\dot{B}_{2,1}^{-\sigma_1}}^\ell$ is bounded.  In this paper, motivated by the works \cite{danchin2018optimal,  guo2012decay, shi2019thelarge, strain2006almost, xin2018optimal,zhai2019longtime}, in the case of $N\geq 3$,  we intend to remove the smallness assumption for  $\|(a_0, \bv_0, \theta_0)\|_{\dot{B}_{2,\infty}^{-\sigma_1}}^\ell$ in \cite{danchin2018optimal, shi2019thelarge}. That is, on the condition that $\|(a_0, \bv_0, \theta_0)\|_{\dot{B}_{2,\infty}^{-\sigma_1}}^\ell$ is bounded, we are going to establish the optimal decay of solutions for system \eqref{1.1} in the $L^p$ critical Besov spaces.

\section{Main results}\label{s:2}
\setcounter{equation}{0}\setcounter{section}{2}\indent
As in \cite{danchin2018optimal},  taking $\mathcal{A}\equ \mu\Delta+(\lambda+\mu)\nabla{\rm div}, \rho=\rho^\ast(1+b)>0$ and $\mathcal{T}=\mathcal{T}^\ast+\mathfrak{T}$, we derive from system \eqref{1.2} that the triplet $(b,\bu,\mathfrak{T})$ satisfies
{\small \begin{equation}\label{2.1}
 \left\{\begin{array}{ll}\displaystyle\partial_t b+\bu\cdot\nabla b+(1+b){\rm div}\bu
=0,\\[2ex]
 \displaystyle\partial_t\bu+\bu\cdot\nabla\bu-\frac{\mathcal{A\bu}}{\rho^\ast(1+b)}
 +\frac{\partial_{\rho}P(\rho^\ast(1+b),\mathcal{T}^\ast)}{1+b}\nabla b-\frac{\pi_1(\rho^\ast(1+b))}{\rho^\ast(1+b)}\nabla\mathfrak{T}+\frac{\pi_1^\prime(\rho^\ast(1+b))}{1+b}\mathfrak{T}\nabla b=0,\\[3ex]
 \displaystyle\partial_t\mathfrak{T}+\bu\cdot\nabla\mathfrak{T}+(\mathcal{T}^\ast+\mathfrak{T})\frac{\pi_1(\rho^\ast(1+b))}{\rho^\ast C_v(1+b)}{\rm div}\bu-\frac{\kappa}{\rho^\ast C_v(1+b)}\Delta\mathfrak{T}=\frac{\frac{\mu}{2}
 |\nabla\bu+(\nabla\bu)^T|^2+\lambda({\rm div}\bu)^2}{\rho^\ast C_v(1+b)}.
 \end{array}
 \right.
 \end{equation}}
Then, setting $\nu\equ\lambda+2\mu, \bar{\nu}=\nu/\rho^\ast, \chi_0\equ\partial_{\rho}P(\rho^\ast, \mathcal{T}^\ast)^{-\frac{1}{2}}$ and executing the transformation
{\small $$
a(t,x)=b(\bar{\nu}\chi_0^2t,\bar{\nu}\chi_0 x),\,\,\,\,{\bf v}(t,x)=\chi_0\bu(\bar{\nu}\chi_0^2t,\bar{\nu}\chi_0 x),\,\,\,
\theta(t,x)=\chi_0\left(\frac{C_v}{\mathcal{T}^\ast}\right)^{\frac{1}{2}}\mathfrak{T}(\bar{\nu}\chi_0^2t,\bar{\nu}\chi_0 x),
$$}
one has
\begin{equation}\label{2.2}
\left\{\begin{array}{ll}
\partial_ta+{\rm div}\bv=f,\\[1ex]
\partial_t\bv-\widetilde{\mathcal{A}}\bv+\nabla a+\gamma\nabla\theta={\bf g},\\[1ex]
\partial_t\theta-\beta\Delta \theta+\gamma{\rm div}\bv={m},\\[1ex]
(a, {\bf v}, \theta)|_{t=0}=(a_0, {\bf v}_0, \theta_0),
\end{array}
\right.
\end{equation}
with
$$
\widetilde{\mathcal{A}}\equ \frac{\mathcal{A}}{\nu}, \,\,\,\,\,\beta\equ\frac{\kappa}{\nu C_v},\,\,\,\,\,\gamma\equ \frac{\chi_0}{\rho^\ast}\left(\frac{\mathcal{T}^\ast}{C_v}\right)^{\frac{1}{2}}\pi_1(\rho^\ast),
$$
where the nonlinear terms $f, {\bf g}$ and $m$ are defined by
\begin{equation}\nonumber
\begin{aligned}
\left.
\begin{array}{ll}\displaystyle
&f\equ -{\rm div}(a\bv),\\[1ex]\displaystyle
&{\bf g}\equ -\bv\cdot\nabla\bv-I(a)\widetilde{\mathcal{A}}\bv-K_1(a)\nabla a-K_2(a)\nabla\theta-\theta\nabla K_3(a),\\[1ex]
&\displaystyle{m}\equ -\bv\cdot\nabla\theta-\beta I(a)\Delta\theta+\frac{Q(\nabla\bv,\nabla\bv)}{1+a}-(K_2(a)+{K}_4(a)\theta){\rm div}\bv
\end{array}
\right.
\end{aligned}
\end{equation}
with
{\small\begin{equation}\nonumber
\begin{array}{ll}\displaystyle
I(a)\equ \frac{a}{1+a},\,\,\,\,\,\,K_1(a)\equ \frac{\partial_{\rho}P(\rho^\ast(1+a),\mathcal{T}^\ast)}{(1+a)\partial_{\rho}P(\rho^\ast,\mathcal{T}^\ast)}-1,\,\,\,\,\,K_4(a)\equ\frac{\pi_1(\rho^\ast(1+a))}{C_v\rho^\ast(1+a)}
\\[2ex]\displaystyle
K_2(a)\equ \frac{\chi_0}{\rho^\ast}\left(\frac{\mathcal{T}^\ast}{C_v}\right)^{\frac{1}{2}}\left(\frac{\pi_1(\rho^\ast(1+a))}{1+a}
-\pi_1(\rho^\ast)\right),\,\,\,\,K_3(a)\equ\chi_0\left(\frac{\mathcal{T}^\ast}{C_v}\right)^{\frac{1}{2}}
\int_0^a\frac{\pi_1^\prime(\rho^\ast(1+z))}{1+z}dz
\\[3ex]\displaystyle
Q(\nabla\bv,\nabla\bv)\equ\frac{1}{\nu\chi_0}\sqrt{\frac{1}{\mathcal{T}^\ast C_v}}\left(\frac{\mu}{2}
 |\nabla\bv+(\nabla\bv)^T|^2+\lambda({\rm div}\bv)^2\right).
\end{array}
\end{equation}}
Note that the exact value of $K_1, K_2, K_3$ and $K_4$ is not what one cares about and we mainly use that they are smooth and that $K_1(0)=K_2(0)=K_3(0)=0$.

We now state the main results of this paper as follows.
\begin{theo}\label{th2}
Let $N\geq 3$ and $p$ satisfy assumption \eqref{1.5}. Let $(\rho,\bu, \mathcal{T})$ be the global solution addressed by Theorem \ref{th1}. If in addition
$(a_0, \bv_0, \theta_0)^\ell\in \dot{B}_{2,\infty}^{-\sigma_1}\,(2-\frac{N}{2}<\sigma_1\leq \sigma_0
\equ \frac{2N}{p}-\frac{N}{2})$ such that $\|(a_0,\bv_0,\theta_0)\|_{\dot{B}_{2,\infty}^{-\sigma_1}}^\ell$
is bounded, then we have
\begin{equation}\label{1}
\|(a,\bv)(t)\|_{\dot{B}_{p,1}^{\sigma_2}}\lesssim (1+t)^{-\frac{N}{2}(\frac{1}{2}-\frac{1}{p})-\frac{\sigma_2+\sigma_1}{2}}
\end{equation}
 and
\begin{equation}\label{2}
\|\theta(t)\|_{\dot{B}_{p,1}^{\sigma_3}}\lesssim (1+t)^{-\frac{N}{2}(\frac{1}{2}-\frac{1}{p})-\frac{\sigma_3+\sigma_1}{2}},
\end{equation}
where $-\sigma_1
-\frac{N}{2}+\frac{N}{p}<\sigma_2\leq \frac{N}{p}-1$,\,\,{\rm and}\,\,$-\sigma_1
-\frac{N}{2}+\frac{N}{p}<\sigma_3\leq\frac{N}{p}-2$ for all $t\geq 0$.
\end{theo}

Denote $\Lambda^sf\equ \mathcal{F}^{-1}(|\xi|^s\mathcal{F}f)$ for $s\in\mathbb{R}$. By applying improved Gagliardo-Nirenberg inequalities, the  optimal  decay estimates of $\dot{B}_{2,\infty}^{-\sigma_1}$-$L^r$ type could be deduced as follows.
\begin{col}\label{col1}
Let those assumptions of Theorem \ref{th2} be fulfilled. Then the corresponding solution $(a,\bv,\theta)$ admits
\begin{equation}\nonumber
\|\Lambda^l(a,\bv)\|_{L^r}\lesssim (1+t)^{-\frac{N}{2}(\frac{1}{2}-\frac{1}{r})-\frac{l+\sigma_1}{2}}\,\,\,{\rm and}\,\,\,
\|\Lambda^n\theta\|_{L^r}\lesssim (1+t)^{-\frac{N}{2}(\frac{1}{2}-\frac{1}{r})-\frac{n+\sigma_1}{2}}
\end{equation}
where $l, n$ and  $r$ satisfy $-\sigma_1
-\frac{N}{2}+\frac{N}{p}<l+\frac{N}{p}-\frac{N}{r}\leq \frac{N}{p}-1$ and $-\sigma_1
-\frac{N}{2}+\frac{N}{p}<n+\frac{N}{p}-\frac{N}{r}\leq \frac{N}{p}-2$ for $p\leq r\leq \infty$ and $t\geq 0$.
\end{col}
\begin{re}
{\rm  In \cite{shi2019thelarge},  the low-frequency assumption of initial data is that there exists a small positive constant $c$ such that
$
\|(a_0,\bv_0,\theta_0)\|_{\dot{B}_{2,\infty}^{-\sigma_1}}^\ell\leq c\,\,\,{\rm with}\,\,\,1-\frac{N}{2}<\sigma_1\leq \sigma_0.
$
Here, the smallness at the low frequencies is removed in Theorem \ref{th2}. On the other hand, the decay rates in Corollary \ref{col1} are coincide with those obtained in \cite{shi2019thelarge}. As pointed out in  \cite{shi2019thelarge},  the decay rates in Corollary \ref{col1} are optimal and satisfactory. However, since the condition $-\sigma_1
-\frac{N}{2}+\frac{N}{p}<\sigma_3\leq\frac{N}{p}-2$ in Theorem \ref{th2} leads to that the values of $\sigma_1$ need to be more than $2-\frac{N}{2}$, and then the admissible value  $\sigma_1$ should belong to $(2-\frac{N}{2}, \sigma_0]$.}
 \end{re}

  It is noted that Xin and Xu \cite{xin2018optimal} developed a pure energy argument to establish the optimal decay for the barotropic compressible Navier-Stokes equations in the $L^p$ critical framework. As pointed out in \cite{xin2018optimal}, the nonlinear estimates at the low frequencies (that is $\|(f,{\bf g}, m)\|_{\dot{B}_{2,\infty}^{-\sigma_1}}^\ell$) play a fundamental role in the process of  proving Theorem \ref{th2}. Here,  based on the work \cite{xin2018optimal}, we develop two non-classical product estimates in the low frequencies (see \eqref{4.2} and \eqref{4.2000} below), which may allow us to handle the nonlinear terms in system \eqref{1.1}. Especially for terms including the temperature $\theta$, we need to use the product estimate \eqref{4.2000} (see for example the estimates \eqref{4.57} and \eqref{4.70} below). Moreover, we will make full use of the structure of system \eqref{1.1} itself.  For example, when dealing with  the trinomial term $Q(\nabla\bv,\nabla\bv)/(1+a)$, we  are going to take full advantage of its symmetrical structure (see \eqref{4.58}-\eqref{4.61} below).

 The rest of this paper is structured as follows. In  section \ref{s:3}, we recall some basic properties of  the homogeneous Besov spaces and give some classical and non-classical product estimates in Besov spaces.  In section \ref{s:4}, we give the low-frequency and high-frequency estimates to system \eqref{2.2}. Section \ref{s:5} is devoted to the estimation of $L^2$-type Besov norms at low frequencies, which plays a key role in deriving the Lyapunov-type inequality for energy norms. Section \ref{s:6}, i.e., the last section presents the proofs of Theorem \ref{th2} and Corollary \ref{col1}.
\section{Preliminaries}\label{s:3}
\setcounter{equation}{0}\setcounter{section}{3}\indent
Throughout the paper, $C$ stands for a harmless ``constant", and we sometimes write $A\lesssim B$ as an equivalent to $A\leq CB$. The notation $A\approx B$ means that $A\lesssim B$ and $B\lesssim A$. For any Banach space $X$ and $u, v\in X$, we agree that $\|(u,v)\|_X\equ \|u\|_X+\|v\|_X$. For  $p\in [1,+\infty]$ and $T>0$, the notation $L^p(0,T;X)$ or $L^p_T(X)$
denotes the set of measurable functions $f:[0,T]\rightarrow X$ with $t\mapsto
\|f(t)\|_X$ in $L^p(0,T)$, endowed with the norm
$
\|f\|_{L^p_T(X)}\equ\bigl{\|}\|f\|_X\bigr{\|}_{L^p(0,T)}.
$
We denote by $\mathcal{C}([0,T];X)$  the set of continuous functions from
$[0,T]$ to $X$.

We first recall the definition of homogeneous Besov spaces, which could be defined by using a dyadic partition of unity in Fourier variables called homogeneous Littlewood-Paley decomposition. Next, the product estimates in homogeneous Besov spaces are presented.
\subsection{Homogeneous Besov spaces}
 At this point, choose a radial function $\varphi\in \mathcal{S}(\mathbb{R}^N)$ supported in $\mathcal{C}=\{\xi\in\mathbb{R}^N, \frac{3}{4}\leq |\xi|\leq \frac{8}{3}\}$ such that
$
\sum_{j\in\mathbb{Z}}\varphi(2^{-j}\xi)=1\quad\!\!{\rm if}\quad\!\!\xi\neq 0.
$
The homogeneous frequency localization operator $\dot{\Delta}_j$ and $\dot{S}_j$ are defined by
$$
\dot{\Delta}_j u=\varphi (2^{-j}D)u, \quad\,\dot{S}_j u=\sum_{k\leq j-1}\dot{\Delta}_k u\quad\,{\rm for}\quad\,j\in\mathbb{Z}.
$$

Let us denote the space $\mathcal{Y}^\prime(\mathbb{R}^N)$ by the quotient space of $\mathcal{S}^\prime(\mathbb{R}^N)/\mathcal{P}$ with the polynomials space $\mathcal{P}$. The formal equality
$
u=\sum_{k\in\mathbb{Z}}\dot{\Delta}_k u
$
holds true for $u\in \mathcal{Y}^\prime(\mathbb{R}^N)$ and is called the homogeneous Littlewood-Paley decomposition.

We then define, for $s\in\mathbb{R}$, $1\leq p, r\leq +\infty$, the homogeneous Besov space
$$
\dot{B}_{p,r}^s={\Big\{}f\in\mathcal{Y}^\prime(\mathbb{R}^N): \|f\|_{\dot{B}_{p,r}^s}<+\infty{\Big\}},
$$
where
$$
\|f\|_{\dot{B}_{p,r}^s}\equ\|2^{ks}\|\dot{\Delta}_k f\|_{L^p}\|_{\ell^r}.
$$

We next introduce the so-called Chemin-Lerner space $\widetilde{L}_T^\rho(\dot{B}_{p,r}^s)$ (see\,\cite{chemin1995flot}):
$$
\widetilde{L}_T^\rho(\dot{B}_{p,r}^s)={\Big\{}f\in (0,+\infty)\times\mathcal{Y}^\prime(\mathbb{R}^N):
\|f\|_{\widetilde{L}_T^\rho(\dot{B}_{p,r}^s)}<+\infty{\Big\}},
$$
where
$
\|f\|_{\widetilde{L}_T^\rho(\dot{B}_{p,r}^s)}\equ\bigl{\|}2^{ks}\|\dot{\Delta}_k f(t)\|_{L^\rho(0,T;L^p)}\bigr{\|}_{\ell^r}.
$
The index $T$ will be omitted if $T=+\infty$ and we shall denote by $\widetilde{\mathcal{C}}_b([0,T]; \dot{B}^s_{p,r})$ the subset of functions of $\widetilde{L}^\infty_T(\dot{B}^s_{p,r})$ which are also continuous from
$[0,T]$ to $\dot{B}^s_{p,r}$. A direct application of Minkowski's inequality implies that
$$
L_T^\rho(\dot{B}_{p,r}^s)\hookrightarrow \widetilde{L}_T^\rho(\dot{B}_{p,r}^s)\,\,\,{\rm if}\,\,\,r\geq \rho,
\quad\,{\rm and}\quad\,
\widetilde{L}_T^\rho(\dot{B}_{p,r}^s)\hookrightarrow {L}_T^\rho(\dot{B}_{p,r}^s)\,\,\,{\rm if}\,\,\,\rho\geq r.
$$
We will repeatedly use the following Bernstein's inequality throughout the paper:
\begin{lem}\label{le2.1}
{\rm(}see {\rm \cite{chemin1998perfect}}{\rm )} Let $\mathcal{C}$ be an annulus and $\mathcal{B}$ a ball, $1\leq p\leq q\leq +\infty$. Assume that $f\in L^p(\mathbb{R}^N)$, then for any nonnegative integer $k$, there exists constant $C$ independent of $f$, $k$ such that
$$
{\rm supp} \hat{f}\subset\lambda \mathcal{B}\Rightarrow\|D^k f\|_{L^q(\mathbb{R}^N)}:=\sup_{|\alpha|=k}\|\partial^\alpha f\|_{L^q(\mathbb{R}^N)}\leq C^{k+1}\lambda^{k+N(\frac{1}{p}-\frac{1}{q})}\|f\|_{L^p(\mathbb{R}^N)},
$$
\be\nonumber
{\rm supp} \hat{f}\subset\lambda\mathcal{C}\Rightarrow C^{-k-1}\lambda^k\|f\|_{L^p(\mathbb{R}^N)}\leq \|D^k f\|_{L^p(\mathbb{R}^N)}\leq
C^{k+1}\lambda^k\|f\|_{L^p(\mathbb{R}^N)}.
\ee
\end{lem}

More generally, if $v$ satisfies ${\rm supp}\mathcal{F}v\subset\{\xi\in\mathbb{R}^N: R_1\lambda\leq |\xi|\leq R_2\lambda\}$
for some $0< R_1<R_2$ and $\lambda>0$, then for any smooth homogeneous of degree $m$ function $A$ on $\mathbb{R}^N\backslash\{0\}$ and $1\leq q\leq \infty$, it holds that (see e.g. Lemma 2.2 in \cite{bahouri2011fourier}):
\be\label{2.100}
\|A(D)v\|_{L^q}\lesssim \lambda^m\|v\|_{L^q}.
\ee

The following nonlinear generalization of \eqref{2.100} will be applied (see Lemma 8 in \cite{danchin2010well-posedness}):

\begin{prop}\label{pr2.3}
If ${\rm Supp}\mathcal{F}f\subset\{\xi\in\mathbb{R}^N: R_1\lambda\leq |\xi|\leq R_2\lambda\}$ then there exists $c$ depending only on $N, R_1$ and $R_2$ so that for all $1<p<\infty$,
$$
c\lambda^2\left(\frac{p-1}{p^2}\right)\int_{\mathbb{R}^N}|f|^pdx\leq (p-1)\int_{\mathbb{R}^N}|\nabla f|^2|f|^{p-2}dx
=-\int_{\mathbb{R}^N}\Delta f|f|^{p-2}fdx.
$$
\end{prop}

Let us now state some classical properties for the Besov spaces.
\begin{prop}\label{pr2.1} The following properties hold true:

\medskip
{\rm 1)} Derivation: There exists a universal constant $C$ such that
$$
C^{-1}\|f\|_{\dot{B}_{p,r}^s}\leq \|\nabla f\|_{\dot{B}_{p,r}^{s-1}}\leq C\|f\|_{\dot{B}_{p,r}^s}.
$$

{\rm 2)} Sobolev embedding: If $1\leq p_1\leq p_2\leq\infty$ and $1\leq r_1\leq r_2\leq\infty$, then $\dot{B}_{p_1, r_1}^s\hookrightarrow \dot{B}_{p_2, r_2}^{s-\frac{N}{p_1}+\frac{N}{p_2}}$.

{\rm 3)} Real interpolation: $\|f\|_{\dot{B}_{p,r}^{\theta s_1+(1-\theta)s_2}}\leq \|f\|_{\dot{B}_{p,r}^{s_1}}^{\theta}\|f\|_{\dot{B}_{p,r}^{s_2}}^{1-\theta}$.

{\rm 4)} Algebraic properties: for $s>0$, $\dot{B}_{p,1}^s\cap L^\infty$ is an algebra.
\end{prop}
\subsection{Product estimates} We recall a few nonlinear estimates in Besov spaces which may be derived
by using paradifferential calculus.  Introduced by  Bony in \cite{bony1981calcul}, the paraproduct between $f$
and $g$ is defined by
$$
T_f g=\sum_{k\in\mathbb{Z}}\dot{S}_{k-1}f\dot{\Delta}_k g,
$$
and the remainder is given by
$$
R(f,g)=\sum_{k\in\mathbb{Z}}\dot{\Delta}_k f\widetilde{\dot{\Delta}}_k g\,\,\,\,{\rm with}\,\,\,\,\widetilde{\dot{\Delta}}_k g\equ(\dot{\Delta}_{k-1}+\dot{\Delta}_{k}+\dot{\Delta}_{k+1})g.
$$
One has  the following so-called Bony's decomposition:
\be\label{2.3}
fg=T_g f+T_f g+R(f,g).
\ee

The paraproduct $T$ and the remainder $R$ operators satisfy the following continuous properties (see e.g. \cite{bahouri2011fourier}).
 \begin{prop}\label{pr2.2}
Suppose that $s\in\mathbb{R}, \sigma>0,$ and $1\leq p, p_1, p_2, r, r_1, r_2\leq \infty$.  Then we have
 \medskip

{\rm 1)} The paraproduct $T$ is a bilinear, continuous operator from $L^\infty\times\dot{B}_{p,r}^s$ to $\dot{B}_{p,r}^s$, and from $\dot{B}_{\infty, r_1}^{-\sigma}\times\dot{B}_{p,r_2}^s$ to $\dot{B}_{p,r}^{s-\sigma}$ with
$\frac{1}{r}=\min\{1, \frac{1}{r_1}+\frac{1}{r_2}\}$.

\medskip
{\rm 2)} The remainder $R$ is bilinear continuous from $\dot{B}_{p_1,r_1}^{s_1}\times\dot{B}_{p_2,r_2}^{s_2}$ to $\dot{B}_{p,r}^{s_1+s_2}$ with $s_1+s_2>0$, $\frac{1}{p}=\frac{1}{p_1}+\frac{1}{p_2}\leq 1$, and $\frac{1}{r}=\frac{1}{r_1}+\frac{1}{r_2}\leq 1$.

 \end{prop}
 \medbreak
The following  non-classical product estimates enable us to establish the evolution of Besov norms at low frequencies
(see Lemma \ref{le3} below).
\begin{prop}\label{pr4.1}
Let $1-\frac{N}{2}<\sigma_1\leq \frac{2N}{p}-\frac{N}{2} (N\geq 2)$ and $p$ satisfy \eqref{1.5}.
Then the following estimates  hold true:
\be\label{4.1}
\|fg\|_{\dot{B}_{2,\infty}^{-\sigma_1}}\lesssim\|f\|_{\dot{B}_{p,1}^{\frac{N}{p}}}\|g\|_{\dot{B}_{2,\infty}^{-\sigma_1}},
\ee
\be\label{4.2}
\|fg\|_{\dot{B}_{2,\infty}^{-\sigma_1}}^\ell\lesssim\|f\|_{\dot{B}_{p,1}^{\frac{N}{p}-1}}
\left(\|g\|_{\dot{B}_{p,\infty}^{-\sigma_1+\frac{N}{p}-\frac{N}{2}+1}}
+\|g\|_{\dot{B}_{p,\infty}^{-\sigma_1+\frac{2N}{p}-N+1}}\right)
\ee
and
\be\label{4.2000}
\|fg\|_{\dot{B}_{2,\infty}^{-\sigma_1}}^\ell\lesssim\|f\|_{\dot{B}_{p,1}^{\frac{N}{p}-2}}
\left(\|g\|_{\dot{B}_{p,\infty}^{-\sigma_1+\frac{N}{p}-\frac{N}{2}+2}}
+\|g\|_{\dot{B}_{p,\infty}^{-\sigma_1+\frac{2N}{p}-N+1}}\right).
\ee
\end{prop}
\begin{proof}
 Inequalities \eqref{4.1} and \eqref{4.2} have been proved in \cite{bie2019optimal}. Here we only prove \eqref{4.2000}. Denote $p^\ast\equ\frac{2p}{p-2}$, i.e., $\frac{1}{p}+\frac{1}{p^\ast}=\frac{1}{2}$.
By \eqref{2.3}, we decompose $fg$ into $T_fg+R(f,g)+T_gf$. For the paraproduct term $T_fg$, we have
{\small\begin{equation}\label{4.6}
\begin{split}
&\quad\|\dot{\Delta}_j(T_fg)\|_{L^2}
=\|\sum_{|k-j|\leq 4}\dot{\Delta}_j(\dot{S}_{k-1}f\dot{\Delta}_kg)\|_{L^2}\leq\sum_{|k-j|\leq 4}\sum_{k^\prime\leq k-2}\|\dot{\Delta}_j(\dot{\Delta}_{k^\prime}f\dot{\Delta}_kg)\|_{L^2}\\[1ex]
&\lesssim\sum_{|k-j|\leq 4}\sum_{k^\prime\leq k-2}\|\dot{\Delta}_{k^\prime}f\|_{L^{p^\ast}}\|\dot{\Delta}_kg\|_{L^p}\lesssim\sum_{|k-j|\leq 4}\sum_{k^\prime\leq k-2}2^{k^\prime(\frac{2N}{p}-\frac{N}{2})}
\|\dot{\Delta}_{k^\prime}f\|_{L^p}\|\dot{\Delta}_kg\|_{L^p}\\[1ex]
&\lesssim\sum_{|k-j|\leq 4}\sum_{k^\prime\leq k-2}2^{k^\prime(\frac{2N}{p}-\frac{N}{2}+2-\frac{N}{p})}2^{k^\prime(\frac{N}{p}-2)}
\|\dot{\Delta}_{k^\prime}f\|_{L^p}
2^{k(\sigma_1-\frac{N}{p}+\frac{N}{2}-2)}2^{-k(\sigma_1-\frac{N}{p}+\frac{N}{2}-2)}\|\dot{\Delta}_kg\|_{L^p}\\[1ex]
&\lesssim 2^{j\sigma_1}\|f\|_{\dot{B}_{p,1}^{\frac{N}{p}-2}}\|g\|_{\dot{B}_{p,\infty}^{-\sigma_1+\frac{N}{p}-\frac{N}{2}+2}},
\end{split}
\end{equation}}
where we have used that $2+\frac{N}{p}-\frac{N}{2}> 0$ and $p^\ast\geq p$ since $p$ fulfills $2\leq p\leq \min(4,\frac{2N}{N-2})$.

For the remainder term, one gets
{\small\begin{equation}\label{4.7}
\begin{split}
&\quad\|\dot{\Delta}_jR(f,g)\|_{L^2}\leq \sum_{k\geq j-3}\sum_{|k-k^\prime|\leq 1}
\|\dot{\Delta}_j(\dot{\Delta}_{k}f\dot{\Delta}_{k^\prime}g)\|_{L^2}\\[1ex]
&\lesssim 2^{j(\frac{2N}{p}-\frac{N}{2})}\sum_{k\geq j-3}\sum_{|k-k^\prime|\leq 1}2^{k(2-\frac{N}{p})}2^{k(\frac{N}{p}-2)}\|\dot{\Delta}_{k}f\|_{L^{p}}2^{k^\prime(\sigma_1-\frac{N}{p}+\frac{N}{2}-2)}
2^{-k^\prime(\sigma_1-\frac{N}{p}+\frac{N}{2}-2)}\|\dot{\Delta}_{k^\prime}g\|_{L^{p}}\\[1ex]
&\lesssim2^{j(\frac{2N}{p}-\frac{N}{2})}\sum_{k\geq j-3}2^{k(\sigma_1-\frac{2N}{p}+\frac{N}{2})}c(k)\|f\|_{\dot{B}_{p,1}^{\frac{N}{p}-2}}
\|g\|_{\dot{B}_{p,\infty}^{-\sigma_1+\frac{N}{p}-\frac{N}{2}+2}}\\[1ex]
&
\lesssim2^{j\sigma_1}\|f\|_{\dot{B}_{p,1}^{\frac{N}{p}-2}}\|g\|_{\dot{B}_{p,\infty}^{-\sigma_1+\frac{N}{p}-\frac{N}{2}+2}},
\end{split}
\end{equation}}
here $\|c(k)\|_{l^1}=1$ and  we have used the condition $\sigma_1\leq \frac{2N}{p}-\frac{N}{2}$ in the last inequality.

For the  term $T_gf$,  we  have that
{\small\begin{equation}\label{4.8}
\begin{split}
&\quad\|\dot{\Delta}_j(T_gf)\|_{L^2}\leq \sum_{|k-j|\leq 4}\sum_{k^\prime\leq k-2}\|\dot{\Delta}_j(\dot{\Delta}_{k^\prime}g\dot{\Delta}_kf)\|_{L^2}\leq\sum_{|k-j|\leq 4}\sum_{k^\prime\leq k-2}\|\dot{\Delta}_{k^\prime}g\|_{L^{p^\ast}}\|\dot{\Delta}_kf\|_{L^p}\\[1ex]
&\lesssim\sum_{|k-j|\leq 4}\sum_{k^\prime\leq k-2}2^{k^\prime(\frac{2N}{p}-\frac{N}{2})}\|\dot{\Delta}_{k^\prime}g\|_{L^p}\|\dot{\Delta}_kf\|_{L^p}\\[1ex]
&\lesssim\sum_{|k-j|\leq 4}\sum_{k^\prime\leq k-2}2^{k^\prime(\frac{2N}{p}-\frac{N}{2}+\sigma_1-\frac{2N}{p}+N-1)}2^{k^\prime(-\sigma_1+\frac{2N}{p}-N+1)}
\|\dot{\Delta}_{k^\prime}g\|_{L^p}2^{k(2-\frac{N}{p})}2^{k(\frac{N}{p}-2)}
\|\dot{\Delta}_kf\|_{L^p}\\[1ex]
&\lesssim 2^{j(\sigma_1+\frac{N}{2}-\frac{N}{p}+1)}
\|f\|_{\dot{B}_{p,1}^{\frac{N}{p}-2}}\|g\|_{\dot{B}_{p,\infty}^{-\sigma_1+\frac{2N}{p}-N+1}},
\end{split}
\end{equation}}
where we used that $\sigma_1>1-\frac{N}{2}$ in the last inequality.

From \eqref{4.6} and \eqref{4.7}, we deduce
\be\label{4.9}
\|T_fg+R(f,g)\|_{\dot{B}_{2,\infty}^{-\sigma_1}}\lesssim\|f\|_{\dot{B}_{p,1}^{\frac{N}{p}-2}}
\|g\|_{\dot{B}_{p,\infty}^{-\sigma_1+\frac{N}{p}-\frac{N}{2}+2}}
\ee
and from \eqref{4.8}, we get that
\be\label{4.10}
\|T_gf\|_{\dot{B}_{2,\infty}^{-\sigma_1}}^\ell\leq
\|T_gf\|_{\dot{B}_{2,\infty}^{-\sigma_1+\frac{N}{p}-\frac{N}{2}-1}}^\ell
\lesssim\|f\|_{\dot{B}_{p,1}^{\frac{N}{p}-2}}\|g\|_{\dot{B}_{p,\infty}^{-\sigma_1+\frac{2N}{p}-N+1}}.
\ee
Note that  only here we used the low frequency condition to ensure that
\be\label{A1}
\|u\|_{\dot{B}_{2,\infty}^{-\sigma_1}}^\ell\leq
\|u\|_{\dot{B}_{2,\infty}^{-\sigma_1+\frac{N}{p}-\frac{N}{2}-1}}^\ell
\ee
since $\frac{N}{p}-\frac{N}{2}-1<0$ when $p\geq 2$. Combining \eqref{4.9} and \eqref{4.10}, we derive that \eqref{4.2000} holds true.
\end{proof}
\begin{re}
{\rm  We can see from  the proof of \eqref{4.2000} that if $f=g$, using the property of symmetry, we no longer need the estimate \eqref{4.10} and in turn \eqref{A1} at the low-frequency. Then, we delete the condition of low-frequency in \eqref{4.2000} and get directly that
\be\label{4.2002}
\|ff\|_{\dot{B}_{2,\infty}^{-\sigma_1}}\lesssim\|f\|_{\dot{B}_{p,1}^{\frac{N}{p}-2}}
\|f\|_{\dot{B}_{p,\infty}^{-\sigma_1+\frac{N}{p}-\frac{N}{2}+2}}.
\ee
Similarly, \eqref{4.2} becomes
\be\label{4.2001}
\|ff\|_{\dot{B}_{2,\infty}^{-\sigma_1}}\lesssim\|f\|_{\dot{B}_{p,1}^{\frac{N}{p}-1}}
\|f\|_{\dot{B}_{p,\infty}^{-\sigma_1+\frac{N}{p}-\frac{N}{2}+1}}.
\ee}
\end{re}
From Bony's decomposition \eqref{2.3} and Proposition \ref{pr2.2}, we could as well infer the following product estimates:
\begin{col}\label{co2.1} {\rm(}\cite{bahouri2011fourier}, \cite{danchin2002zero}{\rm )}
{\rm\,(i)} Let $s>0$ and $1\leq p,r\leq \infty$. Then $\dot{B}_{p,r}^s\cap L^\infty$ is an algebra and
$$
\|uv\|_{\dot{B}_{p,r}^s}\lesssim \|u\|_{L^\infty}\|v\|_{\dot{B}_{p,r}^s}+\|v\|_{L^\infty}\|u\|_{\dot{B}_{p,r}^s}.
$$

{\rm (ii)}\,If $u\in\dot{B}_{p_1,1}^{s_1}$ and $v\in\dot{B}_{p_2,1}^{s_2}$ with $1\leq p_1\leq p_2\leq \infty,~s_1\leq \frac{N}{p_1},~s_2\leq \frac{N}{p_2}$ and $s_1+s_2>0$, then
$uv\in\dot{B}_{p_2,1}^{s_1+s_2-\frac{N}{p_1}}$ and there exists a constant $C$, depending only on $N, s_1, s_2, p_1$ and $p_2$, such that
\be\nonumber
\|uv\|_{\dot{B}_{p_2,1}^{s_1+s_2-\frac{N}{p_1}}}\leq C\|u\|_{\dot{B}_{p_1,1}^{s_1}}
\|v\|_{\dot{B}_{p_2,1}^{s_2}}.
\ee
\end{col}

\begin{col}\label{co2.2}
Let $\sigma_1$ satisfy $1-\frac{N}{2}<\sigma_1\leq\frac{2N}{p}-\frac{N}{2}\,(N\geq 2)$ and  $p$ fulfill \eqref{1.5}, then we have
\be\nonumber
\|fg\|_{\dot{B}_{p,\infty}^{-\sigma_1+\frac{N}{p}-\frac{N}{2}+1}}\lesssim\|f\|_{\dot{B}_{p,1}^{\frac{N}{p}}}
\|g\|_{\dot{B}_{p,\infty}^{-\sigma_1+\frac{N}{p}-\frac{N}{2}+1}},
\ee
and
\be\nonumber
\|fg\|_{\dot{B}_{p,\infty}^{-\sigma_1+\frac{2N}{p}-N+1}}\lesssim\|f\|_{\dot{B}_{p,1}^{\frac{N}{p}}}
\|g\|_{\dot{B}_{p,\infty}^{-\sigma_1+\frac{2N}{p}-N+1}}.
\ee
Moreover, if $2-\frac{N}{2}<\sigma_1\leq\frac{2N}{p}-\frac{N}{2}\,(N\geq 3)$ and  $p$ satisfies \eqref{1.5}, it holds that
\be\nonumber
\|fg\|_{\dot{B}_{p,\infty}^{-\sigma_1+\frac{2N}{p}-N+2}}\lesssim\|f\|_{\dot{B}_{p,1}^{\frac{N}{p}}}
\|g\|_{\dot{B}_{p,\infty}^{-\sigma_1+\frac{2N}{p}-N+2}}.
\ee
\end{col}

We also need the following  composition lemma (see \cite{bahouri2011fourier,danchin2000global, runst1996sobolev}).
\begin{prop}\label{pra.4}

Let $F:\mathbb{R}\rightarrow \mathbb{R}$ be smooth with $F(0)=0$. For all $1\leq p,r\leq \infty$ and $s>0$, it holds that $F(u)\in \dot{B}_{p,r}^s\cap L^\infty$ for $u\in\dot{B}_{p,r}^s\cap L^\infty$, and
$$
\|F(u)\|_{\dot{B}_{p,r}^s}\leq C\|u\|_{\dot{B}_{p,r}^s}
$$
with $C$ depending only on $\|u\|_{L^\infty}$, $F^\prime$ (and higher derivatives), $s, p$ and $N$.

In the case $s>-\min(\frac{N}{p}, \frac{N}{p^\prime})$, then $u\in\dot{B}_{p,r}^s\cap\dot{B}_{p,1}^{\frac{N}{p}}$ implies that
$F(u)\in\dot{B}_{p,r}^s\cap\dot{B}_{p,1}^{\frac{N}{p}}$, and
$$
\|F(u)\|_{\dot{B}_{p,r}^s}\leq C(1+\|u\|_{\dot{B}_{p,1}^{\frac{N}{p}}})\|u\|_{\dot{B}_{p,r}^s},
$$
where $\frac{1}{p}+\frac{1}{p^\prime}=1$.
\end{prop}

The following commutator estimate (see \cite{danchin2016optimal}) has been employed in the high-frequency estimates.
\begin{prop}\label{pra.5}
Let $1\leq p, p_1\leq \infty$ and
\be\label{A4}
-\min\Big{\{}\frac{N}{p_1}, \frac{N}{p^\prime}\Big{\}}<\sigma\leq 1+
\min\Big{\{}\frac{N}{p}, \frac{N}{p_1}\Big{\}}.
\ee
There exists a constant $C>0$ depending only on $\sigma$ such that
for all $j\in\mathbb{Z}$ and $i\in \{1,\cdots, N\}$, we have
\be\label{A5}
\|[v\cdot\nabla,\partial_i{\dot{\Delta}_j}]a\|_{L^p}\leq Cc_j 2^{-j(\sigma-1)}
\|\nabla v\|_{\dot{B}_{p_1,1}^{\frac{N}{p_1}}}\|\nabla a\|
_{\dot{B}_{p,1}^{\sigma-1}},
\ee
where the commutator $[\cdot,\cdot]$ is defined by $[f,g]=fg-gf$, and
 $(c_j)_{j\in\mathbb{Z}}$ denotes a sequence such that $\|(c_j)\|_{\ell^1}\leq 1$ and $\frac{1}{p^\prime}+\frac{1}{p}=1$.
\end{prop}

At last,  we present the optimal regularity estimates for the heat equation (see e.g. \cite{bahouri2011fourier}).
\begin{prop}\label{pra.6}
Let $\sigma\in\mathbb{R},\,\, (p,r)\in [1,\infty]^2$ and $1\leq \rho_2\leq \rho_1\leq \infty$. Let $u$
satisfy
\be\nonumber
\left\{\begin{array}{ll}\medskip\D
\partial_t u-\mu \Delta u=f,\\ \D
u|_{t=0}=u_0.
\end{array}
\right.
\ee
Then for all $T>0$, the following a prior estimate is satisfied:
\be\nonumber
\mu^{\frac{1}{\rho_1}}\|u\|_{\widetilde{L}_T^{\rho_1}(\dot{B}_{p,r}^{\sigma+\frac{2}{\rho_1}})}
\lesssim \|u_0\|_{\dot{B}_{p,r}^\sigma}+\mu^{\frac{1}{\rho_2}-1}
\|f\|_{\widetilde{L}_T^{\rho_2}(\dot{B}_{p,r}^{\sigma-2+\frac{2}{\rho_2}})}.
\ee
\end{prop}

\section{Low-frequency and high-frequency estimates}\label{s:4}
In order to obtain a Lyapunov-type inequality for energy norms in next section, we now give the low-frequency and high-frequency estimates to system \eqref{2.2}.

\subsection{Low-frequency estimates}
\begin{lemma}\label{lem2}
Let $k_0$ be some integer. Then the solution $(a, \bv,\theta)$ to system \eqref{2.2} satisfies
\begin{equation}\label{3.100}
\frac{d}{dt}\|(a,\bv,\theta)\|_{\dot{B}_{2,1}^{\frac{N}{2}-1}}^\ell
+\|(a,\bv,\theta)\|_{\dot{B}_{2,1}^{\frac{N}{2}+1}}^\ell\lesssim \|(f,{\bf g},{m})\|_{\dot{B}_{2,1}^{\frac{N}{2}-1}}^\ell
\end{equation}
for all $t\geq 0$, where
$
\|z\|_{\dot{B}_{2,1}^s}^\ell\equ \sum_{k\leq k_0}2^{ks}\|\dot{\Delta}_kz\|_{L^2}\,\,\,{\rm for}\,\,\,s\in\mathbb{R}.
$
\end{lemma}
\begin{proof}
Denote $\Lambda^sz\equ \mathcal{F}^{-1}(|\xi|^s\mathcal{F}z),\,\,s\in\mathbb{R}$ and set
\be\label{4.125}
\omega=\Lambda^{-1}{\rm div}{\bf v},\,\,\,\,\, {\bf \Omega}=\Lambda^{-1}{\rm curl}{\bf v}.
\ee
Then system \eqref{2.2} can be rewritten as
\begin{equation}\label{3.1}
\left\{
\begin{array}{ll}
\partial_t a+\Lambda\omega=f,\\[1ex]
\partial_t\omega-\Delta\omega-\Lambda a-\gamma\Lambda\theta=\Lambda^{-1}{\rm div}{\bf g},\\[1ex]
\partial_t{\bf\Omega}-\frac{\mu}{\nu}\Delta{\bf \Omega}=\Lambda^{-1}{\rm curl}{\bf g},\\[1ex]
\partial_t\theta-\beta\Delta \theta+\gamma\Lambda\omega=m,\\[1ex]
\bv=-\Lambda^{-1}\nabla\omega+\Lambda^{-1}{\rm div}{\bf\Omega}.
\end{array}
\right.
\end{equation}
Denote $n_k\equ \dot{\Delta}_kn$. Let us consider the following frequency localized system:
\begin{equation}\label{3.2}
\left\{
\begin{array}{ll}
\partial_t a_k+\Lambda\omega_k=f_k,\\[1ex]
\partial_t\omega_k-\Delta\omega_k-\Lambda a_k-\gamma\Lambda\theta_k=\Lambda^{-1}{\rm div}{\bf g}_k,\\[1ex]
\partial_t{\bf\Omega}_k-\frac{\mu}{\nu}\Delta{\bf \Omega}_k=\Lambda^{-1}{\rm curl}{\bf g}_k,\\[1ex]
\partial_t\theta_k-\beta\Delta \theta_k+\gamma\Lambda\omega_k=m_k.
\end{array}
\right.
\end{equation}
Taking the $L^2$ scalar product of \eqref{3.2}$_1$ with $a_k$, \eqref{3.2}$_2$ with $\omega_k$,
 \eqref{3.2}$_3$ with ${\bf\Omega}_k$, and \eqref{3.2}$_4$ with $\theta_k$, respectively, we have that
\begin{equation}\label{3.3}
\frac{1}{2}\frac{d}{dt}\|a_k\|_{L^2}^2+(\Lambda\omega_k,a_k)=(f_k, a_k),
\end{equation}
\begin{equation}\label{3.4}
\frac{1}{2}\frac{d}{dt}\|\omega_k\|_{L^2}^2+\|\Lambda \omega_k\|_{L^2}^2-(\Lambda a_k,\omega_k)
-\gamma(\Lambda\theta_k, \omega_k)=(\Lambda^{-1}{\rm div}{\bf g}_k, \omega_k),
\end{equation}
\begin{equation}\label{3.4000}
\frac{1}{2}\frac{d}{dt}\|{\bf\Omega}_k\|_{L^2}^2+\frac{\mu}{\nu}\|\Lambda {\bf\Omega}_k\|_{L^2}^2=(\Lambda^{-1}{\rm curl}{\bf g}_k, {\bf\Omega}_k),
\end{equation}
\begin{equation}\label{3.5}
\frac{1}{2}\frac{d}{dt}\|\theta_k\|_{L^2}^2+\beta\|\Lambda \theta_k\|_{L^2}^2
+\gamma(\Lambda\omega_k, \theta_k)=(m_k, \theta_k).
\end{equation}
Notice that
$$
(\Lambda \omega_k, a_k)=(\Lambda a_k, \omega_k)\,\,\,{\rm and}\,\,\,(\Lambda \theta_k, \omega_k)=(\Lambda \omega_k, \theta_k).
$$
Combing \eqref{3.3}, \eqref{3.4}, \eqref{3.4000} and \eqref{3.5} yields
\begin{equation}\label{3.7}
\begin{aligned}
&\quad\frac{1}{2}\frac{d}{dt}\left(\|a_k\|_{L^2}^2+\|\omega_k\|_{L^2}^2+\|{\bf\Omega}_k\|_{L^2}^2+\|\theta_k\|_{L^2}^2\right)
+\|\Lambda\omega_k\|_{L^2}^2
+\frac{\mu}{\nu}\|\Lambda {\bf\Omega}_k\|_{L^2}^2+\beta\|\Lambda\theta_k\|_
{L^2}^2\\[1ex]
&=(f_k, a_k)+(\Lambda^{-1}{\rm div}{\bf g}_k, \omega_k)+(\Lambda^{-1}{\rm curl}{\bf g}_k, {\bf\Omega}_k)+(m_k, \theta_k).
\end{aligned}
\end{equation}

Taking the $L^2$ scalar product of \eqref{3.2}$_1$ with $\Lambda\omega_k$, \eqref{3.2}$_2$ with $\Lambda a_k$, and \eqref{3.2}$_1$ with $\Lambda^2 a_k$, we obtain, respectively, that
\begin{equation}\nonumber
\left.
\begin{array}{ll}
(\partial_t a_k,\Lambda\omega_k)+\|\Lambda\omega_k\|_{L^2}^2=(f_k,\Lambda\omega_k),\\[1ex]
(\partial_t\omega_k, \Lambda a_k)+(\Lambda^2\omega_k, \Lambda a_k)-\|\Lambda a_k\|_{L^2}^2-\gamma(\Lambda\theta_k,\Lambda a_k)=(\Lambda^{-1}{\rm div}{\bf g}_k,\Lambda a_k),\\[1ex]\displaystyle
\frac{1}{2}\frac{d}{dt}\|\Lambda a_k\|_{L^2}^2+(\Lambda \omega_k, \Lambda^2a_k)=(\Lambda f_k, \Lambda a_k)
\end{array}
\right.
\end{equation}
which yields
\begin{equation}\label{3.8}
\begin{split}
&\quad\frac{1}{2}\frac{d}{dt}\left(\|\Lambda a_k\|_{L^2}^2-2(a_k,\Lambda\omega_k)\right)+\|\Lambda a_k\|_
{L^2}^2-\|\Lambda\omega_k\|_{L^2}^2+\gamma(\Lambda\theta_k,\Lambda a_k)\\[1ex]
&=(\Lambda f_k, \Lambda a_k)-(f_k, \Lambda \omega_k)-(\Lambda^{-1}{\rm div}{\bf g}_k, \Lambda a_k).
\end{split}
\end{equation}
Set
$$
\mathcal{J}_k^2(t)\equ \|a_k\|_{L^2}^2+\|\omega_k\|_{L^2}^2+\|{\bf \Omega}_k\|_{L^2}^2
+\|\theta_k\|_{L^2}^2+\xi\left(\|\Lambda a_k\|_{L^2}^2-2(a_k, \Lambda \omega_k)\right)
$$
for some $\xi>0$, we get from \eqref{3.7} and \eqref{3.8} that
\begin{equation}\label{3.105}
\begin{split}
&\quad\frac{1}{2}\frac{d}{dt}\mathcal{J}_k^2(t)+(1-\xi)\|\Lambda\omega_k\|_{L^2}^2
+\frac{\mu}{\nu}\|\Lambda{\bf \Omega}_k\|_{L^2}^2+\beta\|\Lambda\theta_k\|_{L^2}^2
+\xi\left(\|\Lambda a_k\|_{L^2}^2+\gamma(\Lambda\theta_k,\Lambda a_k)\right)\\[1ex]
&=(f_k, a_k)+(\Lambda^{-1}{\rm div}{\bf g}_k, \omega_k)+(\Lambda^{-1}{\rm curl}{\bf g}_k, {\bf\Omega}_k)+(m_k, \theta_k)\\[1ex]
&\quad\quad\quad\quad\quad\quad\quad\quad\quad\quad\quad\quad\quad\quad\quad\quad+
\xi\Big[(\Lambda f_k, \Lambda a_k)-(f_k, \Lambda \omega_k)-(\Lambda^{-1}{\rm div}{\bf g}_k, \Lambda a_k)\Big].
\end{split}
\end{equation}
Thus thanks to Young's inequality, one has that
\begin{equation}\label{3.103}
\mathcal{J}_k^2(t)\thickapprox\|(a_k, \Lambda a_k, \omega_k, {\bf \Omega}_k, \theta_k)\|_{L^2}^2
\thickapprox\|(a_k, \omega_k, {\bf \Omega}_k, \theta_k)\|_{L^2}^2
\end{equation}
for $k\leq k_0$. As a consequence,  from \eqref{3.105}, we get  in the low-frequency that
\begin{equation}\label{3.9}
\frac{1}{2}\frac{d}{dt}\mathcal{J}_k^2+2^{2k}\mathcal{J}_k^2\lesssim\|(f_k, \Lambda^{-1}{\rm div}{\bf g}_k, \Lambda^{-1}{\rm curl}{\bf g}_k, m_k)\|_{L^2}\mathcal{J}_k,
\end{equation}
which gives
\begin{equation}\label{3.10}
\frac{d}{dt}\mathcal{J}_k+2^{2k}\mathcal{J}_k\lesssim\|(f_k, \Lambda^{-1}{\rm div}{\bf g}_k, \Lambda^{-1}{\rm curl}{\bf g}_k, m_k)\|_{L^2}
\end{equation}
for $k\leq k_0$. Therefore, multiplying both sides of \eqref{3.10} by $2^{k(N/2-1)}$ and summing up on $k\leq k_0$, we finally get \eqref{3.100}.
\end{proof}
\subsection{High-frequency estimates} In the high-frequency regime, the term ${\rm div}(a\bv)$ would cause a loss of one derivative as there is no smoothing effect for $a$. To get around  this difficulty, as in \cite{haspot2011existence}, we introduce the effective velocity
\begin{equation}\label{3.102}
\bw\equ \nabla (-\Delta)^{-1}(a-{\rm div}\bv).
\end{equation}
\begin{lemma}\label{lem3}
Let $k_0$ be chosen suitably large. Then the solution $(a, \bv,\theta)$ to system \eqref{2.2} fulfills
\begin{equation}\label{3.11}
\begin{split}
&\quad\frac{d}{dt}\Big(\|(\nabla a,\bv)\|_{\dot{B}_{p,1}^{\frac{N}{p}-1}}^h+\|\theta\|_{\dot{B}_{p,1}^{\frac{N}{p}-2}}^h\Big)+\Big(\|(a, \theta)\|_{\dot{B}_{p,1}^{\frac{N}{p}}}^h
+\|\bv\|_{\dot{B}_{p,1}^{\frac{N}{p}+1}}^h\Big)\\[1ex]
&\lesssim \|(f,m)\|_{\dot{B}_{p,1}^{\frac{N}{p}-2}}^h+\|{\bf g}\|_{\dot{B}_{p,1}^{\frac{N}{p}-1}}^h
+\|\nabla\bv\|_{\dot{B}_{p,1}^{\frac{N}{p}}}\|a\|_{\dot{B}_{p,1}^{\frac{N}{p}}}
\end{split}
\end{equation}
for all $t\geq 0$, where
$\|z\|_{\dot{B}_{2,1}^s}^h\equ \sum_{k\geq k_0+1}2^{ks}\|\dot{\Delta}_kz\|_{L^2}\,\,\,{\rm for}\,\,\,s\in\mathbb{R}$.
\end{lemma}
\begin{proof}
Let $\mathcal{P}\equ {\rm Id}+\nabla(-\Delta)^{-1}{\rm div}$ be the Leray projector onto divergence-free vector fields, and
$\bw$ be defined in \eqref{3.102}. Then from system \eqref{2.2}, we get that $\mathcal{P}\bv, \Lambda^{-1}\theta$ and $\bw$ fulfill the heat equation, respectively, and
$a$ satisfies a damped transport equation as follows.
\begin{equation}\label{3.12}
\left\{
\begin{array}{ll}
\partial_t\mathcal{P}\bv-\mu\nu^{-1}\Delta\mathcal{P}\bv=\mathcal{P}{\bf g},\\[1ex]
\partial_t\Lambda^{-1}\theta-\beta\Lambda^{-1}\Delta\theta=\Lambda^{-1}{m}-\gamma\Lambda^{-1}{\rm div}\bw-\gamma \Lambda^{-1}a,\\[1ex]
\partial_t\bw-\Delta\bw=\nabla(-\Delta)^{-1}(f-{\rm div}{\bf g})+\bw-(-\Delta)^{-1}\nabla a-\gamma\nabla\theta,\\[1ex]
\partial_t a+a=f-{\rm div}\bw.
\end{array}
\right.
\end{equation}
Applying $\dot{\Delta}_k$ to \eqref{3.12}$_1$ yields for all $k\in\mathbb{Z}$,
$$
\partial_t\mathcal{P}\bv_k-\mu\nu^{-1}\Delta\mathcal{P}\bv_k=\mathcal{P}{\bf g}_k.
$$
Then, multiplying each component of the above equation by $|(\mathcal{P}v_k)^i|^{p-2}(\mathcal{P}v_k)^i$ and integrating
over $\mathbb{R}^N$ gives for $i=1,2,\cdots,N$,
\be\nonumber
\begin{split}
&\quad\frac{1}{p}\frac{d}{dt}\|\mathcal{P}v_k^i\|_{L^p}^p-\mu\nu^{-1}
\int_{\mathbb{R}^N}\Delta(\mathcal{P}v_k)^i|(\mathcal{P}v_k)^i|^{p-2}(\mathcal{P}v_k)^idx
=\int_{\mathbb{R}^N}|(\mathcal{P}v_k)^i|^{p-2}(\mathcal{P}v_k)^i\mathcal{P}g_k^idx.
\end{split}
\ee
Applying Proposition \ref{pr2.3} and summing on $i=1,2,\cdots,N$, we get for some constant $c_p$ depending only on $p$ that
\be\nonumber
\frac{1}{p}\frac{d}{dt}\|\mathcal{P}\bv_k\|_{L^p}^p+c_p\mu\nu^{-1}2^{2k}\|\mathcal{P}\bv_k\|_{L^p}^p\leq
\|\mathcal{P}{\bf g}_k\|_{L^p}\|\mathcal{P}\bu_k\|_{L^p}^{p-1}
\ee
which leads to
\be\label{3.13}
\frac{d}{dt}\|\mathcal{P}\bv_k\|_{L^p}+c_p\mu\nu^{-1}2^{2k}\|\mathcal{P}\bv_k\|_{L^p}\leq
\|\mathcal{P}{\bf g}_k\|_{L^p}.
\ee
On the other hand, from \eqref{3.12}$_2$ and \eqref{3.12}$_3$, we argue exactly as for proving \eqref{3.13} and obtain that
\be\label{3.14}
\frac{d}{dt}\|\Lambda^{-1}\theta_k\|_{L^p}+c_p\beta 2^{2k}\|\Lambda^{-1}\theta_k\|_{L^p}\leq
\|\Lambda^{-1}{m}_k\|_{L^p}+C\|\bw_k\|_{L^p}+C2^{-2k}\|\nabla a_k\|_{L^p}
\ee
and
\be\label{3.15}
\frac{d}{dt}\|\bw_k\|_{L^p}+c_p2^{2k}\|\bw_k\|_{L^p}\leq
C2^{-k}\|{f}_k\|_{L^p}+\|({\bf g}_k,\bw_k)\|_{L^p}+C2^{2k}\|\Lambda^{-1}\theta_k\|_{L^p}+C2^{-2k}\|\nabla a_k\|_{L^p}.
\ee
Since the function $a$ fulfills the damped transport equation \eqref{3.12}$_4$, then performing the operator $\partial_i\dot{\Delta}_k$ to \eqref{3.12}$_4$ and denoting $R_k^i\equ [\bv\cdot\nabla, \partial_i\dot{\Delta}_k]a$, one has
\be\label{3.16}
\partial_t\partial_ia_k+\bv\cdot\nabla\partial_ia_k+\partial_ia_k=-\partial_i\dot{\Delta}_k(a{\rm div}\bv)
-\partial_i{\rm div}\bw_k+R_k^i,\,\,\,i=1,2,\cdots,N.
\ee
Multiplying both sides of \eqref{3.16} by $|\partial_ia_k|^{p-2}\partial_ia_k$, integrating on $\mathbb{R}^N$, and performing an integration by parts in the second term, we arrive at
\be\nonumber
\begin{split}
\frac{1}{p}\frac{d}{dt}\|\partial_ia_k\|_{L^p}^p+\|\partial_ia_k\|_{L^p}^p=\frac{1}{p}&\int_{\mathbb{R}^N}
{\rm div}\bv|\partial_ia_k|^pdx\\[1ex]
&+\int_{\mathbb{R}^N}(R_k^i-\partial_i\dot{\Delta}_k(a{\rm div}\bv)
-\partial_i{\rm div}\bw_k)|\partial_ia_k|^{p-2}\partial_ia_kdx.
\end{split}
\ee
Summing up on $i=1,2,\cdots,N$ and applying H\"{o}lder and Bernstein inequalities imply
\be\label{3.17}
\begin{split}
\frac{1}{p}\frac{d}{dt}\|\nabla a_k\|_{L^p}^p+\|\nabla a_k\|_{L^p}^p\leq &\Big(\frac{1}{p}\|{\rm div}\bv\|_{L^\infty}
\|\nabla a_k\|_{L^p}+\|\nabla\dot{\Delta}_k(a{\rm div}\bv)\|_{L^p}\\[1ex]
&\quad\quad\quad\quad\quad\quad+C2^{2k}\|\bw_k\|_{L^p}+\|R_k\|_{L^p}
\Big)\|\nabla a_k\|_{L^p}^{p-1},
\end{split}
\ee
which leads to
\be\label{3.18}
\begin{split}
&\quad\frac{1}{p}\frac{d}{dt}\|\nabla a_k\|_{L^p}+\|\nabla a_k\|_{L^p}\\[1ex]
&\leq \frac{1}{p}\|{\rm div}\bv\|_{L^\infty}
\|\nabla a_k\|_{L^p}+\|\nabla\dot{\Delta}_k(a{\rm div}\bv)\|_{L^p}
+C2^{2k}\|\bw_k\|_{L^p}+\|R_k\|_{L^p}.
\end{split}
\ee
Adding \eqref{3.18} (multiplying by $\eta c_p$ for some $\eta>0$), \eqref{3.13}, \eqref{3.14}  and \eqref{3.15} (multiplying by $\alpha$ for some $\alpha>0$) together gives
\be\nonumber
\begin{split}
  &\quad\frac{d}{dt}\left(\|(\mathcal{P}\bv_k, \alpha\bw_k, \Lambda^{-1}\theta_k)\|_{L^p}+\eta c_p\|\nabla a_k\|_{L^p}\right)\\[1ex]
  &\quad\quad\quad\quad\quad\quad\quad\quad\quad
  +c_p2^{2k}(\mu\nu^{-1}\|\mathcal{P}\bv_k\|_{L^p}+\|(\alpha\bw_k, \beta\Lambda^{-1}\theta_k)\|_{L^p})+\eta c_p\|\nabla a_k\|_{L^p}\\[1ex]
  &\leq\|\mathcal{P}{\bf g}_k\|_{L^p}+\|\Lambda^{-1}{m}_k\|_{L^p}+C\|\bw_k\|_{L^p}+C2^{-2k}\|\nabla a_k\|_{L^p}\\[1ex]
  &\quad+\eta c_p\left(\frac{1}{p}\|{\rm div}\bv\|_{L^\infty}
\|\nabla a_k\|_{L^p}+\|\nabla\dot{\Delta}_k(a{\rm div}\bv)\|_{L^p}
+C2^{2k}\|\bw_k\|_{L^p}+\|R_k\|_{L^p}\right)\\[1ex]
  &\quad+C\alpha\Big(2^{-k}\|{f}_k\|_{L^p}+\|({\bf g}_k,\bw_k)\|_{L^p}+C2^{-2k}\|\nabla a_k\|_{L^p}+C2^{2k}\|\Lambda^{-1}\theta_k\|_{L^p}\Big).
\end{split}
\ee
Choosing $k_0$ suitably large and $\eta$ and $\alpha$ sufficiently small, we deduce that there exists a constant $c_0>0$ such that
for all $k\geq k_0+1$,
\be\nonumber
\begin{split}
  &\quad\frac{d}{dt}\|(\mathcal{P}\bv_k, \bw_k, \Lambda^{-1}\theta_k,\nabla a_k)\|_{L^p}+c_0\left(2^{2k}\|(\mathcal{P}\bv_k,\bw_k,\Lambda^{-1}\theta_k)\|_{L^p}
  +\|\nabla a_k\|_{L^p}\right)\\[1ex]
  &\lesssim 2^{-k}\|f_k\|_{L^p}+\|(\Lambda^{-1}m_k, {\bf g}_k)\|_{L^p}+\|{\rm div}\bv\|_{L^\infty}
\|\nabla a_k\|_{L^p}+\|\nabla\dot{\Delta}_k(a{\rm div}\bv)\|_{L^p}
+\|R_k\|_{L^p}.
\end{split}
\ee
Since
$$
\bv=\bw-\nabla (-\Delta)^{-1}a+\mathcal{P}\bv,
$$
it holds that
\be\nonumber
\begin{split}
  &\quad\frac{d}{dt}\|(\nabla a_k, \bv_k, \Lambda^{-1}\theta_k)\|_{L^p}+c_0\|(\nabla a_k, 2^{2k}\bv_k, 2^{2k}\Lambda^{-1}\theta_k)\|_{L^p}\\[1ex]
  &\lesssim \|(2^{-k}f_k, 2^{-k}m_k, {\bf g}_k)\|_{L^p}+\|{\rm div}\bv\|_{L^\infty}
\|\nabla a_k\|_{L^p}+\|\nabla\dot{\Delta}_k(a{\rm div}\bv)\|_{L^p}
+\|R_k\|_{L^p}.
\end{split}
\ee
Thus, multiplying by $2^{k(\frac{N}{p}-1)}$, summing up over $k\geq k_0+1$ and using Corollary \ref{co2.1}
and Proposition \ref{pra.5}, we conclude \eqref{3.11}.
\end{proof}

\section{Estimation of $L^2$-type Besov norms at low frequencies}\label{s:5}
This section is devoted to bounding the $L^2$-type Besov norms at low frequencies, which is the main ingredient in the proof of Theorem \ref{th2}.
\begin{lemma}\label{le3}
Let $2-\frac{N}{2}<\sigma_1\leq \frac{2N}{p}-\frac{N}{2}$ and $p$ satisfy \eqref{1.5}. Then the solution $(a, \bv,\theta)$ to system \eqref{2.2} satisfies
\begin{equation}\label{4.11}
\begin{split}
&\Big(\|(a, \bv,\theta)(t)\|_{\dot{B}_{2,\infty}^{-\sigma_1}}^\ell\Big)^2\lesssim
\Big(\|(a_0, \bv_0,\theta_0)\|_{\dot{B}_{2,\infty}^{-\sigma_1}}^\ell\Big)^2\\[1ex]
&\quad\quad\quad\quad\quad+\int_0^tA_1(\tau)\Big(\|(a, \bv,\theta)(\tau)\|_{\dot{B}_{2,\infty}^{-\sigma_1}}^\ell\Big)^2 d\tau
+\int_0^tA_2(\tau)\|(a, \bv,\theta)(\tau)\|_{\dot{B}_{2,\infty}^{-\sigma_1}}^\ell d\tau,
\end{split}
\end{equation}
where
{\begin{equation}\nonumber
\begin{split}
A_1(t)&\equ\|(a,\bv,\theta)\|_{\dot{B}_{2,1}^{\frac{N}{2}+1}}^\ell+\|a\|_{\dot{B}_{p,1}^{\frac{N}{p}}}^h
+\|\bv\|_{\dot{B}_{p,1}^{\frac{N}{p}+1}}^h+\|a\|_{\dot{B}_{p,1}^{\frac{N}{p}}}^2+\|\theta\|_{\dot{B}_{p,1}^{\frac{N}{p}}}^h\\[1ex]
&\quad+\|a\|_{\dot{B}_{p,1}^{\frac{N}{p}}}
\Big(\|\bv\|_{\dot{B}_{2,1}^{\frac{N}{2}+1}}^\ell+\|\bv\|_{\dot{B}_{p,1}^{\frac{N}{p}+1}}^h
+\|\theta\|_{\dot{B}_{p,1}^{\frac{N}{p}}}^h\Big)+\|a\|_{\dot{B}_{p,1}^{\frac{N}{p}}}
\Big(\|\theta\|_{\dot{B}_{2,1}^{\frac{N}{2}}}^\ell+\|\bv\|_{\dot{B}_{p,1}^{\frac{N}{p}}}^h\Big)
\end{split}
\end{equation}}
and
{\begin{equation}\nonumber
\begin{split}
A_2(t)&\equ\Big(\|(a,\bv)\|_{\dot{B}_{p,1}^{\frac{N}{p}}}^h\Big)^2
+\|\bv\|_{\dot{B}_{p,1}^{\frac{N}{p}+1}}^h\|a\|_{\dot{B}_{p,1}^{\frac{N}{p}}}\|a\|_{\dot{B}_{p,1}^{\frac{N}{p}}}^h
+\|a\|_{\dot{B}_{p,1}^{\frac{N}{p}}}^2\|a\|_{\dot{B}_{p,1}^{\frac{N}{p}}}^h
+\|a\|_{\dot{B}_{p,1}^{\frac{N}{p}}}^h\|\bv\|_{\dot{B}_{p,1}^{\frac{N}{p}+1}}^h\\[1ex]
&\quad+\|\theta\|_{\dot{B}_{p,1}^{\frac{N}{p}}}^h\|a\|_{\dot{B}_{p,1}^{\frac{N}{p}}}^h
+\|\theta\|_{\dot{B}_{p,1}^{\frac{N}{p}}}^h\|a\|_{\dot{B}_{p,1}^{\frac{N}{p}}}^2+\|\theta\|_{\dot{B}_{p,1}^{\frac{N}{p}}}^h
\|\bv\|_{\dot{B}_{p,1}^{\frac{N}{p}-1}}^h\\[1ex]
&\quad+\Big(\|\bv\|_{\dot{B}_{p,1}^{\frac{N}{p}}}^h\Big)^2\|a\|_{\dot{B}_{p,1}^{\frac{N}{p}}}+\|\theta\|_{\dot{B}_{p,1}^{\frac{N}{p}-2}}^h\|\bv\|_{\dot{B}_{p,1}^{\frac{N}{p}+1}}^h
+\|\bv\|_{\dot{B}_{p,1}^{\frac{N}{p}}}^h\|\theta\|_{\dot{B}_{p,1}^{\frac{N}{p}-1}}^h\|a\|_{\dot{B}_{p,1}^{\frac{N}{p}}}.
\end{split}
\end{equation}}
\end{lemma}
\begin{proof}
Recalling  \eqref{3.103} and \eqref{3.9}, we get for $k\leq k_0$ that
 \begin{equation}\label{low-frequency}
 \begin{split}
 &\quad\frac{1}{2}\frac{d}{dt}\|(a_k, \omega_k, {\bf \Omega}_k, \theta_k)\|_{L^2}^2+\|(\Lambda a_k,\Lambda \omega_k, \Lambda {\bf \Omega}_k, \Lambda \theta_k)\|_{L^2}^2\\[1ex]
 &\lesssim \|(f_k, \Lambda^{-1}{\rm div}{\bf g}_k, \Lambda^{-1}{\rm curl}{\bf g}_k, m_k)\|_{L^2}\|(a_k, \omega_k, {\bf \Omega}_k, \theta_k)\|_{L^2},
 \end{split}
 \end{equation}
which implies that
\begin{equation}\label{4.12}
 \begin{split}
 &\quad\|(a_k, \omega_k, {\bf \Omega}_k, \theta_k)\|_{L^2}^2+\int_0^t\|(\Lambda a_k,\Lambda \omega_k, \Lambda {\bf \Omega}_k, \Lambda \theta_k)\|_{L^2}^2ds\\[1ex]
 &\lesssim \|(a_{0k}, \omega_{0k}, {\bf \Omega}_{0k}, \theta_{0k})\|_{L^2}^2+\int_0^t\|(f_k, \Lambda^{-1}{\rm div}{\bf g}_k, \Lambda^{-1}{\rm curl}{\bf g}_k, m_k)\|_{L^2}\|(a_k, \omega_k, {\bf \Omega}_k, \theta_k)\|_{L^2}ds.
 \end{split}
 \end{equation}
Multiplying $2^{2k(-\sigma_1)}$ on both sides of \eqref{4.12}, taking supremum in terms of $k\leq k_0$ and noticing that \eqref{3.1}$_5$, we have
\begin{equation}\label{4.13}
\begin{split}
&\quad\Big(\|(a, \bv,\theta)(t)\|_{\dot{B}_{2,\infty}^{-\sigma_1}}^\ell\Big)^2\\[1ex]
&\lesssim
\Big(\|(a_0, \bv_0,\theta_0)\|_{\dot{B}_{2,\infty}^{-\sigma_1}}^\ell\Big)^2+\int_0^t\|(f, {\bf g},{ m})(\tau)\|_{\dot{B}_{2,\infty}^{-\sigma_1}}^\ell\|(a, \bv,\theta)(\tau)\|_{\dot{B}_{2,\infty}^{-\sigma_1}}^\ell  d\tau.
\end{split}
\end{equation}

In what follows, we focus on the estimates of nonlinear norm $\|(f, {\bf g}, m)\|_{\dot{B}_{2,\infty}^{-\sigma_1}}^\ell$. Firstly, we deal with the term $f=-{\rm div}(a\bv)=-a{\rm div}\bv-\bv\cdot\nabla a$.

\underline{Estimate of $a{\rm div}\bv$}. Decomposing
$a{\rm div}\bv=a^\ell{\rm div}\bv+a^h{\rm div}\bv^\ell+a^h{\rm div}\bv^h$ and
making use of  \eqref{4.1}, we infer that
\be\label{4.14}
\|a^\ell{\rm div}\bv\|_{\dot{B}_{2,\infty}^{-\sigma_1}}\lesssim \|{\rm div}\bv\|_{\dot{B}_{p,1}^{\frac{N}{p}}}
\|a\|_{\dot{B}_{2,\infty}^{-\sigma_1}}^\ell\lesssim \Big(\|\bv\|_{\dot{B}_{2,1}^{\frac{N}{2}+1}}^\ell+\|\bv\|_{\dot{B}_{p,1}^{\frac{N}{p}+1}}^h\Big)
\|a\|_{\dot{B}_{2,\infty}^{-\sigma_1}}^\ell
\ee
and
\be\label{4.15}
\|a^h{\rm div}\bv^\ell\|_{\dot{B}_{2,\infty}^{-\sigma_1}}\lesssim \|a^h\|_{\dot{B}_{p,1}^{\frac{N}{p}}}
\|{\rm div}\bv^\ell\|_{\dot{B}_{2,\infty}^{-\sigma_1}}\lesssim \|a\|_{\dot{B}_{p,1}^{\frac{N}{p}}}^h
\|\bv\|_{\dot{B}_{2,\infty}^{-\sigma_1}}^\ell.
\ee
By means of \eqref{4.2}, one gets
\be\label{4.16}
\begin{split}
\|a^h{\rm div}\bv^h\|_{\dot{B}_{2,\infty}^{-\sigma_1}}^\ell&\lesssim \|a^h\|_{\dot{B}_{p,1}^{\frac{N}{p}-1}}
\Big(\|{\rm div}\bv^h\|_{\dot{B}_{p,\infty}^{-\sigma_1+\frac{N}{p}-\frac{N}{2}+1}}+\|{\rm div}\bv^h\|_{\dot{B}_{p,\infty}^{-\sigma_1+\frac{2N}{p}-N+1}}\Big)\\[1ex]
&\lesssim\|a\|_{\dot{B}_{p,1}^{\frac{N}{p}}}^h\|\bv\|_{\dot{B}_{p,1}^{\frac{N}{p}+1}}^h,
\end{split}
\ee
where we used that
$
-\sigma_1+\frac{2N}{p}-N+2\leq -\sigma_1+\frac{N}{p}-\frac{N}{2}+2<\frac{N}{p}+1
$
since $\sigma_1>1-\frac{N}{2}$ and $p\geq 2$.

\underline{Estimate of $\bv\cdot\nabla a$}. Decomposing $\bv\cdot\nabla a=\bv^\ell\cdot\nabla a^\ell+\bv^h\cdot\nabla a^\ell+\bv^\ell\cdot\nabla a^h+\bv^h\cdot\nabla a^h$, we deduce from \eqref{4.1} that
\be\label{4.17}
\|\bv^\ell\nabla a^\ell\|_{\dot{B}_{2,\infty}^{-\sigma_1}}\lesssim\|\nabla a^\ell\|_{\dot{B}_{p,1}^{\frac{N}{p}}}
\|\bv^\ell\|_{\dot{B}_{2,\infty}^{-\sigma_1}}\lesssim\|a\|_{\dot{B}_{2,1}^{\frac{N}{2}+1}}^\ell
\|\bv\|_{\dot{B}_{2,\infty}^{-\sigma_1}}^\ell,
\ee
and
\be\label{4.18}
\|\bv^h\nabla a^\ell\|_{\dot{B}_{2,\infty}^{-\sigma_1}}\lesssim\|\bv^h\|_{\dot{B}_{p,1}^{\frac{N}{p}}}
\|\nabla a^\ell\|_{\dot{B}_{2,\infty}^{-\sigma_1}}\lesssim\|\bv\|_{\dot{B}_{p,1}^{\frac{N}{p}+1}}^h
\|a\|_{\dot{B}_{2,\infty}^{-\sigma_1}}^\ell.
\ee
It follows from \eqref{4.2} that
\be\label{4.19}
\begin{split}
\|\bv^\ell\nabla a^h\|_{\dot{B}_{2,\infty}^{-\sigma_1}}^\ell&\lesssim\|\nabla a^h\|_{\dot{B}_{p,1}^{\frac{N}{p}-1}}
\Big(\|\bv^\ell\|_{\dot{B}_{p,\infty}^{-\sigma_1+\frac{N}{p}-\frac{N}{2}+1}}
+\|\bv^\ell\|_{\dot{B}_{p,\infty}^{-\sigma_1+\frac{2N}{p}-N+1}}\Big)\\[1ex]
&\lesssim\|a^h\|_{\dot{B}_{p,1}^{\frac{N}{p}}}\|\bv^\ell\|_{\dot{B}_{p,\infty}^{-\sigma_1+\frac{2N}{p}-N+1}}
\lesssim\|a\|_{\dot{B}_{p,1}^{\frac{N}{p}}}^h\|\bv\|_{\dot{B}_{2,\infty}^{-\sigma_1}}^\ell,
\end{split}
\ee
where we used that $-\sigma_1+\frac{2N}{p}-N+1\leq -\sigma_1+\frac{N}{p}-\frac{N}{2}+1$ in the second inequality
and that ${\dot{B}_{2,\infty}^{-\sigma_1}}
\hookrightarrow{\dot{B}_{p,\infty}^{-\sigma_1+\frac{2N}{p}-N+1}}$ at the low frequency in the last inequality when $2\leq p\leq \frac{2N}{N-2}$. For the term $\bv^h\nabla a^h$, by \eqref{4.2} again, it holds that
\be\label{4.20}
\begin{split}
\|\bv^h\nabla a^h\|_{\dot{B}_{2,\infty}^{-\sigma_1}}^\ell&\lesssim\|\nabla a^h\|_{\dot{B}_{p,1}^{\frac{N}{p}-1}}
\Big(\|\bv^h\|_{\dot{B}_{p,\infty}^{-\sigma_1+\frac{N}{p}-\frac{N}{2}+1}}
+\|\bv^h\|_{\dot{B}_{p,\infty}^{-\sigma_1+\frac{2N}{p}-N+1}}\Big)\\[1ex]
&\lesssim\|a^h\|_{\dot{B}_{p,1}^{\frac{N}{p}}}\|\bv^h\|_{\dot{B}_{p,\infty}^{-\sigma_1+\frac{N}{p}-\frac{N}{2}+1}}
\lesssim\|a\|_{\dot{B}_{p,1}^{\frac{N}{p}}}^h\|\bv\|_{\dot{B}_{p,1}^{\frac{N}{p}+1}}^h,
\end{split}
\ee
where we have used that
$
-\sigma_1+\frac{2N}{p}-N+1\leq -\sigma_1+\frac{N}{p}-\frac{N}{2}+1\leq \frac{N}{p}+1,
$
since $\sigma_1\geq-\frac{N}{2}$ and $p\geq 2$.

Now, we are in a position to estimate $\|{\bf g}\|_{\dot{B}_{2,\infty}^{-\sigma_1}}^\ell$. Let us recall that
\be\nonumber
{\bf g}\equ -\bv\cdot\nabla\bv-I(a)\widetilde{\mathcal{A}}\bv-K_1(a)\nabla a-K_2(a)\nabla\theta-\theta\nabla K_3(a).
\ee

\underline{Estimate of $\bv\cdot\nabla\bv$}. Decompose $\bv\cdot\nabla\bv=\bv^\ell\cdot\nabla\bv^\ell+\bv^\ell\cdot\nabla\bv^h
+\bv^h\cdot\nabla\bv^\ell+\bv^h\cdot\nabla\bv^h$.
It holds from \eqref{4.1} that
\be\label{4.23}
\|\bv^\ell\cdot\nabla \bv^\ell\|_{\dot{B}_{2,\infty}^{-\sigma_1}}\lesssim\|\nabla \bv^\ell\|_{\dot{B}_{p,1}^{\frac{N}{p}}}
\|\bv^\ell\|_{\dot{B}_{2,\infty}^{-\sigma_1}}\lesssim\|\bv\|_{\dot{B}_{2,1}^{\frac{N}{2}+1}}^\ell
\|\bv\|_{\dot{B}_{2,\infty}^{-\sigma_1}}^\ell,
\ee
\be\label{4.24}
\|\bv^h\cdot\nabla \bv^\ell\|_{\dot{B}_{2,\infty}^{-\sigma_1}}\lesssim\|\bv^h\|_{\dot{B}_{p,1}^{\frac{N}{p}}}
\|\nabla\bv^\ell\|_{\dot{B}_{2,\infty}^{-\sigma_1}}\lesssim\|\bv\|_{\dot{B}_{p,1}^{\frac{N}{p}+1}}^h
\|\bv\|_{\dot{B}_{2,\infty}^{-\sigma_1}}^\ell.
\ee
By similar calculations to \eqref{4.19} and \eqref{4.20}, one has by \eqref{4.2} that
\be\label{4.124}
\begin{split}
\|\bv^\ell\cdot\nabla \bv^h\|_{\dot{B}_{2,\infty}^{-\sigma_1}}^\ell&\lesssim\|\nabla \bv^h\|_{\dot{B}_{p,1}^{\frac{N}{p}-1}}
\Big(\|\bv^\ell\|_{\dot{B}_{p,\infty}^{-\sigma_1+\frac{N}{p}-\frac{N}{2}+1}}
+\|\bv^\ell\|_{\dot{B}_{p,\infty}^{-\sigma_1+\frac{2N}{p}-N+1}}\Big)\\[1ex]
&\lesssim\|\bv^h\|_{\dot{B}_{p,1}^{\frac{N}{p}}}\|\bv^\ell\|_{\dot{B}_{p,\infty}^{-\sigma_1+\frac{2N}{p}-N+1}}
\lesssim\|\bv\|_{\dot{B}_{p,1}^{\frac{N}{p}+1}}^h\|\bv\|_{\dot{B}_{2,\infty}^{-\sigma_1}}^\ell,
\end{split}
\ee
and
\be\label{4.25}
\begin{split}
\|\bv^h\cdot\nabla \bv^h\|_{\dot{B}_{2,\infty}^{-\sigma_1}}^\ell&\lesssim\|\nabla \bv^h\|_{\dot{B}_{p,1}^{\frac{N}{p}-1}}
\Big(\|\bv^h\|_{\dot{B}_{p,\infty}^{-\sigma_1+\frac{N}{p}-\frac{N}{2}+1}}
+\|\bv^h\|_{\dot{B}_{p,\infty}^{-\sigma_1+\frac{2N}{p}-N+1}}\Big)\\[1ex]
&\lesssim\|\bv^h\|_{\dot{B}_{p,1}^{\frac{N}{p}}}\|\bv^h\|_{\dot{B}_{p,\infty}^{-\sigma_1+\frac{N}{p}-\frac{N}{2}+1}}
\lesssim\|\bv\|_{\dot{B}_{p,1}^{\frac{N}{p}}}^h\|\bv\|_{\dot{B}_{p,1}^{\frac{N}{p}}}^h.
\end{split}
\ee

\underline{Estimate of $I(a)\widetilde{\mathcal{A}}\bv$}. Keeping in mind that $I(0)=0$, one may write
$$
I(a)=I^\prime(0)a+\bar{I}(a)a
$$
for some smooth function $\bar{I}$ vanishing at $0$. Thus, through \eqref{4.1} again, we have
\be\label{4.26}
\|a^\ell\widetilde{\mathcal{A}}\bv^\ell\|_{\dot{B}_{2,\infty}^{-\sigma_1}}\lesssim\|\widetilde{\mathcal{A}}\bv^\ell
\|_{\dot{B}_{p,1}^{\frac{N}{p}}}
\|a^\ell\|_{\dot{B}_{2,\infty}^{-\sigma_1}}\lesssim\|\bv\|_{\dot{B}_{2,1}^{\frac{N}{2}+1}}^\ell
\|a\|_{\dot{B}_{2,\infty}^{-\sigma_1}}^\ell,
\ee
and
\be\label{4.27}
\|a^h\widetilde{\mathcal{A}}\bv^\ell\|_{\dot{B}_{2,\infty}^{-\sigma_1}}\lesssim\|a^h\|_{\dot{B}_{p,1}^{\frac{N}{p}}}
\|\widetilde{\mathcal{A}}\bv^\ell\|_{\dot{B}_{2,\infty}^{-\sigma_1}}\lesssim\|a\|_{\dot{B}_{p,1}^{\frac{N}{p}}}^h
\|\bv\|_{\dot{B}_{2,\infty}^{-\sigma_1}}^\ell.
\ee
Arguing similarly as \eqref{4.19} and \eqref{4.20}, one has
\be\label{4.28}
\begin{split}
\|a^\ell\widetilde{\mathcal{A}} \bv^h\|_{\dot{B}_{2,\infty}^{-\sigma_1}}^\ell&\lesssim\|\widetilde{\mathcal{A}} \bv^h\|_{\dot{B}_{p,1}^{\frac{N}{p}-1}}
\Big(\|a^\ell\|_{\dot{B}_{p,\infty}^{-\sigma_1+\frac{N}{p}-\frac{N}{2}+1}}
+\|a^\ell\|_{\dot{B}_{p,\infty}^{-\sigma_1+\frac{2N}{p}-N+1}}\Big)\\[1ex]
&\lesssim\|\bv\|_{\dot{B}_{p,1}^{\frac{N}{p}+1}}^h\|a^\ell\|_{\dot{B}_{p,\infty}^{-\sigma_1+\frac{2N}{p}-N+1}}
\lesssim\|\bv\|_{\dot{B}_{p,1}^{\frac{N}{p}+1}}^h\|a\|_{\dot{B}_{2,\infty}^{-\sigma_1}}^\ell,
\end{split}
\ee
and
\be\label{4.29}
\begin{split}
\|a^h\widetilde{\mathcal{A}}\bv^h\|_{\dot{B}_{2,\infty}^{-\sigma_1}}^\ell&\lesssim\|\mathcal{A}\bv^h\|_{\dot{B}_{p,1}^{\frac{N}{p}-1}}
\Big(\|a^h\|_{\dot{B}_{p,\infty}^{-\sigma_1+\frac{N}{p}-\frac{N}{2}+1}}
+\|a^h\|_{\dot{B}_{p,\infty}^{-\sigma_1+\frac{2N}{p}-N+1}}\Big)\\[1ex]
&\lesssim\|\bv^h\|_{\dot{B}_{p,1}^{\frac{N}{p}+1}}\|a^h\|_{\dot{B}_{p,\infty}^{-\sigma_1+\frac{N}{p}-\frac{N}{2}+1}}
\lesssim\|\bv\|_{\dot{B}_{p,1}^{\frac{N}{p}+1}}^h\|a\|_{\dot{B}_{p,1}^{\frac{N}{p}}}^h.
\end{split}
\ee
On the other hand, from \eqref{4.1}, \eqref{4.2}, Proposition \ref{pra.4} and Corollaries \ref{co2.1} and \ref{co2.2}, we have
\be\label{4.30}
\|\bar{I}(a)a\widetilde{\mathcal{A}}\bv^\ell\|_{\dot{B}_{2,\infty}^{-\sigma_1}}
\lesssim\|\bar{I}(a)a\|_{\dot{B}_{p,1}^{\frac{N}{p}}}\|\widetilde{\mathcal{A}}\bv^\ell\|_{\dot{B}_{2,\infty}^{-\sigma_1}}
\lesssim\|a\|_{\dot{B}_{p,1}^{\frac{N}{p}}}^2\|\bv\|_{\dot{B}_{2,\infty}^{-\sigma_1}}^\ell,
\ee
and
\be\label{4.31}
\begin{split}
\|\bar{I}(a)a\widetilde{\mathcal{A}}\bv^h\|_{\dot{B}_{2,\infty}^{-\sigma_1}}^\ell
&\lesssim\|\widetilde{\mathcal{A}}\bv^h\|_{\dot{B}_{p,1}^{\frac{N}{p}-1}}
\Big(\|\bar{I}(a)a\|_{\dot{B}_{p,\infty}^{-\sigma_1+\frac{N}{p}-\frac{N}{2}+1}}
+\|\bar{I}(a)a\|_{\dot{B}_{p,\infty}^{-\sigma_1+\frac{2N}{p}-N+1}}\Big)\\[1ex]
&\lesssim\|\bv^h\|_{\dot{B}_{p,1}^{\frac{N}{p}+1}}\|\bar{I}(a)\|_{\dot{B}_{p,1}^{\frac{N}{p}}}
\Big(\|a\|_{\dot{B}_{p,\infty}^{-\sigma_1+\frac{N}{p}-\frac{N}{2}+1}}
+\|a\|_{\dot{B}_{p,\infty}^{-\sigma_1+\frac{2N}{p}-N+1}}\Big)\\[1ex]
&\lesssim\|\bv\|_{\dot{B}_{p,1}^{\frac{N}{p}+1}}^h
\|a\|_{\dot{B}_{p,1}^{\frac{N}{p}}}\Big(\|a\|_{\dot{B}_{p,\infty}^{-\sigma_1+\frac{N}{p}-\frac{N}{2}+1}}^h
+\|a\|_{\dot{B}_{p,\infty}^{-\sigma_1+\frac{2N}{p}-N+1}}^\ell\Big)\\[1ex]
&\lesssim\|\bv\|_{\dot{B}_{p,1}^{\frac{N}{p}+1}}^h
\|a\|_{\dot{B}_{p,1}^{\frac{N}{p}}}\Big(\|a\|_{\dot{B}_{p,1}^{\frac{N}{p}}}^h
+\|a\|_{\dot{B}_{2,\infty}^{-\sigma_1}}^\ell\Big).
\end{split}
\ee

\underline{Estimate of $K_1(a)\nabla a$}. In view of $K_1(0)=0$, we may write $K_1(a)=K_1^\prime(0)a+\bar{K}_1(a)a$, here $\bar{K}_1$ is a smooth function fulfilling $\bar{K}_1(0)=0$. For the term $a\nabla a$, we obtain
\be\label{4.32}
\|a^\ell\nabla a^\ell\|_{\dot{B}_{2,\infty}^{-\sigma_1}}\lesssim\|\nabla a^\ell\|_{\dot{B}_{p,1}^{\frac{N}{p}}}
\|a^\ell\|_{\dot{B}_{2,\infty}^{-\sigma_1}}\lesssim\|a\|_{\dot{B}_{2,1}^{\frac{N}{2}+1}}^\ell
\|a\|_{\dot{B}_{2,\infty}^{-\sigma_1}}^\ell,
\ee
and
\be\label{4.33}
\|a^h\nabla a^\ell\|_{\dot{B}_{2,\infty}^{-\sigma_1}}\lesssim\|a^h\|_{\dot{B}_{p,1}^{\frac{N}{p}}}
\|\nabla a^\ell\|_{\dot{B}_{2,\infty}^{-\sigma_1}}\lesssim\|a\|_{\dot{B}_{p,1}^{\frac{N}{p}}}^h
\|a\|_{\dot{B}_{2,\infty}^{-\sigma_1}}^\ell.
\ee
Similar to \eqref{4.28} and \eqref{4.29}, one has
\be\label{4.34}
\begin{split}
\|a^\ell\nabla a^h\|_{\dot{B}_{2,\infty}^{-\sigma_1}}^\ell&\lesssim\|\nabla a^h\|_{\dot{B}_{p,1}^{\frac{N}{p}-1}}
\Big(\|a^\ell\|_{\dot{B}_{p,\infty}^{-\sigma_1+\frac{N}{p}-\frac{N}{2}+1}}
+\|a^\ell\|_{\dot{B}_{p,\infty}^{-\sigma_1+\frac{2N}{p}-N+1}}\Big)\\[1ex]
&\lesssim\|a\|_{\dot{B}_{p,1}^{\frac{N}{p}}}^h\|a^\ell\|_{\dot{B}_{p,\infty}^{-\sigma_1+\frac{2N}{p}-N+1}}
\lesssim\|a\|_{\dot{B}_{p,1}^{\frac{N}{p}}}^h\|a\|_{\dot{B}_{2,\infty}^{-\sigma_1}}^\ell,
\end{split}
\ee
and
\be\label{4.35}
\begin{split}
\|a^h\nabla a^h\|_{\dot{B}_{2,\infty}^{-\sigma_1}}^\ell&\lesssim\|\nabla a^h\|_{\dot{B}_{p,1}^{\frac{N}{p}-1}}
\Big(\|a^h\|_{\dot{B}_{p,\infty}^{-\sigma_1+\frac{N}{p}-\frac{N}{2}+1}}
+\|a^h\|_{\dot{B}_{p,\infty}^{-\sigma_1+\frac{2N}{p}-N+1}}\Big)\\[1ex]
&\lesssim\|a^h\|_{\dot{B}_{p,1}^{\frac{N}{p}}}\|a^h\|_{\dot{B}_{p,\infty}^{-\sigma_1+\frac{N}{p}-\frac{N}{2}+1}}
\lesssim\|a\|_{\dot{B}_{p,1}^{\frac{N}{p}}}^h\|a\|_{\dot{B}_{p,1}^{\frac{N}{p}}}^h.
\end{split}
\ee
As for the term $\bar{K}_1(a)a\nabla a$, we use the decomposition $\bar{K}_1(a)a\nabla a=\bar{K}_1(a)a\nabla a^\ell+\bar{K}_1(a)a\nabla a^h$ and get from \eqref{4.1}-\eqref{4.2}, Corollary \ref{co2.2} and Proposition \ref{pra.4} again that
\be\label{4.36}
\begin{split}
\|\bar{K}_1(a)a\nabla a^\ell\|_{\dot{B}_{2,\infty}^{-\sigma_1}}\lesssim\|\bar{K}_1(a)a\|_{\dot{B}_{p,1}^{\frac{N}{p}}}
\|\nabla a\|_{\dot{B}_{2,\infty}^{-\sigma_1}}^\ell\lesssim\|a\|_{\dot{B}_{p,1}^{\frac{N}{p}}}^2
\|a\|_{\dot{B}_{2,\infty}^{-\sigma_1}}^\ell,
\end{split}
\ee
and
\be\label{4.37}
\begin{split}
\|\bar{K}_1(a)a\nabla a^h\|_{\dot{B}_{2,\infty}^{-\sigma_1}}^\ell&\lesssim\|\nabla a^h\|_{\dot{B}_{p,1}^{\frac{N}{p}-1}}
\Big(\|\bar{K}_1(a)a\|_{\dot{B}_{p,\infty}^{-\sigma_1+\frac{N}{p}-\frac{N}{2}+1}}
+\|\bar{K}_1(a)a\|_{\dot{B}_{p,\infty}^{-\sigma_1+\frac{2N}{p}-N+1}}\Big)\\[1ex]
&\lesssim\|a^h\|_{\dot{B}_{p,1}^{\frac{N}{p}}}\|\bar{K}_1(a)\|_{\dot{B}_{p,1}^{\frac{N}{p}}}
\Big(\|a\|_{\dot{B}_{p,\infty}^{-\sigma_1+\frac{N}{p}-\frac{N}{2}+1}}
+\|a\|_{\dot{B}_{p,\infty}^{-\sigma_1+\frac{2N}{p}-N+1}}\Big)\\[1ex]
&\lesssim\|a\|_{\dot{B}_{p,1}^{\frac{N}{p}}}^2\Big(\|a\|_{\dot{B}_{2,\infty}^{-\sigma_1}}^\ell
+\|a\|_{\dot{B}_{p,1}^{\frac{N}{p}}}^h\Big).
\end{split}
\ee

\underline{Estimate of $K_2(a)\nabla\theta$}.  Similarly, we rewrite $K_2(a)=K_2^\prime(0)a+\bar{K}_2(a)a$, here $\bar{K}_2$ is a smooth function fulfilling $\bar{K}_2(0)=0$. For the term $a\nabla \theta$, we infer that
\be\label{4.38}
\|a^\ell\nabla \theta^\ell\|_{\dot{B}_{2,\infty}^{-\sigma_1}}\lesssim\|\nabla \theta^\ell\|_{\dot{B}_{p,1}^{\frac{N}{p}}}
\|a^\ell\|_{\dot{B}_{2,\infty}^{-\sigma_1}}\lesssim\|\theta\|_{\dot{B}_{2,1}^{\frac{N}{2}+1}}^\ell
\|a\|_{\dot{B}_{2,\infty}^{-\sigma_1}}^\ell,
\ee
and
\be\label{4.39}
\|a^h\nabla \theta^\ell\|_{\dot{B}_{2,\infty}^{-\sigma_1}}\lesssim\|a^h\|_{\dot{B}_{p,1}^{\frac{N}{p}}}
\|\nabla \theta^\ell\|_{\dot{B}_{2,\infty}^{-\sigma_1}}\lesssim\|a\|_{\dot{B}_{p,1}^{\frac{N}{p}}}^h
\|\theta\|_{\dot{B}_{2,\infty}^{-\sigma_1}}^\ell.
\ee
Arguing similarly as \eqref{4.28} and \eqref{4.29}, one has
\be\label{4.40}
\begin{split}
\|a^\ell\nabla \theta^h\|_{\dot{B}_{2,\infty}^{-\sigma_1}}^\ell&\lesssim\|\nabla \theta^h\|_{\dot{B}_{p,1}^{\frac{N}{p}-1}}
\Big(\|a^\ell\|_{\dot{B}_{p,\infty}^{-\sigma_1+\frac{N}{p}-\frac{N}{2}+1}}
+\|a^\ell\|_{\dot{B}_{p,\infty}^{-\sigma_1+\frac{2N}{p}-N+1}}\Big)\\[1ex]
&\lesssim\|\theta\|_{\dot{B}_{p,1}^{\frac{N}{p}}}^h\|a^\ell\|_{\dot{B}_{p,\infty}^{-\sigma_1+\frac{2N}{p}-N+1}}
\lesssim\|\theta\|_{\dot{B}_{p,1}^{\frac{N}{p}}}^h\|a\|_{\dot{B}_{2,\infty}^{-\sigma_1}}^\ell,
\end{split}
\ee
and
\be\label{4.41}
\begin{split}
\|a^h\nabla \theta^h\|_{\dot{B}_{2,\infty}^{-\sigma_1}}^\ell&\lesssim\|\nabla \theta^h\|_{\dot{B}_{p,1}^{\frac{N}{p}-1}}
\Big(\|a^h\|_{\dot{B}_{p,\infty}^{-\sigma_1+\frac{N}{p}-\frac{N}{2}+1}}
+\|a^h\|_{\dot{B}_{p,\infty}^{-\sigma_1+\frac{2N}{p}-N+1}}\Big)\\[1ex]
&\lesssim\|\theta^h\|_{\dot{B}_{p,1}^{\frac{N}{p}}}\|a^h\|_{\dot{B}_{p,\infty}^{-\sigma_1+\frac{N}{p}-\frac{N}{2}+1}}
\lesssim\|\theta\|_{\dot{B}_{p,1}^{\frac{N}{p}}}^h\|a\|_{\dot{B}_{p,1}^{\frac{N}{p}}}^h.
\end{split}
\ee
Regarding  for the term $\bar{K}_2(a)a\nabla \theta$, we have
\be\label{4.42}
\begin{split}
\|\bar{K}_2(a)a\nabla \theta^\ell\|_{\dot{B}_{2,\infty}^{-\sigma_1}}\lesssim\|\bar{K}_2(a)a\|_{\dot{B}_{p,1}^{\frac{N}{p}}}
\|\nabla \theta\|_{\dot{B}_{2,\infty}^{-\sigma_1}}^\ell\lesssim\|a\|_{\dot{B}_{p,1}^{\frac{N}{p}}}^2
\|\theta\|_{\dot{B}_{2,\infty}^{-\sigma_1}}^\ell,
\end{split}
\ee
and
\be\label{4.43}
\begin{split}
\|\bar{K}_2(a)a\nabla \theta^h\|_{\dot{B}_{2,\infty}^{-\sigma_1}}^\ell&\lesssim\|\nabla \theta^h\|_{\dot{B}_{p,1}^{\frac{N}{p}-1}}
\Big(\|\bar{K}_2(a)a\|_{\dot{B}_{p,\infty}^{-\sigma_1+\frac{N}{p}-\frac{N}{2}+1}}
+\|\bar{K}_2(a)a\|_{\dot{B}_{p,\infty}^{-\sigma_1+\frac{2N}{p}-N+1}}\Big)\\[1ex]
&\lesssim\|\theta^h\|_{\dot{B}_{p,1}^{\frac{N}{p}}}\|\bar{K}_2(a)\|_{\dot{B}_{p,1}^{\frac{N}{p}}}
\Big(\|a\|_{\dot{B}_{p,\infty}^{-\sigma_1+\frac{N}{p}-\frac{N}{2}+1}}
+\|a\|_{\dot{B}_{p,\infty}^{-\sigma_1+\frac{2N}{p}-N+1}}\Big)\\[1ex]
&\lesssim\|\theta\|_{\dot{B}_{p,1}^{\frac{N}{p}}}^h
\|a\|_{\dot{B}_{p,1}^{\frac{N}{p}}}\Big(\|a\|_{\dot{B}_{2,\infty}^{-\sigma_1}}^\ell
+\|a\|_{\dot{B}_{p,1}^{\frac{N}{p}}}^h\Big).
\end{split}
\ee

\underline{Estimate of $\theta\nabla{K_3(a)}$}. Decomposing $K_3(a)=K_3^\prime(0)a+\bar{K}_3(a)a$ implies
$\nabla K_3(a)=K_3^\prime(0)\nabla a+\nabla (\bar{K_3}(a)a)$. Then we have from \eqref{4.1} and \eqref{4.2} that
\be\label{4.44}
\|\nabla a^\ell \theta^\ell\|_{\dot{B}_{2,\infty}^{-\sigma_1}}\lesssim\|\nabla a^\ell\|_{\dot{B}_{p,1}^{\frac{N}{p}}}
\|\theta^\ell\|_{\dot{B}_{2,\infty}^{-\sigma_1}}\lesssim\|a\|_{\dot{B}_{2,1}^{\frac{N}{2}+1}}^\ell
\|\theta\|_{\dot{B}_{2,\infty}^{-\sigma_1}}^\ell,
\ee
\be\label{4.45}
\|\nabla a^\ell \theta^h\|_{\dot{B}_{2,\infty}^{-\sigma_1}}\lesssim\|\theta^h\|_{\dot{B}_{p,1}^{\frac{N}{p}}}
\|\nabla a^\ell\|_{\dot{B}_{2,\infty}^{-\sigma_1}}\lesssim\|\theta\|_{\dot{B}_{p,1}^{\frac{N}{p}}}^h
\|a\|_{\dot{B}_{2,\infty}^{-\sigma_1}}^\ell,
\ee
and
\be\label{4.46}
\begin{split}
\|\nabla a^h \theta\|_{\dot{B}_{2,\infty}^{-\sigma_1}}^\ell&\lesssim\|\nabla a^h\|_{\dot{B}_{p,1}^{\frac{N}{p}-1}}
\Big(\|\theta\|_{\dot{B}_{p,\infty}^{-\sigma_1+\frac{N}{p}-\frac{N}{2}+1}}
+\|\theta\|_{\dot{B}_{p,\infty}^{-\sigma_1+\frac{2N}{p}-N+1}}\Big)\\[1ex]
&\lesssim\|a^h\|_{\dot{B}_{p,1}^{\frac{N}{p}}}
\Big(\|\theta\|_{\dot{B}_{p,\infty}^{-\sigma_1+\frac{N}{p}-\frac{N}{2}+1}}^h
+\|\theta\|_{\dot{B}_{p,\infty}^{-\sigma_1+\frac{2N}{p}-N+1}}^\ell\Big)\\[1ex]
&\lesssim\|a\|_{\dot{B}_{p,1}^{\frac{N}{p}}}^h\Big(\|\theta\|_{\dot{B}_{p,1}^{\frac{N}{p}}}^h+\|\theta\|_{\dot{B}_{2,\infty}^{-\sigma_1}}^\ell
\Big).
\end{split}
\ee
In addition, the remaining term with $\bar{K}_3(a)a$ may be estimated  similarly as
\be\label{4.47}
\begin{split}
\|\nabla (\bar{K}_3(a)a) \theta\|_{\dot{B}_{2,\infty}^{-\sigma_1}}^\ell&\lesssim\|\bar{K}_3(a)a\|_{\dot{B}_{p,1}^{\frac{N}{p}}}
\Big(\|\theta\|_{\dot{B}_{p,\infty}^{-\sigma_1+\frac{N}{p}-\frac{N}{2}+1}}
+\|\theta\|_{\dot{B}_{p,\infty}^{-\sigma_1+\frac{2N}{p}-N+1}}\Big)\\[1ex]
&\lesssim\|a\|_{\dot{B}_{p,1}^{\frac{N}{p}}}^2
\Big(\|\theta\|_{\dot{B}_{p,\infty}^{-\sigma_1+\frac{N}{p}-\frac{N}{2}+1}}^h
+\|\theta\|_{\dot{B}_{p,\infty}^{-\sigma_1+\frac{2N}{p}-N+1}}^\ell\Big)\\[1ex]
&\lesssim\|a\|_{\dot{B}_{p,1}^{\frac{N}{p}}}^2\Big(\|\theta\|_{\dot{B}_{p,1}^{\frac{N}{p}}}^h+\|\theta\|_{\dot{B}_{2,\infty}^{-\sigma_1}}^\ell
\Big).
\end{split}
\ee

Next, we handle each term in $m$. Notice that
$$
{m}\equ -\bv\cdot\nabla\theta-\beta I(a)\Delta\theta+\frac{Q(\nabla\bv,\nabla\bv)}{1+a}-(K_2(a)+{K}_4(a)\theta){\rm div}\bv.
$$

\underline{Estimate of $\bv\cdot\nabla\theta$}. Decompose $\bv\cdot\nabla\theta=\bv^\ell\cdot\nabla\theta^\ell
+\bv^\ell\cdot\nabla\theta^h+\bv^h\cdot\nabla\theta^\ell+\bv^h\cdot\nabla\theta^h$. It follows from \eqref{4.1} and \eqref{4.2} that
\be\label{4.48}
\|\bv^\ell \nabla\theta^\ell\|_{\dot{B}_{2,\infty}^{-\sigma_1}}\lesssim\|\nabla \theta^\ell\|_{\dot{B}_{p,1}^{\frac{N}{p}}}
\|\bv^\ell\|_{\dot{B}_{2,\infty}^{-\sigma_1}}\lesssim\|\theta\|_{\dot{B}_{2,1}^{\frac{N}{2}+1}}^\ell
\|\bv\|_{\dot{B}_{2,\infty}^{-\sigma_1}}^\ell,
\ee
\be\label{4.49}
\|\bv^h\nabla \theta^\ell\|_{\dot{B}_{2,\infty}^{-\sigma_1}}\lesssim\|\bv^h\|_{\dot{B}_{p,1}^{\frac{N}{p}}}
\|\nabla \theta^\ell\|_{\dot{B}_{2,\infty}^{-\sigma_1}}\lesssim\|\bv\|_{\dot{B}_{p,1}^{\frac{N}{p}+1}}^h
\|\theta\|_{\dot{B}_{2,\infty}^{-\sigma_1}}^\ell,
\ee
\be\label{4.50}
\begin{split}
\|\bv^\ell\nabla \theta^h\|_{\dot{B}_{2,\infty}^{-\sigma_1}}^\ell&\lesssim\|\nabla \theta^h\|_{\dot{B}_{p,1}^{\frac{N}{p}-1}}
\Big(\|\bv^\ell\|_{\dot{B}_{p,\infty}^{-\sigma_1+\frac{N}{p}-\frac{N}{2}+1}}
+\|\bv^\ell\|_{\dot{B}_{p,\infty}^{-\sigma_1+\frac{2N}{p}-N+1}}\Big)\\[1ex]
&\lesssim\|\theta^h\|_{\dot{B}_{p,1}^{\frac{N}{p}}}
\|\bv\|_{\dot{B}_{p,\infty}^{-\sigma_1+\frac{2N}{p}-N+1}}^\ell
\lesssim\|\theta\|_{\dot{B}_{p,1}^{\frac{N}{p}}}^h\|\bv\|_{\dot{B}_{2,\infty}^{-\sigma_1}}^\ell,
\end{split}
\ee
and
\be\label{4.51}
\begin{split}
\|\bv^h\nabla \theta^h\|_{\dot{B}_{2,\infty}^{-\sigma_1}}^\ell&\lesssim\|\bv^h\|_{\dot{B}_{p,1}^{\frac{N}{p}-1}}
\Big(\|\nabla\theta^h\|_{\dot{B}_{p,\infty}^{-\sigma_1+\frac{N}{p}-\frac{N}{2}+1}}
+\|\nabla\theta^h\|_{\dot{B}_{p,\infty}^{-\sigma_1+\frac{2N}{p}-N+1}}\Big)\\[1ex]
&\lesssim\|\bv^h\|_{\dot{B}_{p,1}^{\frac{N}{p}-1}}
\|\theta\|_{\dot{B}_{p,\infty}^{-\sigma_1+\frac{N}{p}-\frac{N}{2}+2}}^h
\lesssim\|\bv\|_{\dot{B}_{p,1}^{\frac{N}{p}-1}}^h\|\theta\|_{\dot{B}_{p,1}^{\frac{N}{p}}}^h,
\end{split}
\ee
where in the last inequality we used that $\sigma_1>2-\frac{N}{2}.$

\underline{Estimate of $I(a)\Delta\theta$}.  One may write
$
I(a)=I^\prime(0)a+\bar{I}(a)a
$
for some smooth function $\bar{I}$ vanishing at $0$. Decomposing  $a\Delta\theta=a^\ell\Delta\theta^\ell+a^\ell\Delta\theta^h+a^h\Delta\theta^\ell+a^h\Delta\theta^h$, we have
\be\label{4.52}
\|a^\ell\Delta\theta^\ell\|_{\dot{B}_{2,\infty}^{-\sigma_1}}\lesssim\|\Delta\theta^\ell
\|_{\dot{B}_{p,1}^{\frac{N}{p}}}
\|a^\ell\|_{\dot{B}_{2,\infty}^{-\sigma_1}}\lesssim\|\theta\|_{\dot{B}_{2,1}^{\frac{N}{2}+1}}^\ell
\|a\|_{\dot{B}_{2,\infty}^{-\sigma_1}}^\ell,
\ee
and
\be\label{4.53}
\|a^h\Delta\theta^\ell\|_{\dot{B}_{2,\infty}^{-\sigma_1}}\lesssim\|a^h\|_{\dot{B}_{p,1}^{\frac{N}{p}}}
\|\Delta\theta^\ell\|_{\dot{B}_{2,\infty}^{-\sigma_1}}\lesssim\|a\|_{\dot{B}_{p,1}^{\frac{N}{p}}}^h
\|\theta\|_{\dot{B}_{2,\infty}^{-\sigma_1}}^\ell.
\ee
By using \eqref{4.2000}, one has
\be\label{4.54}
\begin{split}
\|a^\ell\Delta\theta^h\|_{\dot{B}_{2,\infty}^{-\sigma_1}}^\ell&\lesssim\|\Delta\theta^h\|_{\dot{B}_{p,1}^{\frac{N}{p}-2}}
\Big(\|a^\ell\|_{\dot{B}_{p,\infty}^{-\sigma_1+\frac{N}{p}-\frac{N}{2}+2}}
+\|a^\ell\|_{\dot{B}_{p,\infty}^{-\sigma_1+\frac{2N}{p}-N+1}}\Big)\\[1ex]
&\lesssim\|\theta\|_{\dot{B}_{p,1}^{\frac{N}{p}}}^h\|a^\ell\|_{\dot{B}_{p,\infty}^{-\sigma_1+\frac{2N}{p}-N+1}}
\lesssim\|\theta\|_{\dot{B}_{p,1}^{\frac{N}{p}}}^h\|a\|_{\dot{B}_{2,\infty}^{-\sigma_1}}^\ell,
\end{split}
\ee
and
\be\label{4.55}
\begin{split}
\|a^h\Delta\theta^h\|_{\dot{B}_{2,\infty}^{-\sigma_1}}^\ell&\lesssim\|\Delta\theta^h\|_{\dot{B}_{p,1}^{\frac{N}{p}-2}}
\Big(\|a^h\|_{\dot{B}_{p,\infty}^{-\sigma_1+\frac{N}{p}-\frac{N}{2}+2}}
+\|a^h\|_{\dot{B}_{p,\infty}^{-\sigma_1+\frac{2N}{p}-N+1}}\Big)\\[1ex]
&\lesssim\|\theta^h\|_{\dot{B}_{p,1}^{\frac{N}{p}}}\|a^h\|_{\dot{B}_{p,\infty}^{-\sigma_1+\frac{N}{p}-\frac{N}{2}+2}}
\lesssim\|\theta\|_{\dot{B}_{p,1}^{\frac{N}{p}}}^h\|a\|_{\dot{B}_{p,1}^{\frac{N}{p}}}^h.
\end{split}
\ee
On the other hand, from \eqref{4.1}, \eqref{4.2000}, Proposition \ref{pra.4} and Corollaries \ref{co2.1} and \ref{co2.2} again, we still have
\be\label{4.56}
\|\bar{I}(a)a\Delta\theta^\ell\|_{\dot{B}_{2,\infty}^{-\sigma_1}}
\lesssim\|\bar{I}(a)a\|_{\dot{B}_{p,1}^{\frac{N}{p}}}\|\Delta\theta^\ell\|_{\dot{B}_{2,\infty}^{-\sigma_1}}
\lesssim\|a\|_{\dot{B}_{p,1}^{\frac{N}{p}}}^2\|\theta\|_{\dot{B}_{2,\infty}^{-\sigma_1}}^\ell,
\ee
and
\be\label{4.57}
\begin{split}
\|\bar{I}(a)a\Delta\theta^h\|_{\dot{B}_{2,\infty}^{-\sigma_1}}^\ell
&\lesssim\|\Delta\theta^h\|_{\dot{B}_{p,1}^{\frac{N}{p}-2}}
\Big(\|\bar{I}(a)a\|_{\dot{B}_{p,\infty}^{-\sigma_1+\frac{N}{p}-\frac{N}{2}+2}}
+\|\bar{I}(a)a\|_{\dot{B}_{p,\infty}^{-\sigma_1+\frac{2N}{p}-N+1}}\Big)\\[1ex]
&\lesssim\|\theta^h\|_{\dot{B}_{p,1}^{\frac{N}{p}}}\|\bar{I}(a)\|_{\dot{B}_{p,1}^{\frac{N}{p}}}
\Big(\|a\|_{\dot{B}_{p,\infty}^{-\sigma_1+\frac{N}{p}-\frac{N}{2}+2}}
+\|a\|_{\dot{B}_{p,\infty}^{-\sigma_1+\frac{2N}{p}-N+1}}\Big)\\[1ex]
&\lesssim\|\theta\|_{\dot{B}_{p,1}^{\frac{N}{p}}}^h
\|a\|_{\dot{B}_{p,1}^{\frac{N}{p}}}\Big(\|a\|_{\dot{B}_{p,\infty}^{-\sigma_1+\frac{N}{p}-\frac{N}{2}+2}}^h
+\|a\|_{\dot{B}_{p,\infty}^{-\sigma_1+\frac{2N}{p}-N+1}}^\ell\Big)\\[1ex]
&\lesssim\|\theta\|_{\dot{B}_{p,1}^{\frac{N}{p}}}^h
\|a\|_{\dot{B}_{p,1}^{\frac{N}{p}}}\Big(\|a\|_{\dot{B}_{p,1}^{\frac{N}{p}}}^h
+\|a\|_{\dot{B}_{2,\infty}^{-\sigma_1}}^\ell\Big).
\end{split}
\ee

\underline{Estimate of $Q(\nabla\bv,\nabla\bv)/(1+a)$}. We first get from \eqref{4.1} that
\be\label{4.58}
\|Q(\nabla\bv,\nabla\bv)/(1+a)\|_{\dot{B}_{2,\infty}^{-\sigma_1}}^\ell\lesssim \Big(1+\|a\|_{\dot{B}_{p,1}^{\frac{N}{p}}}\Big)
\||\nabla\bv|^2\|_{\dot{B}_{2,\infty}^{-\sigma_1}}.
\ee
In what follows, we focus on the estimation of $\||\nabla\bv|^2\|_{\dot{B}_{2,\infty}^{-\sigma_1}}$. By \eqref{4.1} again, one derives that
\be\label{4.59}
\|\nabla\bv^\ell\nabla\bv^\ell\|_{\dot{B}_{2,\infty}^{-\sigma_1}}\lesssim \|\nabla\bv^\ell\|_{\dot{B}_{p,1}^{\frac{N}{p}}}
\|\nabla\bv^\ell\|_{\dot{B}_{2,\infty}^{-\sigma_1}}
\lesssim\|\bv\|_{\dot{B}_{2,1}^{\frac{N}{2}+1}}^\ell\|\bv\|_{\dot{B}_{2,\infty}^{-\sigma_1}}^\ell,
\ee
\be\label{4.60}
\|\nabla\bv^\ell\nabla\bv^h\|_{\dot{B}_{2,\infty}^{-\sigma_1}}\lesssim \|\nabla\bv^h\|_{\dot{B}_{p,1}^{\frac{N}{p}}}
\|\nabla\bv^\ell\|_{\dot{B}_{2,\infty}^{-\sigma_1}}
\lesssim\|\bv\|_{\dot{B}_{p,1}^{\frac{N}{p}+1}}^h\|\bv\|_{\dot{B}_{2,\infty}^{-\sigma_1}}^\ell.
\ee
For the term $\nabla\bv^h\nabla\bv^h$, we apply estimate \eqref{4.2001} to get that
\be\label{4.61}
\|\nabla\bv^h\nabla\bv^h\|_{\dot{B}_{2,\infty}^{-\sigma_1}}\lesssim \|\nabla\bv^h\|_{\dot{B}_{p,1}^{\frac{N}{p}-1}}
\|\nabla\bv^h\|_{\dot{B}_{p,\infty}^{-\sigma_1+\frac{N}{p}-\frac{N}{2}+1}}
\lesssim\|\bv\|_{\dot{B}_{p,1}^{\frac{N}{p}}}^h\|\bv\|_{\dot{B}_{p,1}^{\frac{N}{p}}}^h,
\ee
where we have used that $\sigma_1>2-\frac{N}{2}$.

\underline{Estimate of $K_2(a){\rm div}\bv$}. As before, writing  $K_2(a)=K_2^\prime(0)a+\bar{K}_2(a)a$, we have
\be\label{4.62}
\|a^\ell{\rm div}\bv^\ell\|_{\dot{B}_{2,\infty}^{-\sigma_1}}\lesssim\|{\rm div}\bv^\ell\|_{\dot{B}_{p,1}^{\frac{N}{p}}}
\|a^\ell\|_{\dot{B}_{2,\infty}^{-\sigma_1}}\lesssim\|\bv\|_{\dot{B}_{2,1}^{\frac{N}{2}+1}}^\ell
\|a\|_{\dot{B}_{2,\infty}^{-\sigma_1}}^\ell,
\ee
\be\label{4.64}
\|a^\ell{\rm div}\bv^h\|_{\dot{B}_{2,\infty}^{-\sigma_1}}\lesssim\|{\rm div}\bv^h\|_{\dot{B}_{p,1}^{\frac{N}{p}}}
\|a^\ell\|_{\dot{B}_{2,\infty}^{-\sigma_1}}
\lesssim\|\bv\|_{\dot{B}_{p,1}^{\frac{N}{p}+1}}^h\|a\|_{\dot{B}_{2,\infty}^{-\sigma_1}}^\ell,
\ee
\be\label{4.63}
\|a^h{\rm div}\bv^\ell\|_{\dot{B}_{2,\infty}^{-\sigma_1}}\lesssim\|a^h\|_{\dot{B}_{p,1}^{\frac{N}{p}}}
\|{\rm div}\bv^\ell\|_{\dot{B}_{2,\infty}^{-\sigma_1}}\lesssim\|a\|_{\dot{B}_{p,1}^{\frac{N}{p}}}^h
\|\bv\|_{\dot{B}_{2,\infty}^{-\sigma_1}}^\ell
\ee
and
\be\label{4.65}
\begin{split}
\|a^h{\rm div}\bv^h\|_{\dot{B}_{2,\infty}^{-\sigma_1}}^\ell&\lesssim\|{\rm div}\bv^h\|_{\dot{B}_{p,1}^{\frac{N}{p}-1}}
\Big(\|a^h\|_{\dot{B}_{p,\infty}^{-\sigma_1+\frac{N}{p}-\frac{N}{2}+1}}
+\|a^h\|_{\dot{B}_{p,\infty}^{-\sigma_1+\frac{2N}{p}-N+1}}\Big)\\[1ex]
&\lesssim\|\bv^h\|_{\dot{B}_{p,1}^{\frac{N}{p}}}\|a^h\|_{\dot{B}_{p,\infty}^{-\sigma_1+\frac{N}{p}-\frac{N}{2}+1}}
\lesssim\|\bv\|_{\dot{B}_{p,1}^{\frac{N}{p}+1}}^h\|a\|_{\dot{B}_{p,1}^{\frac{N}{p}}}^h.
\end{split}
\ee
As for the term $\bar{K}_2(a)a{\rm div}\bv$, we deduce
\be\label{4.66}
\begin{split}
\|\bar{K}_2(a)a{\rm div}\bv^\ell\|_{\dot{B}_{2,\infty}^{-\sigma_1}}\lesssim\|\bar{K}_2(a)a\|_{\dot{B}_{p,1}^{\frac{N}{p}}}
\|{\rm div}\bv\|_{\dot{B}_{2,\infty}^{-\sigma_1}}^\ell\lesssim\|a\|_{\dot{B}_{p,1}^{\frac{N}{p}}}^2
\|\bv\|_{\dot{B}_{2,\infty}^{-\sigma_1}}^\ell,
\end{split}
\ee
and
\be\label{4.67}
\begin{split}
\|\bar{K}_2(a)a{\rm div}\bv^h\|_{\dot{B}_{2,\infty}^{-\sigma_1}}^\ell&\lesssim\|{\rm div}\bv^h\|_{\dot{B}_{p,1}^{\frac{N}{p}-1}}
\Big(\|\bar{K}_2(a)a\|_{\dot{B}_{p,\infty}^{-\sigma_1+\frac{N}{p}-\frac{N}{2}+1}}
+\|\bar{K}_2(a)a\|_{\dot{B}_{p,\infty}^{-\sigma_1+\frac{2N}{p}-N+1}}\Big)\\[1ex]
&\lesssim\|\bv^h\|_{\dot{B}_{p,1}^{\frac{N}{p}}}\|\bar{K}_2(a)\|_{\dot{B}_{p,1}^{\frac{N}{p}}}
\Big(\|a\|_{\dot{B}_{p,\infty}^{-\sigma_1+\frac{N}{p}-\frac{N}{2}+1}}
+\|a\|_{\dot{B}_{p,\infty}^{-\sigma_1+\frac{2N}{p}-N+1}}\Big)\\[1ex]
&\lesssim\|\bv\|_{\dot{B}_{p,1}^{\frac{N}{p}+1}}^h
\|a\|_{\dot{B}_{p,1}^{\frac{N}{p}}}\Big(\|a\|_{\dot{B}_{2,\infty}^{-\sigma_1}}^\ell
+\|a\|_{\dot{B}_{p,1}^{\frac{N}{p}}}^h\Big).
\end{split}
\ee

\underline{Estimate of $K_4(a)\theta{\rm div}\bv$}. Noticing  that $K_4(0)\neq 0$, we decompose $K_4(a)=K_4(0)+\bar{K}_4(a)$, where $\bar{K}_4(a)$ is a smooth function satisfying $\bar{K}_4(0)=0$. In the following, we handle the term $\theta{\rm div}\bv$ and $\bar{K}_4(a)\theta{\rm div}\bv$, respectively. Firstly, for the term $\theta{\rm div}\bv$, using \eqref{4.1} and \eqref{4.2000}, one has
\be\label{4.68}
\|\theta^\ell{\rm div}\bv\|_{\dot{B}_{2,\infty}^{-\sigma_1}}\lesssim\|{\rm div}\bv\|_{\dot{B}_{p,1}^{\frac{N}{p}}}\|\theta\|_{\dot{B}_{2,\infty}^{-\sigma_1}}^\ell
\lesssim\Big(\|\bv\|_{\dot{B}_{2,1}^{\frac{N}{2}+1}}^\ell+\|\bv\|_{\dot{B}_{p,1}^{\frac{N}{p}+1}}^h\Big)
\|\theta\|_{\dot{B}_{2,\infty}^{-\sigma_1}}^\ell,
\ee
\be\label{4.69}
\|\theta^h{\rm div}\bv^\ell\|_{\dot{B}_{2,\infty}^{-\sigma_1}}\lesssim\|\theta^h\|_{\dot{B}_{p,1}^{\frac{N}{p}}}
\|{\rm div}\bv^\ell\|_{\dot{B}_{2,\infty}^{-\sigma_1}}\lesssim
\|\theta\|_{\dot{B}_{p,1}^{\frac{N}{p}}}^h\|\bv\|_{\dot{B}_{2,\infty}^{-\sigma_1}}^\ell,
\ee
and
\be\label{4.70}
\begin{split}
\|\theta^h{\rm div}\bv^h\|_{\dot{B}_{2,\infty}^{-\sigma_1}}^\ell&\lesssim\|\theta^h\|_{\dot{B}_{p,1}^{\frac{N}{p}-2}}
\Big(\|{\rm div}\bv^h\|_{\dot{B}_{p,\infty}^{-\sigma_1+\frac{N}{p}-\frac{N}{2}+2}}+\|{\rm div}\bv^h\|_{\dot{B}_{p,\infty}^{-\sigma_1+\frac{2N}{p}-N+1}}\Big)\\[1ex]
&\lesssim \|\theta\|_{\dot{B}_{p,1}^{\frac{N}{p}-2}}^h\|\bv\|_{\dot{B}_{p,\infty}^{-\sigma_1+\frac{N}{p}-\frac{N}{2}+3}}^h
\lesssim\|\theta\|_{\dot{B}_{p,1}^{\frac{N}{p}-2}}^h\|\bv\|_{\dot{B}_{p,1}^{\frac{N}{p}+1}}^h.
\end{split}
\ee
For the term $\bar{K}_4(a)\theta{\rm div}\bv$, we have from \eqref{4.1} and \eqref{4.2} again that
\be\label{4.71}
\begin{split}
\|\bar{K}_4(a)\theta{\rm div}\bv^\ell\|_{\dot{B}_{2,\infty}^{-\sigma_1}}&\lesssim\|\bar{K}_4(a)\theta\|_{\dot{B}_{p,1}^{\frac{N}{p}}}
\|{\rm div}\bv\|_{\dot{B}_{2,\infty}^{-\sigma_1}}^\ell\lesssim \|a\|_{\dot{B}_{p,1}^{\frac{N}{p}}}\|\theta\|_{\dot{B}_{p,1}^{\frac{N}{p}}}\|\bv\|_{\dot{B}_{2,\infty}^{-\sigma_1}}^\ell\\[1ex]
&\lesssim\|a\|_{\dot{B}_{p,1}^{\frac{N}{p}}}\Big(\|\theta\|_{\dot{B}_{2,1}^{\frac{N}{2}}}^\ell
+\|\theta\|_{\dot{B}_{p,1}^{\frac{N}{p}}}^h\Big)\|\bv\|_{\dot{B}_{2,\infty}^{-\sigma_1}}^\ell\\[1ex]
&\lesssim\|a\|_{\dot{B}_{p,1}^{\frac{N}{p}}}\|\theta\|_{\dot{B}_{2,1}^{\frac{N}{2}}}^\ell
\|\bv\|_{\dot{B}_{2,\infty}^{-\sigma_1}}^\ell+\|a\|_{\dot{B}_{p,1}^{\frac{N}{p}}}\|\theta\|_{\dot{B}_{p,1}^{\frac{N}{p}}}^h
\|\bv\|_{\dot{B}_{2,\infty}^{-\sigma_1}}^\ell
\end{split}
\ee
and
\be\label{4.72}
\begin{split}
\|\bar{K}_4(a)\theta{\rm div}\bv^h\|_{\dot{B}_{2,\infty}^{-\sigma_1}}^\ell
&\lesssim\|{\rm div}\bv^h\|_{\dot{B}_{p,1}^{\frac{N}{p}-1}}
\Big(\|\bar{K}_4(a)\theta\|_{\dot{B}_{p,\infty}^{-\sigma_1+\frac{N}{p}-\frac{N}{2}+1}}
+\|\bar{K}_4(a)\theta\|_{\dot{B}_{p,\infty}^{-\sigma_1+\frac{2N}{p}-N+1}}\Big)\\[1ex]
&\lesssim\|\bv^h\|_{\dot{B}_{p,1}^{\frac{N}{p}}}\|\bar{K}_4(a)\|_{\dot{B}_{p,1}^{\frac{N}{p}}}
\Big(\|\theta\|_{\dot{B}_{p,\infty}^{-\sigma_1+\frac{N}{p}-\frac{N}{2}+1}}^h
+\|\theta\|_{\dot{B}_{p,\infty}^{-\sigma_1+\frac{2N}{p}-N+1}}^\ell\Big)\\[1ex]
&\lesssim\|\bv\|_{\dot{B}_{p,1}^{\frac{N}{p}}}^h\|a\|_{\dot{B}_{p,1}^{\frac{N}{p}}}\Big(
\|\theta\|_{\dot{B}_{p,1}^{\frac{N}{p}-1}}^h+\|\theta\|_{\dot{B}_{2,\infty}^{-\sigma_1}}^\ell\Big)\\[1ex]
&\lesssim\|\bv\|_{\dot{B}_{p,1}^{\frac{N}{p}}}^h
\|\theta\|_{\dot{B}_{p,1}^{\frac{N}{p}-1}}^h\|a\|_{\dot{B}_{p,1}^{\frac{N}{p}}}
+\|\bv\|_{\dot{B}_{p,1}^{\frac{N}{p}}}^h\|a\|_{\dot{B}_{p,1}^{\frac{N}{p}}}\|\theta\|_{\dot{B}_{2,\infty}^{-\sigma_1}}^\ell.
\end{split}
\ee
Plugging all  estimates above in \eqref{4.13}, we end up with the proof of \eqref{4.11}.
\end{proof}

By the definition of $\mathcal{X}_p(t)$ in Theorem \ref{th1}, one has
\be\nonumber
\begin{split}
\|(a, \bv, \theta)\|_{L_t^2(\dot{B}_{p,1}^{\frac{N}{p}})}^\ell&\lesssim\|(a, \bv, \theta)\|_{L_t^2(\dot{B}_{2,1}^{\frac{N}{2}})}^\ell\lesssim\Big(\|(a, \bv, \theta)\|_{L_t^\infty(\dot{B}_{2,1}^{\frac{N}{2}-1})}^\ell\Big)^{\frac{1}{2}}\Big(\|(a, \bv, \theta)\|_{L_t^1(\dot{B}_{2,1}^{\frac{N}{2}+1})}^\ell\Big)^{\frac{1}{2}},
\end{split}
\ee
\be\nonumber
\|a\|_{L_t^2(\dot{B}_{p,1}^{\frac{N}{p}})}^h\lesssim\Big(\|a\|_{L_t^\infty(\dot{B}_{p,1}^{\frac{N}{p}})}^h\Big)^{\frac{1}{2}}
\Big(\|a\|_{L_t^1(\dot{B}_{p,1}^{\frac{N}{p}})}^h\Big)^{\frac{1}{2}},
\ee
\be\nonumber
\|\bv\|_{L_t^2(\dot{B}_{p,1}^{\frac{N}{p}})}^h\lesssim\Big(\|\bv\|_{L_t^\infty(\dot{B}_{p,1}^{\frac{N}{p}-1})}^h\Big)^{\frac{1}{2}}
\Big(\|\bv\|_{L_t^1(\dot{B}_{p,1}^{\frac{N}{p}+1})}^h\Big)^{\frac{1}{2}}
\ee
and
\be\nonumber
\|\theta\|_{L_t^2(\dot{B}_{p,1}^{\frac{N}{p}-1})}^h\lesssim\Big(\|\theta\|_{L_t^\infty(\dot{B}_{p,1}^{\frac{N}{p}-2})}^h\Big)^{\frac{1}{2}}
\Big(\|\theta\|_{L_t^1(\dot{B}_{p,1}^{\frac{N}{p}})}^h\Big)^{\frac{1}{2}}.
\ee
On the other hand, it follows that
\be\nonumber
\|a\|_{L_t^\infty(\dot{B}_{p,1}^{\frac{N}{p}})}\lesssim\|a\|_{L_t^\infty(\dot{B}_{p,1}^{\frac{N}{p}})}^\ell
+\|a\|_{L_t^\infty(\dot{B}_{p,1}^{\frac{N}{p}})}^h\lesssim\|a\|_{L_t^\infty(\dot{B}_{2,1}^{\frac{N}{2}-1})}^\ell
+\|a\|_{L_t^\infty(\dot{B}_{p,1}^{\frac{N}{p}})}^h.
\ee
Then, we have
\be\label{4.73}
\int_0^t(A_1(\tau)+A_2(\tau))d\tau\leq \mathcal{X}_p+\mathcal{X}_p^2+\mathcal{X}_p^3\leq C\mathcal{X}_{p,0},
\ee
which yields from Gronwall's inequality that
\be\label{4.74}
\|(a,\bv,\theta)\|_{\dot{B}_{2,\infty}^{-\sigma_1}}^\ell\leq C_0
\ee
for all $t\geq 0$, where $C_0>0$ depends on $\|(a_0, \bv_0, \theta_0)\|_{\dot{B}_{2,\infty}^{-\sigma_1}}^\ell$ and $\mathcal{X}_{p,0}$.
\section{Proofs of main results}\label{s:6}
This section is devoted to proving Theorem \ref{th2} and Corollary \ref{col1}.
\subsection{Proof of Theorem \ref{th2}}
From Lemmas \ref{lem2} and \ref{lem3}, one deduces that
\be\label{5.1}
\begin{split}
  &\quad\frac{d}{dt}\Big(\|(a, \bv, \theta)\|_{\dot{B}_{2,1}^{\frac{N}{2}-1}}^\ell+\|(\nabla a, \bv)\|_{\dot{B}_{p,1}^{\frac{N}{p}-1}}^h+\|\theta\|_{\dot{B}_{p,1}^{\frac{N}{p}-2}}^h\Big)\\[1ex]
  &\quad\quad\quad\quad\quad\quad\quad\quad\quad\quad\quad\quad\quad\quad
  +\Big(\|(a, \bv, \theta)\|_{\dot{B}_{2,1}^{\frac{N}{2}+1}}^\ell
  +\|(a, \theta)\|_{\dot{B}_{p,1}^{\frac{N}{p}}}^h+\|\bv\|_{\dot{B}_{p,1}^{\frac{N}{p}+1}}^h\Big)\\[1ex]
  &\lesssim\|(f, {\bf g}, {m})\|_{\dot{B}_{2,1}^{\frac{N}{2}-1}}^\ell+\|(f, m)\|_{\dot{B}_{p,1}^{\frac{N}{p}-2}}^h
  +\|{\bf g}\|_{\dot{B}_{p,1}^{\frac{N}{p}-1}}^h+\|\nabla\bv\|_{\dot{B}_{p,1}^{\frac{N}{p}}}\|a\|_{\dot{B}_{p,1}^{\frac{N}{p}}}.
\end{split}
\ee
In what follows, we deal with the terms in the right hand of \eqref{5.1} one by one. Firstly, for the last term, we have
\be\label{5.2}
\begin{split}
\|\nabla\bv\|_{\dot{B}_{p,1}^{\frac{N}{p}}}\|a\|_{\dot{B}_{p,1}^{\frac{N}{p}}}
&\lesssim\Big(\|a\|_{\dot{B}_{2,1}^{\frac{N}{2}-1}}^\ell+\|a\|_{\dot{B}_{p,1}^{\frac{N}{p}}}^h\Big)
\Big(\|\bv\|_{\dot{B}_{2,1}^{\frac{N}{2}+1}}^\ell+\|\bv\|_{\dot{B}_{p,1}^{\frac{N}{p}+1}}^h\Big)\\[1ex]
&\lesssim\mathcal{X}_p(t)\Big(\|\bv\|_{\dot{B}_{2,1}^{\frac{N}{2}+1}}^\ell+\|\bv\|_{\dot{B}_{p,1}^{\frac{N}{p}+1}}^h\Big).
\end{split}
\ee
Next, we estimate $\|f\|_{\dot{B}_{p,1}^{\frac{N}{p}-2}}^h$ and notice that
$$
\|f\|_{\dot{B}_{p,1}^{\frac{N}{p}-2}}^h\lesssim \|a\bv\|_{\dot{B}_{p,1}^{\frac{N}{p}-1}}^h.
$$
Decomposing $a\bv=a^\ell\bv^\ell+a^\ell\bv^h+a^h\bv$, we have
$$
\|a^\ell\bv^h\|_{\dot{B}_{p,1}^{\frac{N}{p}-1}}^h\lesssim\|a^\ell\|_{\dot{B}_{p,1}^{\frac{N}{p}-1}}
\|\bv^h\|_{\dot{B}_{p,1}^{\frac{N}{p}}}\lesssim\|a\|_{\dot{B}_{2,1}^{\frac{N}{2}-1}}^\ell
\|\bv\|_{\dot{B}_{p,1}^{\frac{N}{p}+1}}^h\lesssim\mathcal{X}_p(t)\|\bv\|_{\dot{B}_{p,1}^{\frac{N}{p}+1}}^h,
$$
and
$$
\|a^h\bv\|_{\dot{B}_{p,1}^{\frac{N}{p}-1}}^h\lesssim\|a^h\|_{\dot{B}_{p,1}^{\frac{N}{p}}}
\|\bv\|_{\dot{B}_{p,1}^{\frac{N}{p}-1}}\lesssim\mathcal{X}_p(t)\|a\|_{\dot{B}_{p,1}^{\frac{N}{p}}}^h.
$$
It follows from Corollary \ref{co2.1} and Bernstein inequality that
$$
\|a^\ell\bv^\ell\|_{\dot{B}_{p,1}^{\frac{N}{p}-1}}^h\lesssim\|a^\ell\bv^\ell\|_{\dot{B}_{2,1}^{\frac{N}{2}+1}}
\lesssim\|a^\ell\|_{L^\infty}\|\bv^\ell\|_{\dot{B}_{2,1}^{\frac{N}{2}+1}}
+\|\bv^\ell\|_{L^\infty}\|a^\ell\|_{\dot{B}_{2,1}^{\frac{N}{2}+1}}
\lesssim\mathcal{X}_p(t)\|(a,\bv)\|_{\dot{B}_{2,1}^{\frac{N}{2}+1}}^\ell.
$$
Therefore, we conclude that
\be\label{5.3}
\|f\|_{\dot{B}_{p,1}^{\frac{N}{p}-2}}^h\lesssim\mathcal{X}_p(t)\Big(\|a\|_{\dot{B}_{p,1}^{\frac{N}{p}}}^h
+\|(a,\bv)\|_{\dot{B}_{2,1}^{\frac{N}{2}+1}}^\ell+\|\bv\|_{\dot{B}_{p,1}^{\frac{N}{p}+1}}^h\Big).
\ee

Similarly, we deal with $\|{\bf g}\|_{\dot{B}_{p,1}^{\frac{N}{p}-1}}^h$ as follows.
\be\label{5.4}
\|\bv\cdot\nabla \bv\|_{\dot{B}_{p,1}^{\frac{N}{p}-1}}^h\lesssim\|\bv\|_{\dot{B}_{p,1}^{\frac{N}{p}-1}}
\|\nabla\bv\|_{\dot{B}_{p,1}^{\frac{N}{p}}}
\lesssim\mathcal{X}_p(t)\Big(\|\bv\|_{\dot{B}_{2,1}^{\frac{N}{2}+1}}^\ell+\|\bv\|_{\dot{B}_{p,1}^{\frac{N}{p}+1}}^h\Big).
\ee
\be\label{5.5}
\begin{split}
\|I(a)\widetilde{\mathcal{A}}\bv\|_{\dot{B}_{p,1}^{\frac{N}{p}-1}}^h&\lesssim\|I(a)\|_{\dot{B}_{p,1}^{\frac{N}{p}}}
\|\widetilde{\mathcal{A}}\bv\|_{\dot{B}_{p,1}^{\frac{N}{p}-1}}\\[1ex]
&
\lesssim\|a\|_{\dot{B}_{p,1}^{\frac{N}{p}}}\Big(\|\bv\|_{\dot{B}_{2,1}^{\frac{N}{2}+1}}^\ell
+\|\bv\|_{\dot{B}_{p,1}^{\frac{N}{p}+1}}^h\Big)
\lesssim\mathcal{X}_p(t)\Big(\|\bv\|_{\dot{B}_{2,1}^{\frac{N}{2}+1}}^\ell+\|\bv\|_{\dot{B}_{p,1}^{\frac{N}{p}+1}}^h\Big).
\end{split}
\ee

\be\label{5.6}
\begin{split}
&\|K_1(a)\nabla a\|_{\dot{B}_{p,1}^{\frac{N}{p}-1}}^h\lesssim\|K_1(a)\|_{\dot{B}_{p,1}^{\frac{N}{p}}}
\|\nabla a\|_{\dot{B}_{p,1}^{\frac{N}{p}-1}}\lesssim\|a\|_{\dot{B}_{p,1}^{\frac{N}{p}}}^2
\lesssim\|a^\ell\|_{\dot{B}_{p,1}^{\frac{N}{p}}}^2
+\|a^h\|_{\dot{B}_{p,1}^{\frac{N}{p}}}^2\\[1ex]
&\quad\quad\quad\lesssim\|a^\ell\|_{\dot{B}_{p,1}^{\frac{N}{p}-1}}\|a^\ell\|_{\dot{B}_{p,1}^{\frac{N}{p}+1}}
+\|a^h\|_{\dot{B}_{p,1}^{\frac{N}{p}}}\|a^h\|_{\dot{B}_{p,1}^{\frac{N}{p}}}
\lesssim\mathcal{X}_p(t)\Big(\|a\|_{\dot{B}_{2,1}^{\frac{N}{2}+1}}^\ell+\|a\|_{\dot{B}_{p,1}^{\frac{N}{p}}}^h\Big).
\end{split}
\ee
\be\label{5.7}
\begin{split}
&\quad\|K_2(a)\nabla \theta+\theta \nabla K_3(a)\|_{\dot{B}_{p,1}^{\frac{N}{p}-1}}^h\\[1ex]
&\lesssim\|K_2(a)\|_{\dot{B}_{p,1}^{\frac{N}{p}}}
\|\nabla \theta\|_{\dot{B}_{p,1}^{\frac{N}{p}-1}}+\|\theta\|_{\dot{B}_{p,1}^{\frac{N}{p}}}
\|\nabla K_3(a)\|_{\dot{B}_{p,1}^{\frac{N}{p}-1}}
\lesssim\|a\|_{\dot{B}_{p,1}^{\frac{N}{p}}}
\|\theta\|_{\dot{B}_{p,1}^{\frac{N}{p}}}\\[1ex]
&\lesssim\|a\|_{\dot{B}_{p,1}^{\frac{N}{p}}}^\ell\|\theta\|_{\dot{B}_{p,1}^{\frac{N}{p}}}^\ell
+\|a\|_{\dot{B}_{p,1}^{\frac{N}{p}}}^\ell\|\theta\|_{\dot{B}_{p,1}^{\frac{N}{p}}}^h+
\|a\|_{\dot{B}_{p,1}^{\frac{N}{p}}}^h\|\theta\|_{\dot{B}_{p,1}^{\frac{N}{p}}}^\ell
+\|a\|_{\dot{B}_{p,1}^{\frac{N}{p}}}^h\|\theta\|_{\dot{B}_{p,1}^{\frac{N}{p}}}^h\\[1ex]
&\lesssim(\|a\|_{\dot{B}_{2,1}^{\frac{N}{2}}}^\ell)^2+(\|\theta\|_{\dot{B}_{2,1}^{\frac{N}{2}}}^\ell)^2
+\|a\|_{\dot{B}_{2,1}^{\frac{N}{2}-1}}^\ell\|\theta\|_{\dot{B}_{p,1}^{\frac{N}{p}}}^h
+\|a\|_{\dot{B}_{p,1}^{\frac{N}{p}}}^h\|\theta\|_{\dot{B}_{2,1}^{\frac{N}{2}-1}}^\ell
+\|a\|_{\dot{B}_{p,1}^{\frac{N}{p}}}^h\|\theta\|_{\dot{B}_{p,1}^{\frac{N}{p}}}^h\\[1ex]
&\lesssim\|a\|_{\dot{B}_{2,1}^{\frac{N}{2}-1}}^\ell\|a\|_{\dot{B}_{2,1}^{\frac{N}{2}+1}}^\ell
+\|\theta\|_{\dot{B}_{2,1}^{\frac{N}{2}-1}}^\ell\|\theta\|_{\dot{B}_{2,1}^{\frac{N}{2}+1}}^\ell
+\|a\|_{\dot{B}_{2,1}^{\frac{N}{2}-1}}^\ell\|\theta\|_{\dot{B}_{p,1}^{\frac{N}{p}}}^h
+\|a\|_{\dot{B}_{p,1}^{\frac{N}{p}}}^h\|\theta\|_{\dot{B}_{2,1}^{\frac{N}{2}-1}}^\ell
+\|a\|_{\dot{B}_{p,1}^{\frac{N}{p}}}^h\|\theta\|_{\dot{B}_{p,1}^{\frac{N}{p}}}^h\\[1ex]
&\lesssim \mathcal{X}_p(t)
\Big(\|a\|_{\dot{B}_{2,1}^{\frac{N}{2}+1}}^\ell+\|a\|_{\dot{B}_{p,1}^{\frac{N}{p}}}^h
+\|\theta\|_{\dot{B}_{2,1}^{\frac{N}{2}+1}}^\ell+\|\theta\|_{\dot{B}_{p,1}^{\frac{N}{p}}}^h\Big).
\end{split}
\ee
 As a consequence, we end up with
\be\label{5.8}
\|{\bf g}\|_{\dot{B}_{p,1}^{\frac{N}{p}-1}}^h\lesssim\mathcal{X}_p(t)\Big(\|a\|_{\dot{B}_{p,1}^{\frac{N}{p}}}^h
+\|(a,\bv)\|_{\dot{B}_{2,1}^{\frac{N}{2}+1}}^\ell+\|\bv\|_{\dot{B}_{p,1}^{\frac{N}{p}+1}}^h
+\|\theta\|_{\dot{B}_{2,1}^{\frac{N}{2}+1}}^\ell+\|\theta\|_{\dot{B}_{p,1}^{\frac{N}{p}}}^h\Big).
\ee

Let us mention that, in the following estimations, if the term $\|\theta\|_{\dot{B}_{p, 1}^{\frac{N}{p}}}\|a\|_{\dot{B}_{p, 1}^{\frac{N}{p}}}$ appears again, we handle it as in \eqref{5.7}.

We next estimate $\|m\|_{\dot{B}_{p,1}^{\frac{N}{p}-2}}^h$ and recall that
$$
{m}\equ -\bv\cdot\nabla\theta-\beta I(a)\Delta\theta+\frac{Q(\nabla\bv,\nabla\bv)}{1+a}-(K_2(a)+{K}_4(a)\theta){\rm div}\bv.
$$
Using Bony's decomposition, Bernstein's inequality, Propositions \ref{pr2.2} and \ref{pra.4}, and Corollary \ref{co2.1}, we infer that
\be\label{5.9}
\begin{split}
\|\bv\cdot\nabla\theta\|_{\dot{B}_{p,1}^{\frac{N}{p}-2}}^h&\leq \|\bv\cdot\nabla\theta^\ell\|_{\dot{B}_{p,1}^{\frac{N}{p}-2}}^h
+\|\bv\cdot\nabla\theta^h\|_{\dot{B}_{p,1}^{\frac{N}{p}-2}}^h\lesssim \|\bv\cdot\nabla\theta^\ell\|_{\dot{B}_{p,1}^{\frac{N}{p}-1}}^h
+\|\bv\cdot\nabla\theta^h\|_{\dot{B}_{p,1}^{\frac{N}{p}-2}}^h\\[1ex]
&\lesssim\|\bv\|_{\dot{B}_{p,1}^{\frac{N}{p}-1}}\Big(\|\nabla\theta\|_{\dot{B}_{p,1}^{\frac{N}{p}}}^\ell+
\|\nabla\theta\|_{\dot{B}_{p,1}^{\frac{N}{p}-1}}^h\Big)
\lesssim\mathcal{X}_p(t)\Big(\|\theta\|_{\dot{B}_{2,1}^{\frac{N}{2}+1}}^\ell+\|\theta\|_{\dot{B}_{p,1}^{\frac{N}{p}}}^h\Big).
\end{split}
\ee
For the term $I(a)\Delta\theta$,  one has
\be\label{5.10}
\begin{split}
\|I(a)\Delta\theta\|_{\dot{B}_{p,1}^{\frac{N}{p}-2}}^h
&\lesssim\|I(a)\|_{\dot{B}_{p,1}^{\frac{N}{p}}}\|\Delta\theta\|_{\dot{B}_{p,1}^{\frac{N}{p}-2}}\lesssim\|a\|
_{\dot{B}_{p,1}^{\frac{N}{p}}}\|\theta\|_{\dot{B}_{p,1}^{\frac{N}{p}}}\\[1ex]
&\lesssim\mathcal{X}_p(t)
\Big(\|a\|_{\dot{B}_{2,1}^{\frac{N}{2}+1}}^\ell+\|a\|_{\dot{B}_{p,1}^{\frac{N}{p}}}^h
+\|\theta\|_{\dot{B}_{2,1}^{\frac{N}{2}+1}}^\ell+\|\theta\|_{\dot{B}_{p,1}^{\frac{N}{p}}}^h\Big).
\end{split}
\ee
\be\label{5.11}
\begin{split}
&\quad\left\|Q(\nabla v, \nabla v)/(1+a)\right\|_{\dot{B}_{p,1}^{\frac{N}{p}-2}}^h
\lesssim(1+\|a\|_{\dot{B}_{p,1}^{\frac{N}{p}}})\||\nabla\bv|^2\|_{\dot{B}_{p,1}^{\frac{N}{p}-2}}\\[1ex]
&\lesssim(1+\|a\|_{\dot{B}_{p,1}^{\frac{N}{p}}})\|\bv\|_{\dot{B}_{p,1}^{\frac{N}{p}-1}}\|\bv\|_{\dot{B}_{p,1}^{\frac{N}{p}+1}}
\lesssim\mathcal{X}_p(t)\Big(\|\bv\|_{\dot{B}_{p,1}^{\frac{N}{p}+1}}^h
+\|\bv\|_{\dot{B}_{2,1}^{\frac{N}{2}+1}}^\ell\Big).
\end{split}
\ee
\be\label{5.12}
\begin{split}
&\quad\|K_2(a){\rm div}\bv\|_{\dot{B}_{p,1}^{\frac{N}{p}-2}}^h\lesssim\|K_2(a){\rm div}\bv\|_{\dot{B}_{p,1}^{\frac{N}{p}}}^h\lesssim\|K_2(a)\|_{\dot{B}_{p,1}^{\frac{N}{p}}}
\|{\rm div}\bv\|_{\dot{B}_{p,1}^{\frac{N}{p}}}\\[1ex]
&\lesssim\|a\|_{\dot{B}_{p,1}^{\frac{N}{p}}}
\|\bv\|_{\dot{B}_{p,1}^{\frac{N}{p}+1}}\lesssim\Big(\|a\|_{\dot{B}_{p,1}^{\frac{N}{p}}}^h+\|a\|_{\dot{B}_{2,1}^{\frac{N}{2}-1}}^\ell\Big)
\Big(\|\bv\|_{\dot{B}_{p,1}^{\frac{N}{p}+1}}^h
+\|\bv\|_{\dot{B}_{2,1}^{\frac{N}{2}+1}}^\ell\Big)\\[1ex]
&\lesssim\mathcal{X}_p(t)
\Big(\|\bv\|_{\dot{B}_{p,1}^{\frac{N}{p}+1}}^h
+\|\bv\|_{\dot{B}_{2,1}^{\frac{N}{2}+1}}^\ell\Big).
\end{split}
\ee
Split $K_4(a)$ as $K_4(0)+\bar{K}_4(a)$ ( here $\bar{K}_4(a)$ is a smooth function satisfying $\bar{K}_4(0)=0$). Then in the following we estimate the terms $\theta{\rm div}\bv$ and $\bar{K}_4(a)\theta{\rm div}\bv$, respectively.
\be\label{5.12}
\begin{split}
\|\theta{\rm div}\bv\|_{\dot{B}_{p,1}^{\frac{N}{p}-2}}^h&\lesssim\|\theta^\ell{\rm div}\bv\|_{\dot{B}_{p,1}^{\frac{N}{p}-1}}^h+\|\theta^h{\rm div}\bv\|_{\dot{B}_{p,1}^{\frac{N}{p}-2}}^h\\[1ex]
&\lesssim\|\theta\|_{\dot{B}_{p,1}^{\frac{N}{p}-1}}^\ell\|{\rm div}\bv\|_{\dot{B}_{p,1}^{\frac{N}{p}}}
+\|\theta\|_{\dot{B}_{p,1}^{\frac{N}{p}}}^h\|{\rm div}\bv\|_{\dot{B}_{p,1}^{\frac{N}{p}-2}}\\[1ex]
&\lesssim\mathcal{X}_p(t)
\Big(\|\theta\|_{\dot{B}_{p,1}^{\frac{N}{p}}}^h
+\|\bv\|_{\dot{B}_{2,1}^{\frac{N}{2}+1}}^\ell+\|\bv\|_{\dot{B}_{p,1}^{\frac{N}{p}+1}}^h\Big)
\end{split}
\ee
and
\be\label{5.13}
\begin{split}
\|\bar{K}_4(a)\theta{\rm div}\bv\|_{\dot{B}_{p,1}^{\frac{N}{p}-2}}^h&\lesssim\|\bar{K}_4(a)\theta\|_{\dot{B}_{p,1}^{\frac{N}{p}}}\|{\rm div}\bv\|_{\dot{B}_{p,1}^{\frac{N}{p}-2}}\\[1ex]
&\lesssim\|\theta\|_{\dot{B}_{p,1}^{\frac{N}{p}}}\|a\|_{\dot{B}_{p,1}^{\frac{N}{p}}}\|\bv\|_{\dot{B}_{p,1}^{\frac{N}{p}-1}}
\\[1ex]
&\lesssim\mathcal{X}_p(t)
\Big(\|a\|_{\dot{B}_{2,1}^{\frac{N}{2}+1}}^\ell+\|a\|_{\dot{B}_{p,1}^{\frac{N}{p}}}^h
+\|\theta\|_{\dot{B}_{2,1}^{\frac{N}{2}+1}}^\ell+\|\theta\|_{\dot{B}_{p,1}^{\frac{N}{p}}}^h\Big).
\end{split}
\ee
Up to now, we finish the estimates of $m$ and conclude that
\be\label{5.1400}
\|m\|_{\dot{B}_{p,1}^{\frac{N}{p}-2}}^h\lesssim\mathcal{X}_p(t)\Big(\|a\|_{\dot{B}_{p,1}^{\frac{N}{p}}}^h
+\|(a,\bv)\|_{\dot{B}_{2,1}^{\frac{N}{2}+1}}^\ell+\|\bv\|_{\dot{B}_{p,1}^{\frac{N}{p}+1}}^h
+\|\theta\|_{\dot{B}_{2,1}^{\frac{N}{2}+1}}^\ell+\|\theta\|_{\dot{B}_{p,1}^{\frac{N}{p}}}^h\Big).
\ee

Now, we are in a position to bound the low frequency term $\|(f,{\bf g}, {\bf m})\|_{\dot{B}_{2,1}^{\frac{N}{2}-1}}^\ell$ in the right hand of \eqref{5.1}, which
has a little bit more difficult. Let us first introduce the following three inequalities and their proofs are postponed.
\be\label{5.14}
\|T_fg\|_{\dot{B}_{2,1}^{s_1}}\lesssim\|f\|_{\dot{B}_{p,1}^{s_2}}\|{ g}\|_{\dot{B}_{p,1}^{s_1-s_2+\frac{2N}{p}-\frac{N}{2}}}
\ee
if $s_1\in\mathbb{R}, s_2\leq \frac{2N}{p}-\frac{N}{2}$ and $2\leq p\leq 4$.
\be\label{5.15}
\|R(f,g)\|_{\dot{B}_{2,1}^{\frac{N}{2}-1}}\lesssim\|f\|_{\dot{B}_{p,1}^{s}}\|{ g}\|_{\dot{B}_{p,1}^{-s+\frac{2N}{p}-1}}
\ee
if $N\geq 2$  and  $2\leq p\leq 4$. And
\be\label{5.18}
\|R(f,g)\|_{\dot{B}_{2,1}^{\frac{N}{2}-2}}\lesssim\|f\|_{\dot{B}_{p,1}^{s}}\|{ g}\|_{\dot{B}_{p,1}^{-s+\frac{2N}{p}-2}}
\ee
if $N\geq 2$  and  $2\leq p\leq \min\{N,4\}$.

 We claim that
\be\label{5.1300}
\|(f,{\bf g}, { m})\|_{\dot{B}_{2,1}^{\frac{N}{2}-1}}^\ell\lesssim\mathcal{X}_p(t)\Big(\|(a,\bv,\theta)\|_{\dot{B}_{2,1}^{\frac{N}{2}+1}}^\ell
+\|(a, \theta)\|_{\dot{B}_{p,1}^{\frac{N}{p}}}^h+\|\bv\|_{\dot{B}_{p,1}^{\frac{N}{p}+1}}^h\Big).
\ee
Here, we only handle the terms with $\theta$, i.e.,
$$
K_2(a)\nabla\theta,\,\,\theta\nabla K_3(a),\,\,\bv\cdot\nabla\theta,\,\,I(a)\Delta\theta,\,\,K_4(a)\theta{\rm div}\bv
$$
and the term $Q(\nabla\bv, \nabla\bv)/(1+a)$, and the remainder terms can refer to \cite{bie2019optimal, xin2018optimal}. Exploiting Bony's decomposition, Bernstein's inequality, Propositions \ref{pr2.2} and \ref{pra.4}, Corollary \ref{co2.1} and the condition \eqref{1.5} concerning $p$, we could estimate the terms above one by one.
\be\label{5.19}
\begin{split}
&\quad\|K_2(a)\nabla\theta\|_{\dot{B}_{2,1}^{\frac{N}{2}-1}}^\ell\\[1ex]
&\leq \|T_{\nabla\theta}K_2(a)+R(\nabla\theta, K_2(a))\|
_{\dot{B}_{2,1}^{\frac{N}{2}-1}}^\ell+\|T_{K_2(a)}\nabla\theta^\ell\|_{\dot{B}_{2,1}^{\frac{N}{2}-1}}^\ell
+\|T_{K_2(a)}\nabla\theta^h\|_{\dot{B}_{2,1}^{\frac{N}{2}-2}}^\ell\\[1ex]
&\lesssim\|\nabla\theta\|_{\dot{B}_{p, 1}^{\frac{N}{p}-1}}\|K_2(a)\|_{\dot{B}_{p, 1}^{\frac{N}{p}}}
+\|K_2(a)\|_{\dot{B}_{p, 1}^{\frac{N}{p}-1}}\|\nabla\theta\|_{\dot{B}_{p, 1}^{\frac{N}{p}}}^\ell
+\|K_2(a)\|_{\dot{B}_{p, 1}^{\frac{N}{p}-1}}\|\nabla\theta\|_{\dot{B}_{p, 1}^{\frac{N}{p}-1}}^h\\[1ex]
&\lesssim\|\theta\|_{\dot{B}_{p, 1}^{\frac{N}{p}}}\|a\|_{\dot{B}_{p, 1}^{\frac{N}{p}}}
+\|a\|_{\dot{B}_{p, 1}^{\frac{N}{p}-1}}\|\theta\|_{\dot{B}_{p, 1}^{\frac{N}{p}+1}}^\ell
+\|a\|_{\dot{B}_{p, 1}^{\frac{N}{p}-1}}\|\theta\|_{\dot{B}_{p, 1}^{\frac{N}{p}}}^h\\[1ex]
&\lesssim\mathcal{X}_p(t)\Big(\|(a,\theta)\|_{\dot{B}_{2,1}^{\frac{N}{2}+1}}^\ell
+\|(a, \theta)\|_{\dot{B}_{p,1}^{\frac{N}{p}}}^h\Big).
\end{split}
\ee

For the term $\theta\nabla K_3(a)$, we decompose it as
$$\theta\nabla K_3(a)=T_{\nabla K_3(a)}\theta+R(\nabla K_3(a), \theta)+T_\theta(\nabla K_3(a))^h+T_\theta(\nabla K_3(a))^\ell.$$
\be\label{5.20}
\begin{split}
&\quad\|T_{\nabla K_3(a)}\theta+R(\nabla K_3(a), \theta)\|_{\dot{B}_{2,1}^{\frac{N}{2}-1}}^\ell\lesssim\|\nabla K_3(a)\|_{\dot{B}_{p,1}^{\frac{N}{p}-1}}\|\theta\|_{\dot{B}_{p,1}^{\frac{N}{p}}}\\[1ex]
&\lesssim\|a\|_{\dot{B}_{p,1}^{\frac{N}{p}}}\|\theta\|_{\dot{B}_{p,1}^{\frac{N}{p}}}
\lesssim\mathcal{X}_p(t)\Big(\|(a,\theta)\|_{\dot{B}_{2,1}^{\frac{N}{2}+1}}^\ell
+\|(a, \theta)\|_{\dot{B}_{p,1}^{\frac{N}{p}}}^h\Big).
\end{split}
\ee
\be\label{5.21}
\begin{split}
&\quad\|T_{\theta}(\nabla K_3(a))^h\|_{\dot{B}_{2,1}^{\frac{N}{2}-1}}^\ell\lesssim
\|T_{\theta}(\nabla K_3(a))^h\|_{\dot{B}_{2,1}^{\frac{N}{2}-2}}^\ell\\[1ex]
&\lesssim\|\theta\|_{\dot{B}_{p,1}^{\frac{N}{p}-1}}\|\nabla K_3(a)\|_{\dot{B}_{p,1}^{\frac{N}{p}-1}}^h
\lesssim\Big(\|\theta\|_{\dot{B}_{2,1}^{\frac{N}{2}-1}}^\ell+\|\theta\|_{\dot{B}_{p,1}^{\frac{N}{p}}}^h\Big)
\|a\|_{\dot{B}_{p,1}^{\frac{N}{p}}}^h\\[1ex]
&\lesssim\mathcal{X}_p(t)\|(a, \theta)\|_{\dot{B}_{p,1}^{\frac{N}{p}}}^h.
\end{split}
\ee
\be\label{5.22}
\begin{split}
&\quad\|T_{\theta}(\nabla K_3(a))^\ell\|_{\dot{B}_{2,1}^{\frac{N}{2}-1}}^\ell\lesssim
\|T_{\theta^\ell}(\nabla K_3(a))^\ell\|_{\dot{B}_{2,1}^{\frac{N}{2}-1}}^\ell+
\|T_{\theta^h}(\nabla K_3(a))^\ell\|_{\dot{B}_{2,1}^{\frac{N}{2}-1}}^\ell\\[1ex]
&\lesssim\|\theta\|_{\dot{B}_{p,1}^{\frac{N}{p}-1}}^\ell\|\nabla K_3(a)\|_{\dot{B}_{p,1}^{\frac{N}{p}}}^\ell+
\|\theta\|_{L^\infty}^h\|\nabla K_3(a)\|_{\dot{B}_{2,1}^{\frac{N}{2}-1}}^\ell\\[1ex]
&\lesssim\|\theta\|_{\dot{B}_{2,1}^{\frac{N}{2}-1}}^\ell\|a\|_{\dot{B}_{2,1}^{\frac{N}{2}+1}}^\ell
+\|\theta\|_{\dot{B}_{p,1}^{\frac{N}{p}}}^h\|a\|_{\dot{B}_{2,1}^{\frac{N}{2}-1}}^\ell
\\[1ex]
&\lesssim\mathcal{X}_p(t)\Big(\|\theta\|_{\dot{B}_{p,1}^{\frac{N}{p}}}^h+\|a\|_{\dot{B}_{2,1}^{\frac{N}{2}+1}}^\ell\Big).
\end{split}
\ee
Combing \eqref{5.20}, \eqref{5.21} and \eqref{5.22}, we infer that
\be\label{5.23}
\|\theta\nabla K_3(a)\|_{\dot{B}_{2,1}^{\frac{N}{2}-1}}^\ell
\lesssim\mathcal{X}_p(t)\Big(\|(a,\theta)\|_{\dot{B}_{p,1}^{\frac{N}{p}}}^h+\|(a, \theta)\|_{\dot{B}_{2,1}^{\frac{N}{2}+1}}^\ell\Big).
\ee
Regarding the term $\bv\cdot\nabla\theta$, we rewrite it as $\bv\cdot\nabla\theta^\ell+\bv\cdot\nabla\theta^h$ and deduce that
\be\label{5.24}
\begin{split}
\|\bv\cdot\nabla\theta^\ell\|_{\dot{B}_{2,1}^{\frac{N}{2}-1}}^\ell
&\leq\|T_{\bv}\nabla\theta^\ell+R(\bv, \nabla\theta^\ell)\|_{\dot{B}_{2,1}^{\frac{N}{2}-1}}^\ell
+\|T_{\nabla\theta^\ell}\bv\|_{\dot{B}_{2,1}^{\frac{N}{2}-1}}^\ell\\[1ex]
&\lesssim\|\bv\|_{\dot{B}_{p,1}^{\frac{N}{p}-1}}\|\nabla\theta\|_{\dot{B}_{p,1}^{\frac{N}{p}}}^\ell
+\|\nabla\theta\|_{\dot{B}_{p,1}^{\frac{N}{p}-2}}^\ell\|\bv\|_{\dot{B}_{p,1}^{\frac{N}{p}+1}}\\[1ex]
&\lesssim\|\bv\|_{\dot{B}_{p,1}^{\frac{N}{p}-1}}\|\theta\|_{\dot{B}_{2,1}^{\frac{N}{2}+1}}^\ell
+\|\theta\|_{\dot{B}_{2,1}^{\frac{N}{2}-1}}^\ell\|\bv\|_{\dot{B}_{p,1}^{\frac{N}{p}+1}}\\[1ex]
&\lesssim\mathcal{X}_p(t)\Big(\|\bv\|_{\dot{B}_{p,1}^{\frac{N}{p}+1}}^h+\|\bv\|_{\dot{B}_{2,1}^{\frac{N}{2}+1}}^\ell+\| \theta\|_{\dot{B}_{2,1}^{\frac{N}{2}+1}}^\ell\Big)
\end{split}
\ee
and
\be\label{5.25}
\begin{split}
\|\bv\cdot\nabla\theta^h\|_{\dot{B}_{2,1}^{\frac{N}{2}-1}}^\ell
&\leq\|T_{\bv}\nabla\theta^h+R(\bv, \nabla\theta^h)\|_{\dot{B}_{2,1}^{\frac{N}{2}-2}}^\ell
+\|T_{\nabla\theta^h}\bv\|_{\dot{B}_{2,1}^{\frac{N}{2}-1}}^\ell\\[1ex]
&\lesssim\|\bv\|_{\dot{B}_{p,1}^{\frac{N}{p}-1}}\|\nabla\theta\|_{\dot{B}_{p,1}^{\frac{N}{p}-1}}^h
+\|\nabla\theta\|_{\dot{B}_{p,1}^{\frac{N}{p}-1}}^h\|\bv\|_{\dot{B}_{p,1}^{\frac{N}{p}}}^\ell\\[1ex]
&\lesssim\|\bv\|_{\dot{B}_{p,1}^{\frac{N}{p}-1}}\|\theta\|_{\dot{B}_{p,1}^{\frac{N}{p}}}^h
+\|\theta\|_{\dot{B}_{p,1}^{\frac{N}{p}}}^h\|\bv\|_{\dot{B}_{2,1}^{\frac{N}{2}-1}}^\ell\\[1ex]
&\lesssim\mathcal{X}_p(t)\| \theta\|_{\dot{B}_{p,1}^{\frac{N}{p}}}^h.
\end{split}
\ee
As a consequence, we have from from \eqref{5.24} and \eqref{5.25} that
\be\label{5.26}
\|\bv\cdot\nabla\theta\|_{\dot{B}_{2,1}^{\frac{N}{2}-1}}^\ell
\lesssim\mathcal{X}_p(t)\Big(\|\theta\|_{\dot{B}_{p,1}^{\frac{N}{p}}}^h+\| \theta\|_{\dot{B}_{2,1}^{\frac{N}{2}+1}}^\ell
+\|\bv\|_{\dot{B}_{p,1}^{\frac{N}{p}+1}}^h+\|\bv\|_{\dot{B}_{2,1}^{\frac{N}{2}+1}}^\ell\Big).
\ee

For $I(a)\Delta\theta=I(a)\Delta\theta^h+I(a)\Delta\theta^\ell$, we have
\be\label{5.27}
\begin{split}
\|I(a)\Delta\theta^h\|_{\dot{B}_{2,1}^{\frac{N}{2}-1}}^\ell
&\leq\|T_{I(a)}\Delta\theta^h\|_{\dot{B}_{2,1}^{\frac{N}{2}-3}}^\ell+\|R(I(a), \Delta\theta^h)+T_{\Delta\theta^h}I(a)\|_{\dot{B}_{2,1}^{\frac{N}{2}-2}}^\ell
\\[1ex]
&\lesssim\|a\|_{\dot{B}_{p,1}^{\frac{N}{p}-1}}\|\Delta\theta\|_{\dot{B}_{p,1}^{\frac{N}{p}-2}}^h+
\|a\|_{\dot{B}_{p,1}^{\frac{N}{p}}}\|\Delta\theta\|_{\dot{B}_{p,1}^{\frac{N}{p}-2}}^h\\[1ex]
&\lesssim\|a\|_{\dot{B}_{p,1}^{\frac{N}{p}-1}}\|\theta\|_{\dot{B}_{p,1}^{\frac{N}{p}}}^h
+\|\theta\|_{\dot{B}_{p,1}^{\frac{N}{p}}}^h\|a\|_{\dot{B}_{p,1}^{\frac{N}{p}}}\\[1ex]
&\lesssim\mathcal{X}_p(t)\| \theta\|_{\dot{B}_{p,1}^{\frac{N}{p}}}^h
\end{split}
\ee
and
\be\label{5.27}
\begin{split}
\|I(a)\Delta\theta^\ell\|_{\dot{B}_{2,1}^{\frac{N}{2}-1}}^\ell
&\leq\|T_{I(a)}\Delta\theta^\ell\|_{\dot{B}_{2,1}^{\frac{N}{2}-2}}^\ell+\|R(I(a), \Delta\theta^\ell)+T_{\Delta\theta^\ell}I(a)\|_{\dot{B}_{2,1}^{\frac{N}{2}-1}}^\ell\\[1ex]
&\lesssim\|a\|_{\dot{B}_{p,1}^{\frac{N}{p}-1}}\|\Delta\theta\|_{\dot{B}_{p,1}^{\frac{N}{p}-1}}^\ell+
\|a\|_{\dot{B}_{p,1}^{\frac{N}{p}}}\|\Delta\theta\|_{\dot{B}_{p,1}^{\frac{N}{p}-1}}^\ell\\[1ex]
&\lesssim\|a\|_{\dot{B}_{p,1}^{\frac{N}{p}-1}}\|\theta\|_{\dot{B}_{2,1}^{\frac{N}{2}+1}}^\ell
+\|\theta\|_{\dot{B}_{2,1}^{\frac{N}{2}+1}}^\ell\|a\|_{\dot{B}_{p,1}^{\frac{N}{p}}}\\[1ex]
&\lesssim\mathcal{X}_p(t)\| \theta\|_{\dot{B}_{2,1}^{\frac{N}{2}+1}}^\ell
\end{split}
\ee
which together with \eqref{5.26} yields that
\be\label{5.28}
\|I(a)\Delta\theta\|_{\dot{B}_{2,1}^{\frac{N}{2}-1}}^\ell
\lesssim\mathcal{X}_p(t)\Big(\|\theta\|_{\dot{B}_{p,1}^{\frac{N}{p}}}^h+\| \theta\|_{\dot{B}_{2,1}^{\frac{N}{2}+1}}^\ell\Big).
\ee

For the term $K_4(a)\theta{\rm div}\bv$, as before we also decompose $K_4(a)=K_4(0)+\bar{K}_4(a)$ since $K_4(0)\neq 0$, where $\bar{K}_4(a)$ is a smooth function satisfying $\bar{K}_4(0)=0$.
\be\label{5.29}
\begin{split}
  \|\theta{\rm div}\bv\|_{\dot{B}_{2,1}^{\frac{N}{2}-1}}^\ell&\lesssim\|T_\theta{\rm div}\bv+R(\theta, {\rm div}\bv)\|
  _{\dot{B}_{2,1}^{\frac{N}{2}-1}}^\ell+\|T_{{\rm div}\bv}\theta^h\|_{\dot{B}_{2,1}^{\frac{N}{2}-2}}^\ell+
  \|T_{{\rm div}\bv}\theta^\ell\|_{\dot{B}_{2,1}^{\frac{N}{2}-1}}^\ell\\[1ex]
  &\lesssim\|\theta\|_{\dot{B}_{p,1}^{\frac{N}{p}-1}}\|{\rm div}\bv\|_{\dot{B}_{p,1}^{\frac{N}{p}}}
  +\|{\rm div}\bv\|_{\dot{B}_{p,1}^{\frac{N}{p}-2}}\|\theta\|_{\dot{B}_{p,1}^{\frac{N}{p}}}^h
  +\|{\rm div}\bv\|_{L^\infty}\|\theta\|_{\dot{B}_{2,1}^{\frac{N}{2}-1}}^\ell\\[1ex]
  &\lesssim\|\theta\|_{\dot{B}_{p,1}^{\frac{N}{p}-1}}\|\bv\|_{\dot{B}_{p,1}^{\frac{N}{p}+1}}
  +\|\bv\|_{\dot{B}_{p,1}^{\frac{N}{p}-1}}\|\theta\|_{\dot{B}_{p,1}^{\frac{N}{p}}}^h
  +\|\bv\|_{\dot{B}_{p,1}^{\frac{N}{p}+1}}\|\theta\|_{\dot{B}_{2,1}^{\frac{N}{2}-1}}^\ell\\[1ex]
  &\lesssim\mathcal{X}_p(t)\Big(\|\bv\|_{\dot{B}_{p,1}^{\frac{N}{p}+1}}
  +\|\theta\|_{\dot{B}_{p,1}^{\frac{N}{p}}}^h\Big),
\end{split}
\ee
and
\be\label{5.30}
\begin{split}
  &\quad\|\bar{K}_4(a)\theta{\rm div}\bv\|_{\dot{B}_{2,1}^{\frac{N}{2}-1}}^\ell\lesssim\|\bar{K}_4(a)(\theta{\rm div}\bv)^\ell\|_{\dot{B}_{2,1}^{\frac{N}{2}-1}}^\ell+\|\bar{K}_4(a)(\theta{\rm div}\bv)^h\|_{\dot{B}_{2,1}^{\frac{N}{2}-1}}^\ell\\[1ex]
  &\lesssim\|\bar{K}_4(a)\|_{L^\infty}\|\theta{\rm div}\bv\|_{\dot{B}_{2,1}^{\frac{N}{2}-1}}^\ell
  +\|T_{\bar{K}_4(a)}(\theta{\rm div}\bv)^h\|_{\dot{B}_{2,1}^{\frac{N}{2}-3}}^\ell\\[1ex]
  &\quad\quad\quad\quad\quad\quad\quad\quad\quad\quad\quad\quad
  +\|T_{(\theta{\rm div}\bv)^h}\bar{K}_4(a)+R((\theta{\rm div}\bv)^h, \bar{K}_4(a))\|_{\dot{B}_{2,1}^{\frac{N}{2}-2}}^\ell\\[1ex]
  &\lesssim\|a\|_{\dot{B}_{p,1}^{\frac{N}{p}}}\|\theta{\rm div}\bv\|_{\dot{B}_{2,1}^{\frac{N}{2}-1}}^\ell
  +\|a\|_{\dot{B}_{p,1}^{\frac{N}{p}-1}}\|\theta{\rm div}\bv\|_{\dot{B}_{p,1}^{\frac{N}{p}-2}}^h+
  \|a\|_{\dot{B}_{p,1}^{\frac{N}{p}}}\|\theta{\rm div}\bv\|_{\dot{B}_{p,1}^{\frac{N}{p}-2}}^h\\[1ex]
&\lesssim\mathcal{X}_p(t)
\Big(\|\theta\|_{\dot{B}_{p,1}^{\frac{N}{p}}}^h
+\|\bv\|_{\dot{B}_{2,1}^{\frac{N}{2}+1}}^\ell+\|\bv\|_{\dot{B}_{p,1}^{\frac{N}{p}+1}}^h\Big),
  \end{split}
\ee
where in the last inequality we have used the results of \eqref{5.12} and \eqref{5.29}. Consequently, thanks to \eqref{5.29}
and \eqref{5.30}, we get
\be\label{5.31}
\|{K}_4(a)\theta{\rm div}\bv\|_{\dot{B}_{2,1}^{\frac{N}{2}-1}}^\ell
\lesssim\mathcal{X}_p(t)
\Big(\|\theta\|_{\dot{B}_{p,1}^{\frac{N}{p}}}^h
+\|\bv\|_{\dot{B}_{2,1}^{\frac{N}{2}+1}}^\ell+\|\bv\|_{\dot{B}_{p,1}^{\frac{N}{p}+1}}^h\Big).
\ee

Finally, we deal with the term $Q(\nabla\bv, \nabla\bv)/(1+a)=(1-I(a))Q(\nabla\bv, \nabla\bv)$.
\be\label{5.32}
\begin{split}
\|Q(\nabla\bv,\nabla\bv)\|_{\dot{B}_{2,1}^{\frac{N}{2}-1}}^\ell&\lesssim
\|Q(\nabla\bv,\nabla\bv)\|_{\dot{B}_{2,1}^{\frac{N}{2}-2}}^\ell\lesssim\|\nabla\bv\|_{\dot{B}_{p,1}^{\frac{N}{p}-1}}
\|\nabla\bv\|_{\dot{B}_{p,1}^{\frac{N}{p}-1}}\\[1ex]
&\lesssim\|\bv\|_{\dot{B}_{p,1}^{\frac{N}{p}}}^2
\lesssim\|\bv\|_{\dot{B}_{p,1}^{\frac{N}{p}-1}}\|\bv\|_{\dot{B}_{p,1}^{\frac{N}{p}+1}}
\lesssim\mathcal{X}_p(t)\|\bv\|_{\dot{B}_{p,1}^{\frac{N}{p}+1}}
\end{split}
\ee
and
\be\label{5.33}
\begin{split}
&\quad\|I(a)Q(\nabla\bv,\nabla\bv)\|_{\dot{B}_{2,1}^{\frac{N}{2}-1}}^\ell\\[1ex]
&\lesssim\|T_{I(a)}Q(\nabla\bv,\nabla\bv)+R(I(a), Q(\nabla\bv,\nabla\bv))\|_{\dot{B}_{2,1}^{\frac{N}{2}-2}}^\ell
+\|T_{Q(\nabla\bv,\nabla\bv)}I(a)\|_{\dot{B}_{2,1}^{\frac{N}{2}-1}}^\ell\\[1ex]
&\lesssim\|I(a)\|_{\dot{B}_{p,1}^{\frac{N}{p}-1}}\||\nabla\bv|^2\|_{\dot{B}_{p,1}^{\frac{N}{p}-1}}
+\||\nabla\bv|^2\|_{\dot{B}_{p,1}^{\frac{N}{p}-1}}\|I(a)\|_{\dot{B}_{p,1}^{\frac{N}{p}}}\\[1ex]
&\lesssim\|a\|_{\dot{B}_{p,1}^{\frac{N}{p}-1}}\||\bv\|_{\dot{B}_{p,1}^{\frac{N}{p}}}^2
+\|\bv\|_{\dot{B}_{p,1}^{\frac{N}{p}}}^2\|a\|_{\dot{B}_{p,1}^{\frac{N}{p}}}
\lesssim\mathcal{X}_p(t)\|\bv\|_{\dot{B}_{p,1}^{\frac{N}{p}+1}}
\end{split}
\ee
which together with \eqref{5.32} gives
\be\label{5.34}
\|Q(\nabla\bv,\nabla\bv)/(1+a)\|_{\dot{B}_{2,1}^{\frac{N}{2}-1}}^\ell
\lesssim\mathcal{X}_p(t)
\|\bv\|_{\dot{B}_{p,1}^{\frac{N}{p}+1}}.
\ee
Up to now, we deduce that \eqref{5.1300} holds true.

Plugging \eqref{5.2}, \eqref{5.3}, \eqref{5.8}, \eqref{5.1400} and \eqref{5.1300}
into \eqref{5.1} and applying the fact that $\mathcal{X}_p(t)\lesssim\mathcal{X}_{p,0}\ll 1$ for all $t\geq 0$, we end up with
\be\label{5.39}
\begin{split}
  &\quad\frac{d}{dt}\Big(\|(a, \bv, \theta)\|_{\dot{B}_{2,1}^{\frac{N}{2}-1}}^\ell+\|(\nabla a, \bv)\|_{\dot{B}_{p,1}^{\frac{N}{p}-1}}^h+\|\theta\|_{\dot{B}_{p,1}^{\frac{N}{p}-2}}^h\Big)\\[1ex]
  &\quad\quad\quad\quad\quad\quad\quad\quad\quad\quad\quad\quad\quad\quad
  +\Big(\|(a, \bv, \theta)\|_{\dot{B}_{2,1}^{\frac{N}{2}+1}}^\ell
  +\|(a, \theta)\|_{\dot{B}_{p,1}^{\frac{N}{p}}}^h+\|\bv\|_{\dot{B}_{p,1}^{\frac{N}{p}+1}}^h\Big)\leq 0.
\end{split}
\ee

In what follows, we will employ  the following interpolation inequalities:
\begin{prop}\label{pra.7}(\cite{xin2018optimal})
Suppose that $m\neq \rho$. Then it holds that
$$
\|f\|_{\dot{B}_{p,1}^j}^\ell\lesssim (\|f\|_{\dot{B}_{r,\infty}^m}^\ell)^{1-\eta}(\|f\|_{\dot{B}_{r,\infty}^{\rho}}^\ell)^{\eta},\,\,\,\,\,\,
\|f\|_{\dot{B}_{p,1}^j}^h\lesssim (\|f\|_{\dot{B}_{r,\infty}^m}^h)^{1-\eta}(\|f\|_{\dot{B}_{r,\infty}^{\rho}}^h)^{\eta}
$$
where $j+N(\frac{1}{r}-\frac{1}{p})=m(1-\eta)+\rho\eta$ for $0<\eta<1$ and $1\leq r\leq p\leq \infty$.
\end{prop}

Due to $-\sigma_1<\frac{N}{2}-2<\frac{N}{2}-1\leq \frac{N}{p}<\frac{N}{2}+1$, it follows from Proposition \ref{pra.7} that
\be\label{5.40}
\|(a,\bv,\theta)\|_{\dot{B}_{2,1}^{\frac{N}{2}-1}}^\ell
\leq C\Big(\|(a,\bv,\theta)\|_{\dot{B}_{2,\infty}^{-\sigma_1}}^\ell\Big)^{\eta_0}
\Big(\|(a,\bv,\theta)\|_{\dot{B}_{2,\infty}^{\frac{N}{2}+1}}^\ell\Big)^{1-\eta_0},
\ee
where $\eta_0=\frac{2}{N/2+1+\sigma_1}\in (0,1)$. In view of \eqref{4.74}, we have
$$
\|(a,\bv,\theta)\|_{\dot{B}_{2,\infty}^{\frac{N}{2}+1}}^\ell\geq c_0\Big(\|(a,\bv,\theta)\|_{\dot{B}_{2,1}^{\frac{N}{2}-1}}^\ell\Big)^{\frac{1}{1-\eta_0}},
$$
where $c_0=C^{-\frac{1}{1-\eta_0}}C_0^{-\frac{\eta_0}{1-\eta_0}}$.

Moreover, it follows from the fact
$\|a\|_{\dot{B}_{p,1}^{\frac{N}{p}}}^h+\|\bv\|_{\dot{B}_{p,1}^{\frac{N}{p}-1}}^h+\|\theta\|_{\dot{B}_{p,1}^{\frac{N}{p}-2}}^h\leq \mathcal{X}_p(t)\lesssim \mathcal{X}_{p,0}\ll 1$
for all $t\geq 0$ that
$$
\Big(\|a\|_{\dot{B}_{p,1}^{\frac{N}{p}}}^h\Big)^{\frac{1}{1-\eta_0}}\lesssim\|a\|_{\dot{B}_{p,1}^{\frac{N}{p}}}^h,\,\,\,\,\,
\Big(\|\bv\|_{\dot{B}_{p,1}^{\frac{N}{p}-1}}^h\Big)^{\frac{1}{1-\eta_0}}\lesssim
\|\bv\|_{\dot{B}_{p,1}^{\frac{N}{p}+1}}^h\,\,\,\,{\rm and}\,\,\,\Big(\|\theta\|_{\dot{B}_{p,1}^{\frac{N}{p}-2}}^h\Big)^{\frac{1}{1-\eta_0}}\lesssim
\|\theta\|_{\dot{B}_{p,1}^{\frac{N}{p}}}^h.
$$
Thus, there exists a constant $\tilde{c}_0>0$ such that the following Lyapunov-type inequality holds:
{\small\begin{equation}\label{5.41}
\begin{split}
&\quad\frac{d}{dt}\Big(\|(a,\bv,\theta)(t)\|_{\dot{B}_{2,1}^{\frac{N}{2}-1}}^\ell+\|a(t)\|_{\dot{B}_{p,1}^{\frac{N}{p}}}^h
+\|\bv(t)\|_{\dot{B}_{p,1}^{\frac{N}{p}-1}}^h+\|\theta(t)\|_{\dot{B}_{p,1}^{\frac{N}{p}-2}}^h\Big)\\[1ex]
&+\Big(\|(a,\bv,\theta)(t)\|_{\dot{B}_{2,1}^{\frac{N}{2}-1}}^\ell+\|a(t)\|_{\dot{B}_{p,1}^{\frac{N}{p}}}^h
+\|\bv(t)\|_{\dot{B}_{p,1}^{\frac{N}{p}-1}}^h
+\|\theta(t)\|_{\dot{B}_{p,1}^{\frac{N}{p}-2}}^h\Big)^{1+\frac{2}{N/2-1+\sigma_1}}d\tau\lesssim\mathcal{X}_{p,0}.
\end{split}
\end{equation}}
Solving \eqref{5.41}  yields
\be\label{5.42}
\begin{split}
&\quad\|(a,\bv,\theta)(t)\|_{\dot{B}_{2,1}^{\frac{N}{2}-1}}^\ell+\|a(t)\|_{\dot{B}_{p,1}^{\frac{N}{p}}}^h
+\|\bv(t)\|_{\dot{B}_{p,1}^{\frac{N}{p}-1}}^h+\|\theta(t)\|_{\dot{B}_{p,1}^{\frac{N}{p}-2}}^h\\[1ex]
&\lesssim\Big(\mathcal{X}_{p,0}^{-\frac{2}{N/2-1+\sigma_1}}
+\frac{2t}{N/2-1+\sigma_1}\Big)^{-\frac{N/2-1+\sigma_1}{2}}
\lesssim (1+t)^{-\frac{N/2-1+\sigma_1}{2}}
\end{split}
\ee
for all $t\geq 0$.
Resorting to the embedding properties in Proposition \ref{pr2.1}, we arrive at
\be\label{5.43}
\begin{split}
\|(a, \bv)(t)\|_{\dot{B}_{p,1}^{\frac{N}{p}-1}}\lesssim\|(a, \bv)(t)\|_{\dot{B}_{2,1}^{\frac{N}{2}-1}}^\ell+\|(\nabla a, \bv)(t)\|_{\dot{B}_{p,1}^{\frac{N}{p}-1}}^h
\lesssim (1+t)^{-\frac{N/2-1+\sigma_1}{2}}.
\end{split}
\ee
In addition, employing Proposition
\ref{pra.7} again yields  for $\sigma_2\in (-\sigma_1-N(\frac{1}{2}-\frac{1}{p}), \frac{N}{p}-1)$ that
\be\label{5.44}
\begin{split}
\|(a, \bv, \theta)(t)\|_{\dot{B}_{p,1}^{\sigma_2}}^\ell\lesssim\|(a, \bv, \theta)(t)\|_{\dot{B}_{2,1}^{\sigma_2+N(\frac{1}{2}-\frac{1}{p})}}^\ell
\lesssim\Big(\|(a,\bv,\theta)\|_{\dot{B}_{2,\infty}^{-\sigma_1}}^\ell\Big)^{\eta_1}
\Big(\|(a,\bv,\theta)\|_{\dot{B}_{2,\infty}^{\frac{N}{2}-1}}^\ell\Big)^{1-\eta_1},
\end{split}
\ee
where
$$
\eta_1=\frac{\frac{N}{p}-1-\sigma_2}{\frac{N}{2}-1+\sigma_1}\in(0,1).
$$
Note that
$$
\|(a,\bv,\theta)\|_{\dot{B}_{2,\infty}^{-\sigma_1}}^\ell\leq C_0
$$
for all $t\geq 0$.  From \eqref{5.42} and \eqref{5.44}, we deduce that
\be\label{5.45}
\begin{split}
\|(a, \bv, \theta)(t)\|_{\dot{B}_{p,1}^{\sigma_2}}^\ell\lesssim
\Big[(1+t)^{-\frac{N/2-1+\sigma_1}{2}}\Big]^{1-\eta_1}
=(1+t)^{-\frac{N}{2}(\frac{1}{2}-\frac{1}{p})-\frac{\sigma_2+\sigma_1}{2}}
\end{split}
\ee
for all $t\geq 0$, which leads to
\be\label{5.46}
\begin{split}
\|(a, \bv)(t)\|_{\dot{B}_{p,1}^{\sigma_2}}\lesssim
\|(a, \bv)(t)\|_{\dot{B}_{p,1}^{\sigma_2}}^\ell+\|(a, \bv)(t)\|_{\dot{B}_{p,1}^{\sigma_2}}^h
\lesssim (1+t)^{-\frac{N}{2}(\frac{1}{2}-\frac{1}{p})-\frac{\sigma_2+\sigma_1}{2}},
\end{split}
\ee
provided that $\sigma_2\in (-\sigma_1-N(\frac{1}{2}-\frac{1}{p}), \frac{N}{p}-1)$. This together with \eqref{5.43} yields \eqref{1}.

Similarly, for
$\sigma_3\in (-\sigma_1-N(\frac{1}{2}-\frac{1}{p}), \frac{N}{p}-2]$, we could get
\be\nonumber
\begin{split}
\|\theta(t)\|_{\dot{B}_{p,1}^{\sigma_3}}\lesssim
\|\theta(t)\|_{\dot{B}_{p,1}^{\sigma_3}}^\ell+\|\theta(t)\|_{\dot{B}_{p,1}^{\sigma_3}}^h
\lesssim (1+t)^{-\frac{N}{2}(\frac{1}{2}-\frac{1}{p})-\frac{\sigma_3+\sigma_1}{2}},
\end{split}
\ee
which yields \eqref{2}. So far, the proof of Theorem \ref{th2} is completed.

In the following, we give the proofs of inequalities \eqref{5.14}, \eqref{5.15} and \eqref{5.18}.
\begin{proof}[Proof of \eqref{5.14}]
Set $\frac{1}{p^\ast}+\frac{1}{p}=1$ and $\|c(j)\|_{l^1}=1$. From the definition of $T_fg$, we obtain
\begin{equation}\label{5.16}
\begin{split}
&\quad\|\dot{\Delta}_j(T_fg)\|_{L^2}\\[1ex]
&\leq \sum_{|k-j|\leq 4}\sum_{k^\prime\leq k-2}\|\dot{\Delta}_j(\dot{\Delta}_{k^\prime}f\dot{\Delta}_kg)\|_{L^2}\leq\sum_{|k-j|\leq 4}\sum_{k^\prime\leq k-2}\|\dot{\Delta}_{k^\prime}f\|_{L^{p^\ast}}\|\dot{\Delta}_kg\|_{L^p}\\[1ex]
&\lesssim\sum_{|k-j|\leq 4}\sum_{k^\prime\leq k-2}2^{k^\prime(\frac{2N}{p}-\frac{N}{2})}\|\dot{\Delta}_{k^\prime}f\|_{L^p}\|\dot{\Delta}_kg\|_{L^p}\\[1ex]
&\lesssim\sum_{|k-j|\leq 4}\sum_{k^\prime\leq k-2}2^{k^\prime(\frac{2N}{p}-\frac{N}{2}-s_2)}2^{k^\prime s_2}
\|\dot{\Delta}_{k^\prime}f\|_{L^p}2^{k(-s_1+s_2-\frac{2N}{p}+\frac{N}{2})}2^{k(s_1-s_2+\frac{2N}{p}-\frac{N}{2})}
\|\dot{\Delta}_kg\|_{L^p}\\[1ex]
&\lesssim c(j)2^{j(-s_1)}
\|f\|_{\dot{B}_{p,1}^{s_2}}\|g\|_{\dot{B}_{p,1}^{s_1-s_2+\frac{2N}{p}-\frac{N}{2}}},
\end{split}
\end{equation}
which yields \eqref{5.14}. Where we used that $p^\ast\geq p$ if $2\leq p\leq 4$ in the third inequality, and the condition $\frac{2N}{p}-\frac{N}{2}-s_2\geq 0$  in the last inequality.
\end{proof}

\begin{proof}[Proof of \eqref{5.15} and \eqref{5.18}] We only prove \eqref{5.18} and the proof of \eqref{5.15} is similar. It follows from  the definition of $R(f,g)$ that
\begin{equation}\label{5.17}
\begin{split}
&\quad\|\dot{\Delta}_jR(f,g)\|_{L^2}\leq \sum_{k\geq j-3}\sum_{|k-k^\prime|\leq 1}
\|\dot{\Delta}_j(\dot{\Delta}_{k}f\dot{\Delta}_{k^\prime}g)\|_{L^2}\\[1ex]
&\lesssim 2^{j(\frac{2N}{p}-\frac{N}{2})}\sum_{k\geq j-3}\sum_{|k-k^\prime|\leq 1}2^{k(-s)}2^{ks}\|\dot{\Delta}_{k}f\|_{L^{p}}2^{k^\prime(s-\frac{2N}{p}+2)}
2^{k^\prime(-s+\frac{2N}{p}-2)}\|\dot{\Delta}_{k^\prime}g\|_{L^{p}}\\[1ex]
&\lesssim2^{j(\frac{2N}{p}-\frac{N}{2})}\sum_{k\geq j-3}2^{k(-\frac{2N}{p}+2)}c^2(k)\|f\|_{\dot{B}_{p,1}^{s}}
\|g\|_{\dot{B}_{p,1}^{-s+\frac{2N}{p}-2}}\\[1ex]
&\lesssim c(j)2^{j(2-\frac{N}{2})}\|f\|_{\dot{B}_{p,1}^{s}}
\|g\|_{\dot{B}_{p,1}^{-s+\frac{2N}{p}-1}},
\end{split}
\end{equation}
which yields \eqref{5.18}. Where we used that $1\leq \frac{p}{2}\leq 2$  in the second inequality and $2-\frac{2N}{p}\leq 0$ in the last inequality.
\end{proof}

\subsection{Proof of Corollary \ref{col1}} In fact, Corollary \ref{col1} can be regarded as the direct consequence of the following interpolation inequality:
\begin{prop}\label{pra.8}(\cite{bahouri2011fourier})
The following interpolation inequality holds true:
$$
\|\Lambda^lf\|_{L^r}\lesssim \|\Lambda^mf\|_{L^q}^{1-\eta}\|\Lambda^kf\|_{L^q}^\eta,
$$
whenever $0\leq \eta\leq 1, 1\leq q\leq r\leq \infty$ and
$$
l+N\Big(\frac{1}{q}-\frac{1}{r}\Big)=m(1-\eta)+k\eta.
$$
\end{prop}

 With the aid of  Proposition \ref{pra.8}, we define $\eta_2$ by the relation
$$
m(1-\eta_2)+k\eta_2=l+N\Big(\frac{1}{p}-\frac{1}{r}\Big),
$$
where $m=\frac{N}{p}-1$ and $k=-\sigma_1-N(\frac{1}{2}-\frac{1}{p})+\varepsilon$ with $\varepsilon>0$ small enough. When $l\in\mathbb{R}$ satisfying $-\sigma_1
-\frac{N}{2}+\frac{N}{p}<l+\frac{N}{p}-\frac{N}{r}\leq \frac{N}{p}-1$, it is easy to see that $\eta_2\in [0,1)$. As a consequence, we conclude  by $\dot{B}_{p,1}^0\hookrightarrow L^p$ that
\be\nonumber
\begin{split}
&\quad\|\Lambda^l(a,\bv)\|_{L^r}\lesssim
\|\Lambda^m(a,\bv)\|_{L^p}^{1-\eta_2}\|\Lambda^k(a,\bv)\|_{L^p}^{\eta_2}\\[1ex]
&\lesssim \left[(1+t)^{-\frac{N}{2}(\frac{1}{2}-\frac{1}{p})-\frac{m+\sigma_1}{2}}\right]^{1-\eta_2}
\left[(1+t)^{-\frac{N}{2}(\frac{1}{2}-\frac{1}{p})-\frac{k+\sigma_1}{2}}\right]^{\eta_2}
=(1+t)^{-\frac{N}{2}(\frac{1}{2}-\frac{1}{r})-\frac{l+\sigma_1}{2}}
\end{split}
\ee
for $p\leq r\leq \infty$. Similarly,  we define $\eta_3$ by the relation
$
m(1-\eta_3)+k\eta_3=n+N\Big(\frac{1}{p}-\frac{1}{r}\Big)
$
and obtain that
\be\nonumber
\begin{split}
&\quad\|\Lambda^n\theta\|_{L^r}\lesssim
\|\Lambda^m\theta\|_{L^p}^{1-\eta_3}\|\Lambda^k\theta\|_{L^p}^{\eta_3}\\[1ex]
&\lesssim \left[(1+t)^{-\frac{N}{2}(\frac{1}{2}-\frac{1}{p})-\frac{m+\sigma_1}{2}}\right]^{1-\eta_3}
\left[(1+t)^{-\frac{N}{2}(\frac{1}{2}-\frac{1}{p})-\frac{k+\sigma_1}{2}}\right]^{\eta_3}
=(1+t)^{-\frac{N}{2}(\frac{1}{2}-\frac{1}{r})-\frac{n+\sigma_1}{2}}
\end{split}
\ee
provided that $p\leq r\leq \infty$ and $-\sigma_1
-\frac{N}{2}+\frac{N}{p}<n+\frac{N}{p}-\frac{N}{r}\leq \frac{N}{p}-2$.
Thus, we finish the proof of Corollary \ref{col1}.
\providecommand{\href}[2]{#2}
\providecommand{\arxiv}[1]{\href{http://arxiv.org/abs/#1}{arXiv:#1}}
\providecommand{\url}[1]{\texttt{#1}}
\providecommand{\urlprefix}{URL }


\end{document}